\documentclass[final,1p,times]{elsarticle}
\usepackage{amssymb}
 \usepackage{amsthm}
\usepackage{amsmath,amssymb,amsopn,amsfonts,mathrsfs,amsbsy,amscd}
\usepackage{longtable}
\usepackage{multirow}
\usepackage[latin1]{inputenc}
\usepackage{adjustbox}
\usepackage{caption}
\usepackage{tikz-cd}

\newcommand{\esp}{\quad\mbox{and}\quad}
\newcommand{\R}{\mathbb{R}}

\newcommand{\md}{\mathrm d}
\newcommand{\om}{\omega}

\newcommand{\J}{\mathrm J}
\newcommand{\K}{\mathrm K}
\newcommand{\G}{{\mathfrak{g}}}

\newcommand{\A}{{\mathfrak{a}}}
\newcommand{\h}{{\mathfrak{h}}}

\newcommand{\ol}{\overline}
\newcommand{\Perm}{\mathcal{P}} 
\newcommand{\e}{\check{e}}
\newtheorem{Def}{Definition}
\newtheorem{theo}{Theorem}
\newtheorem{pr}{Proposition}
\newtheorem{Le}{Lemma}
\newtheorem{co}{Corollary}

\newtheorem{remark}{Remark}

\usepackage{adjustbox}

\usepackage{hyperref}
\usepackage[all,cmtip]{xy}
\begin{document}
\begin{frontmatter}
\title{ Eight-dimensional  non completely reducible symplectic Lie algebras }
\author[]{T.  Aït Aissa, S. El Bourkadi,   M. W. Mansouri and SM. Sbai}
\address{Department of Mathematics,\\ Faculty of Sciences, Ibn Tofail University,\\ Analysis, Geometry and Applications Laboratory\\
$($LAGA$)$, Kenitra, Morocco\\e-mail$:$ tarik.aitaissa@uit.ac.ma\\
said.elbourkadi@uit.ac.ma\\
mansourimohammed.wadia@uit.ac.ma\\sbaisimo@netcourrier.com}
\begin{abstract}
	A non completely reducible symplectic Lie algebra is a symplectic Lie algebra which cannot be symplectically reduced to the trivial symplectic Lie algebra. Our aim is to provide a complete classification, up to symplectomorphism of non completely reducible symplectic Lie algebras in dimensions~$n \leq 8$ and, furthermore, to provide a complete description of symplectic Lie algebras admitting one-dimensional isotropic ideals.
\end{abstract}

\begin{keyword}
	Symplectic Lie algebras, Symplectic reduction, Symplectic double extension.
	
	MSC2010: 22E60, 17B05, 22E60.
\end{keyword}
	\end{frontmatter}
\section{Introduction}

Let $\G$ be a finite-dimensional real Lie algebra. A \emph{symplectic Lie algebra} is a pair $(\G, \omega)$, where 
$\omega \in Z^2(\G)$ is a closed and non-degenerate $2$-form. The closedness of $\omega$ is equivalent to
\[
\mathop{\resizebox{1.3\width}{!}{$\sum$}}\limits_{\mathrm{cycl}}\om([x,y],z):=\omega([x,y],z)+\omega([y,z],x)+\omega([z,x],y)=0,
\]
for all $x,y,z\in\G$. Two symplectic Lie algebras $(\G_1,\omega_1)$ and $(\G_2,\omega_2)$ are isomorphic
if there exists a Lie algebra isomorphism $\Psi:\G_1\longrightarrow\G_2$  that satisfies $\Psi^\ast\omega_2=\omega_1$. A symplectic Lie algebra has a one-to-one correspondence with a simply connected Lie group that has a symplectic form which is left-invariant. Symplectic forms on Lie groups and algebras naturally arise in Poisson geometry, Hamiltonian mechanics, and geometric problems such as cotangent bundle structures. A particularly interesting subclass consists of Kähler and pseudo-Riemannian Kähler Lie algebras \cite{CF,DM}, which carry compatible complex or pseudo-Riemannian structures. Symplectic Lie algebras are often called quasi-Frobenius Lie algebras, as they generalize Frobenius Lie algebras - those admitting a non-degenerate exact $2$-form $\omega =\mathbf{d}\alpha$, where $\mathbf{d}$ is the Chevalley-Eilenberg differential.

Several classifications of symplectic Lie algebras have been established in low dimensions \cite{ABM,Com1,O}, for the nilpotent case \cite{Bu1, GJK,KGM}, see also \cite{Fish}), and for the non solvable case (also exact symplectic Lie algebras) \cite{AM2,Com2}.

Let $(\G,\omega)$ be a symplectic Lie algebra. An ideal $\mathfrak{j}$ of $(\G,\omega)$ is called isotropic if $\mathfrak{j}\subset\mathfrak{j}^{\perp_\omega}$ with
\begin{align*}
\mathfrak{j}^{\perp_\omega}=\big\{x\in\G~:~\omega(x,y)=0,\quad\text{for all}\quad y\in\mathfrak{j}\big\}.
\end{align*}
The orthogonal
$\mathfrak{j}^{\perp_\omega}$ is a subalgebra of $\G$ which contains $\mathfrak{j}$, and therefore $\omega$ descends to a symplectic form $\overline{\omega
}$ on the quotient Lie algebra $\overline{\G}=\mathfrak{j}^{\perp_\omega}/\mathfrak{j}$. The pair $(\overline{\G},\overline{\omega})$ is called the \textit{symplectic reduction} of $(\G,\omega)$ with respect to the isotropic ideal $\mathfrak{j}$. If $\mathfrak{j}^{\perp_\omega}$ is an ideal in $\G$ we call the \textit{symplectic reduction normal}, and $\mathfrak{j}$ is called a
normal isotropic ideal. The reduction of $(\G,\omega)$ with respect to $\mathfrak{j}$ is called \textit{Lagrangian reduction} if $\mathfrak{j}$ is a Lagrangian ideal, i.e., $\mathfrak{j}^{\perp_\omega}=\mathfrak{j}$.

A symplectic Lie algebra $(\G,\omega)$ can be symplectically reduced to a symplectic Lie algebra $(\overline{\G},\overline{\omega})$ if there is a sequence of subsequent symplectic algebras reductions 
\begin{align*}
(\G,\omega)\to(\G_1,\omega_1)\to\cdot\cdot\cdot\to(\G_k,\omega_k)=(\overline{\G},\overline{\omega}).
\end{align*}
The integer $k$ is called the length of the reduction sequence and the symplectic Lie
algebra $(\overline{\G},\overline{\omega})$ is called its \text{base}. The reduction sequence is called \text{complete} if it has an irreducible base, that is, if $(\overline{\G},\overline{\omega})$ admits no nontrivial isotropic ideal.  If $(\G,\omega)$ cannot be
symplectically reduced, that is, if $\G$ does not have a non-trivial, isotropic ideal $\mathfrak{j}$, then
$(\G,\omega)$ will be called \textit{symplectically irreducible}, in this case $k=0$ and $(\G,\omega)=(\overline{\G},\overline{\omega})$.  Every irreducible symplectic Lie
algebra is metabelian and has the additional structure of a flat Kähler Lie algebra\footnote{See \cite{DM,LM} for this notion.}.

A notable class of symplectic Lie algebras that admit the trivial symplectic Lie algebra as a base are  called \textit{completely reducible}. Nilpotent symplectic Lie algebras are completely reducible \cite{B-C}. Every completely solvable symplectic Lie algebra $(\G,\omega)$ is completely reducible; This follows from the definition of complete solvability implies the existence of a one-dimensional ideal $\mathfrak{j}$ in $\G$, which is automatically isotropic in $(\G,\omega)$. Symplectic Novikov Lie algebras are completely reducible (see Theorem~$2$ in \cite{AM1}).

 A \textit{non completely reducible symplectic} Lie algebra is a symplectic Lie algebra which cannot be symplectically reduced to the trivial symplectic Lie algebra.  
As a first result of this paper, we present a refined classification of eight-dimensional non-completely reducible symplectic Lie algebras.

\begin{theo}
	Let $(\G,\om)$ be an eight-dimensional non completely reducible symplectic Lie algebra. Then,
	$(\G,\om)$ is isomorphic to exactly one of the following symplectic Lie algebras:
	
	\captionof{table}{Irreducible symplectic Lie algebras as 
0-step symplectic reductions}\label{t1}
\begin{center}
\begin{small}
	\begin{tabular}{lll}
\hline
Algebra	&Non vanishing Lie brackets& Symplectic structures \\
\hline
$\G_{8,1}$& $[e_7,e_1]=-e_2,\hspace{0.24cm}[e_7,e_2]= e_1,\hspace{0.5cm}[e_8,e_3]= -e_4, $	&$\omega_1=e^{12}+e^{34}+e^{56}+\lambda e^{78}$\\
 &$[e_8,e_4]= e_3,\hspace{0.45cm}[e_7, e_5] = -a e_6,\hspace{0.14cm} [e_7, e_6] = a e_5,$&$\omega_2=e^{12}+e^{34}-e^{56}+\lambda e^{78}$\\
&$ [e_8, e_5] = -be_6,\hspace{0.1cm} [e_8, e_6] = b e_5$&$a,b\in\R^\ast$, $a,b\neq1$,~~ $\lambda\in\R^{\ast}$\\
&&\\
$\G_{8,2}$& $	[e_7,e_1]=-e_2,\hspace{0.26cm}[e_7,e_2]= e_1,\hspace{0.5cm}[e_8,e_3]= -e_4,$	&$\omega_1=e^{12}+e^{34}+e^{56}+\lambda e^{78}$\\
&$[e_8,e_4]= e_3,\hspace{0.46cm}[e_7, e_5] = -a e_6,\hspace{0.14cm} [e_7, e_6] = a e_5$&$\omega_2=e^{12}+e^{34}-e^{56}+\lambda e^{78}$\\
&&$\omega_3=e^{12}+e^{34}+e^{56}+\mu e^{67}+\lambda e^{78}$\\
&&$\omega_4=e^{12}+e^{34}-e^{56}+\mu e^{67}+\lambda e^{78}$\\
&&$a\in\R^{\ast+}$, $a\neq1$,~~ $ \mu\in\R^{\ast+},~~\lambda\in\R^{\ast}$\\
&&\\
$\G_{8,3}$& $[e_7,e_1]=-e_2,\hspace{0.26cm}[e_7,e_2]= e_1,\hspace{0.5cm}[e_8,e_3]= -e_4, $	&$\omega_1=e^{12}+e^{34}+e^{56}+\lambda e^{78}$\\
&$[e_8,e_4]= e_3,\hspace{0.46cm}[e_7, e_5] = -e_6,\hspace{0.3cm} [e_7, e_6] =  e_5$&$\omega_2=e^{12}+e^{34}-e^{56}+\lambda e^{78}$\\
&&$\lambda\in\R^{\ast}$\\

		\hline
\end{tabular}
\end{small}
\end{center}
	\captionof{table}{Reducible symplectic Lie algebras as 
1-step symplectic reductions}\label{t2}
\end{theo}	

\begin{center}
\begin{small}
	\begin{tabular}{lll}
		\hline
	Algebra& Non vanishing Lie brackets & Symplectic structures \\
		\hline
	$\G_{8,1}^{\dagger}$&$[e_1, e_5] = e_2, \hspace{0.5cm} 
		[e_2, e_5] = -e_1,\hspace{0.5cm} [e_3, e_6] = e_4, $&$\omega=e^{12}+e^{34}+\eta e^{56}-e^{78}$	\\
		&$[e_4, e_6] = -e_3,\hspace{0.3cm} [e_8, e_5] = e_7,\hspace{0.7cm} [e_8, e_6] =\alpha e_7$&$ \alpha\in\R^+,~~\eta,\in\R^{\ast+}$\\
		&&	\\
		$\G_{8,2}^{\dagger}$&$[e_1, e_5] = e_2, \hspace{0.5cm} 
		[e_2, e_5] = -e_1,\hspace{0.5cm} [e_3, e_6] = e_4, $&$\omega=e^{12}+e^{34}+\eta e^{56}-e^{78},$\\
		&$[e_4, e_6] = -e_3,\hspace{0.3cm}[e_8, e_6] = e_7$&$ \eta\in\R^{\ast+}$\\
	&&\\
		$\G_{8,3}^{\dagger}$&$[e_1, e_5] = e_2, \hspace{0.5cm} 
		[e_2, e_5] = -e_1,\hspace{0.5cm} [e_3, e_6] = e_4, $&$\omega=e^{12}+e^{34}+\eta e^{56}-e^{78}$\\
	&$[e_4, e_6] = -e_3$&$\eta\in\R^{\ast+}$\\
		&&\\
		$\G_{8,4}^{\dagger\dagger}$&$[e_1, e_5] = e_2, \hspace{0.5cm} 
		[e_2, e_5] = -e_1,\hspace{0.5cm} [e_3, e_6] = e_4, $&$\omega=e^{12}+e^{34}+\eta e^{56}+\alpha e^{58}+\beta e^{68}-e^{78}$\\
		&$[e_4, e_6] = -e_3,\hspace{0.3cm}[e_8, e_7] =e_7$&$\eta\in\R^{\ast+},~\alpha,\beta\in\R$\\
		&&\\
$\G_{8,5}^{\dagger\dagger}$&$[e_1, e_5] = e_2, \hspace{0.5cm} 
		[e_2, e_5] = -e_1,\hspace{0.5cm} [e_3, e_6] = e_4, $&$\omega=e^{12}+e^{34}+\eta e^{56}+\alpha e^{58}+\beta e^{68}-\frac{1}{t}e^{78}$\\
		&$[e_4, e_6] = -e_3,\hspace{0.3cm}[e_1, e_2] =t e_7,\hspace{0.6cm}  [e_3, e_4] =t e_7,$&$\eta\in\R^{\ast+},~\alpha,\beta\in\R$,~~$t\in\R^\ast$\\
		&$[e_8, e_1] =\frac{1}{2}e_1,\hspace{0.31cm}  [e_8, e_2] =\frac{1}{2}e_2,\hspace{0.5cm}[e_8, e_3] =\frac{1}{2}e_3,$&\\
		& $ [e_8,e_4]=\frac{1}{2}e_4,\hspace{0.31cm}[e_8, e_7] =e_7$&\\
		&&\\
		$\G_{8,6}^{\dagger\dagger}$&$[e_1, e_5] = e_2, \hspace{0.5cm} 
		[e_2, e_5] = -e_1,\hspace{0.5cm} [e_3, e_6] = e_4, $&$\omega=e^{12}+e^{34}+\eta e^{56}+\alpha e^{58}+\beta e^{68}-\frac{1}{t}e^{78}$\\
		&$[e_4, e_6] = -e_3,\hspace{0.3cm}[e_1, e_2] =t e_7,\hspace{0.61cm} [e_8, e_1] =\frac{1}{2}e_1$&$\eta\in\R^{\ast+},~\alpha,\beta\in\R$,~~$t\in\R^\ast$\\
		&$[e_8, e_2] =\frac{1}{2}e_2,\hspace{0.32cm}[e_8, e_7] =e_7$&\\
		&&\\
		$\G_{8,7}^{\dagger\dagger}$&$[e_1, e_5] = e_2, \hspace{0.5cm} 
		[e_2, e_5] = -e_1,\hspace{0.5cm} [e_3, e_6] = e_4, $&$\omega=e^{12}+e^{34}+\eta e^{56}+\alpha e^{58}+\beta e^{68}-\frac{1}{t}e^{78}$\\
		&$[e_4, e_6] = -e_3,\hspace{0.3cm}[e_3, e_4] =t e_7,\hspace{0.6cm}[e_8, e_3] =\frac{1}{2}e_3,$&$\eta\in\R^{\ast+},~\alpha,\beta\in\R$,~~$t\in\R^\ast$\\
		&$[e_8,e_4]=\frac{1}{2}e_4,\hspace{0.3cm}[e_8, e_7] =e_7$&\\
		&&\\
$\G_{8,8}^{\dagger\dagger\dagger}$&$[e_1, e_5] = e_2,\hspace{0.49cm}  
		[e_2, e_5] = -e_1,\hspace{0.5cm}  
		[e_3, e_6] = e_4,$&$\omega=e^{12}+e^{34}+\eta e^{56}+e^{78}$
\\
&$[e_4, e_6] = -e_3,\hspace{0.3cm} [e_7, e_6] = -\mu e_7,\hspace{0.33cm}[e_8,e_6]=\mu e_8$&$\eta\in\R^{\ast+}$,~~$\mu\in\R^\ast$\\
		&&\\
		$\G_{8,9}^{\dagger\dagger\dagger}$&$[e_1, e_5] = e_2,\hspace{0.49cm}  
		[e_2, e_5] = -e_1,\hspace{0.5cm}  
		[e_3, e_6] = e_4,$&$\omega=e^{12}+e^{34}+\eta e^{56}+e^{78}$
\\
&$[e_4, e_6] = -e_3,\hspace{0.3cm} [e_7, e_5] = -\mu e_7,\hspace{0.33cm}[e_8,e_5]=\mu e_8$&$\eta\in\R^{\ast+}$,~~$\mu\in\R^\ast$\\
		&&\\
$\G_{8,10}^{\dagger\dagger\dagger}$&$[e_1, e_5] = e_2,\hspace{0.49cm}  
		[e_2, e_5] = -e_1,\hspace{0.5cm}  
		[e_3, e_6] = e_4,$&$\omega=e^{12}+e^{34}+\eta e^{56}+e^{78}$
\\
&$[e_4, e_6] = -e_3,\hspace{0.3cm} [e_7, e_5] = -\mu_1 e_7,\hspace{0.2cm}[e_7,e_6]=-\mu_2e_7,$&$\eta\in\R^{\ast+}$,~~$\mu_1,\mu_2\in\R^\ast$	\\
&$[e_8,e_5]=\mu_1 e_8,\hspace{0.21cm}[e_8,e_6]=\mu_2 e_8$&	
		\\\hline
	\end{tabular}	
	\end{small}

\end{center}
$\triangleright$ $\G_{8,j}^\dagger$ Lie algebras obtained by central symplectic oxidation.\\
$\triangleright$ $\G_{8,j}^{\dagger\dagger}$ Lie algebras obtained by normal symplectic oxidation.\\
$\triangleright$  $\G_{8,j}^{\dagger\dagger\dagger}$ Lie algebras obtained by generalized symplectic oxidation (neither central nor normal oxidation).

In the following, we briefly describe the main ideas behind the classification of non completely-reducible symplectic Lie algebras of dimension~$8$. Let $(\G,\omega)$ be a non completely reducible symplectic Lie algebra  of dimension $8$. We begin with the irreducible case, where $(\G,\omega)$ has no isotropic ideals. In this case, $(\G,\omega)$ is an irreducible symplectic Lie algebra and is $2$-step solvable, as characterized by Baues and Cortés~\cite{B-C} (see also Theorem~$\ref{Bause}$ and its corollary~$\ref{Cortes}$). Next, Propositions~$\ref{Classi8}$ and $\ref{Classi8forms}$ provide a classification of $8$-dimensional irreducible symplectic Lie algebras up to symplectomorphism. If $(\G, \omega)$ admits an isotropic ideal $\mathfrak{j}$, then $\mathfrak{j}$ is necessarily one-dimensional, and the symplectic reduction $(\overline{\G}, \overline{\omega})$ yields a six-dimensional irreducible symplectic Lie algebra. Proposition~\ref{Classi6} establishes that, up to isomorphism, there exists a unique six-dimensional irreducible symplectic Lie algebra. Furthermore, the Lie algebra structure of $\G$ can be explicitly described by Equation~$(\ref{Brackestgenera})$ under the hypotheses of Proposition~$\ref{prgeneralized}$.  A Lie algebra whose bracket structure is given by Equation~(\ref{Brackestgenera}) under the hypotheses of Proposition~$\ref{prgeneralized}$ will be called a \textit{generalized symplectic oxidation}. Moreover, Proposition~$\ref{Inverse}$ establishes that any symplectic Lie algebra admitting a one-dimensional isotropic ideal is a generalized symplectic oxidation.

As a result of generalized symplectic oxidation, we will develop a classification scheme. There are two special cases of generalized symplectic oxidation (see Remark~$\ref{remarkcentral}$): 
\begin{itemize}
    \item[$\triangleright$] Central symplectic oxidation
    \item[$\triangleright$] Normal symplectic oxidation
\end{itemize}
The next step is to characterize isomorphisms between two central symplectic oxidations and between two normal symplectic oxidations. This analysis proceeds by first considering the general situation and then restricting to these special cases. Theorem~$\ref{isombetewnoncentral}$ provides the isomorphism characterization for non-central oxidations, while Propositions~\ref{CharaCentral} and~\ref{Prnormal} specifically characterize isomorphisms of  central symplectic oxidations and normal symplectic oxidations. Consequently, Propositions~$\ref{ClassiCentral}$ and $\ref{ClassifiNormal}$ give a complete description of the two special cases (central symplectic oxidation and normal symplectic oxidation).

For the classification of generalized symplectic oxidations, we introduce a new Lie algebra construction via an endomorphism (called a \emph{$D$-extension}), as detailed in Subsection~$\ref{D-exten}$ and Proposition~$\ref{G,mu,D Liebrackets}$. Proposition~\ref{Characterisationgeneralized} establishes isomorphism criteria for $D$-extensions of Lie algebras. The non-central case is treated in two parts: Proposition~$\ref{Extennoncentral}$ analyzes the structure of non-central extensions, while Proposition~$\ref{isoofnocentral}$ characterizes their isomorphisms. We then establish necessary and sufficient conditions for a $D$-extension of a non-central extension to admit a symplectic structure (see Proposition~$\ref{Condisymp}$). Moreover, Theorem~$\ref{Moreconditions}$ characterizes isomorphisms between symplectic $D$-extensions of non-central extensions. Following this approach, we restrict the $D$-extension of symplectic Lie algebras to obtain our characterization of generalized symplectic oxidations (see Corollary~$\ref{Morerest}$). Finally, Proposition~$\ref{Classigeneral}$ provides a complete classification of generalized symplectic oxidations.

The paper is organized as follows. We begin by recalling fundamental results on irreducible symplectic Lie algebras in Section~\ref{se1}, followed by a classification, up to isomorphism, of irreducible eight-dimensional Lie algebras and their symplectic forms. In Section~\ref{se2}, we introduce the notion of generalized symplectic oxidation for Lie algebras and prove that every symplectic Lie algebra admitting an isotropic ideal arises as a generalized symplectic oxidation.
In Section~$\ref{se3}$, we describe the general framework and establish isomorphisms between central symplectic oxidations on one hand and normal symplectic oxidations on the other. Furthermore, we provide a complete classification of these Lie algebras in dimension~$8$. In  Section~\ref{se4},  we introduce a new construction of Lie algebras via endomorphisms (called \emph{D-extensions}) and investigate their basic properties. We then focus on D-extensions of non-central extensions, and establish an isomorphism theorem for symplectic D-extensions of non-central extensions of Lie algebras. At the end of Section~\ref{se4}, we complete the classification of eigh-dimensional non completely reducible Lie algebras by providing the classification of generalized symplectic oxidations and their  symplectic forms up to symplectomorphism in dimension~$8$.
\\\\

\textbf{Notations}:  Let $\G$ be a Lie algebra with basis $\{e_i\}_{i=1}^n$, and let $\{e^i\}_{i=1}^n$ denote the corresponding dual basis in $\G^\ast$. We denote by $e^{ij}$ the 2-form $e^i \wedge e^j \in \wedge^2 \G^\ast$. For any family $\mathcal{F} \subset \mathfrak{g}$, we denote $
\langle \mathcal{F} \rangle := \mathrm{span}_{\mathbb{R}}(\mathcal{F}) $
the Lie subalgebra generated by $\mathcal{F}$. 

Given an endomorphism $\mathrm{D}:\G\to\G$, a skew-symmetric bilinear form $\varphi:\G\times\G\to\K$, a linear map $\Phi:\G_1\to\G_2$, and linear forms $\lambda,\mu\in\G^\ast$, we define the following operations:
\begin{align*}
\varphi_{\mathrm{D}}(x,y) &:= \varphi(\mathrm{D}(x),y) + \varphi(x,\mathrm{D}(y)), \\
(\Delta\mathrm{D})(x,y) &:= \mathrm{D}([x,y]) - [\mathrm{D}(x),y] - [x,\mathrm{D}(y)], \\
(\Delta\Phi)(x,y) &:= \Phi([x,y]_1) - [\Phi(x), \Phi (y)]_2, \\
(\lambda \otimes \mathrm{D})(x,y) &:= \lambda(x)\mathrm{D}(y) - \lambda(y)\mathrm{D}(x), \\
(\lambda \otimes \mu)(x,y) &:= \lambda(x)\mu(y) - \lambda(y)\mu(x),
\end{align*}
where $[\cdot,\cdot]_1$ and $[\cdot,\cdot]_2$ denote Lie brackets in  $\G_1$ and $\G_2$, respectively.

\section{Irreducible symplectic Lie algebras}\label{se1}

We start this section by recalling Baues and Cortés' characterization of irreducible symplectic Lie algebras.

\begin{theo}\textsc{\cite{B-C}}\label{Bause}
Let $(\G,\omega)$ be a real symplectic Lie algebra. Then $(\G,\omega)$ is irreducible over the reals if and only if the following conditions hold$:$
\begin{enumerate}
\item The derived algebra $\mathcal{D}(\G)$ of $\G$ is a maximal abelian ideal of $\G$, which is non-degenerate
with respect to $\omega$.
\item The symplectic Lie algebra $(\G,\omega)$ is an orthogonal semi-direct sum of an abelian
symplectic subalgebra $(\h,\omega)$ and the ideal $(\mathcal{D}(\G),\omega)$.
\item The abelian ideal $(\mathcal{D}(\G),\omega)$ decomposes into an orthogonal sum of two-dimensional
purely imaginary irreducible submodules for $\h$, which are pairwise non-isomorphic.
\end{enumerate}

\end{theo}

\begin{co}\label{Cortes}
Let $(\G,\omega)$ be an irreducible symplectic Lie algebra. Then $\G$ is isomorphic to a Lie algebra $\G_{k,\lambda_1,\ldots,\lambda_m}$, where $k \geq 1$, $m \geq 2k$ and $\lambda_1,\ldots,\lambda_m$ is a set of mutually distinct
non-zero characters that span $\h^\ast$.
\end{co}

The classification of irreducible symplectic Lie algebras,  up to isomorphism, remains a significant and  unresolved problem. Dardié and Medina \cite{D-M} provided a description of Lie algebras admitting irreducible symplectic structures, while Baues and Cortés \cite{B-C} subsequently characterized these algebras. As the classification  demonstrates, the first  example of an irreducible symplectic Lie algebra occurs in dimension six.

Let $\G$ be an irreducible symplectic Lie algebra it can be written as: $\G = \h\ltimes \mathfrak{a}$. According to Theorem 2.4.3 in \cite{B-C} (see also Theorem~$\ref{Bause}$ above), and its proof, there exists a basis $\mathcal{B}=\{e_1,\cdots,e_{2h},e_1^1,e_2^1,\cdots,e_1^m,e_2^m\}$ such that: 
\[ 
\h = <e_1,\cdots,e_{2h}> \quad and \quad \mathfrak{a} = <e_1^1,e_2^1>\oplus \cdots \oplus <e_1^m,e_2^m> , \quad m\geq 2h,
\]

with: 
\[
[e_i,e_1^k]=\lambda_k(e_i)\J(e^k_2), \quad and \quad  [e_i,e_2^k]=\lambda_k(e_i)\J(e^k_1), \quad \lambda_k\in \h^*,
\]
where, the set of characters \(\lambda_1,\ldots,\lambda_{2h}\)
spans $\h^*$ and $\J$ is a complex structure on $\mathfrak{a}$.   Using the fact that $\mathfrak{a}$ is abelian, without loss of generality, we can choose the canonical form of the complex structure (i.e., $\J(e_1^k)=-e_2^k$,  for all $k\in\{1,\ldots, m\}$).

Now, we identify the basis $\langle e_1^1, e_2^1 \rangle \oplus \cdots \oplus \langle e_1^m, e_2^m \rangle$ of $\mathfrak{a}$ with $\langle e_1, e_2, \ldots, e_{2m} \rangle$ and the basis $\langle e_1, \ldots, e_{2h} \rangle$ of $\h$ with $\langle e_{2m+1}, \ldots, e_{2n} \rangle$. Under this identification, the Lie brackets of $\G$ are given by
\begin{align}
	[e_i, e_{2j-1}] &= -\tilde{\lambda}_j(e_i) e_{2j},\quad 2m+1\leq i\leq 2n,\quad 1\leq j\leq m, \label{LB1}\\
	[e_i, e_{2j}] &= \tilde{\lambda}_j(e_i) e_{2j-1}, \quad \tilde{\lambda}_j \in \mathfrak{h}^*. \label{LB2}
\end{align}
Furthermore, we express $\tilde{\lambda}_j$ as $\tilde{\lambda}_j = \sum_{k=2m+1}^{2n} \lambda_k e^k \in \mathfrak{h}^\ast$ and denote $\G = \G_{\lambda_1, \ldots, \lambda_{4(m+n+1)}}$.
\subsection{Six-dimensional irreducible symplectic Lie algebras}
We begin by presenting a classification of six-dimensional irreducible symplectic Lie algebras. Every symplectic form on the Lie algebra, exhibited in the following proposition, leads to an irreducible symplectic Lie algebra. 

\begin{pr}\label{Classi6}
	Let $\G_{(\lambda_1,\lambda_2,\lambda_3,\lambda_4)}$ be a six-dimensional irreducible symplectic Lie algebra. Then, $\G_{(\lambda_1,\lambda_2,\lambda_3,\lambda_4)}$ is isomorphic to $\G_{(1,0,0,1)}$, defined by the non-vanishing Lie brackets$:$
	\begin{align*}
		[e_1, e_5] &= e_2, \quad 
		[e_2, e_5] = -e_1, \\
		[e_3, e_6] &= e_4, \quad 
		[e_4, e_6] = -e_3,
	\end{align*}
	and the symplectic form
	\[
	\omega = e^{12} + e^{34} + \eta e^{56}, \quad \eta \in \mathbb{R}^{\ast +}.
	\]
\end{pr}
\begin{proof}
	Let $\G$ be a six-dimensional irreducible symplectic Lie algebra. We keep the previous notations, there exists a basis $\{e_1,e_2,e_3,e_4,e_5,e_6\}$ of $\G$, where $\mathfrak{h} = \langle e_5, e_6 \rangle$ and $\mathfrak{a} = \langle e_1, e_2, e_3, e_4 \rangle$ such that
	\begin{align*}
		[e_5, e_1] &= -\lambda_1 e_2, &[e_5, e_2] &= \lambda_1 e_1, &[e_6, e_1] &= -\lambda_2 e_2, &[e_6, e_2]& = \lambda_2 e_1, \\
		[e_5, e_3] &= -\lambda_3 e_4, & [e_5, e_4]& = \lambda_3 e_3, &[e_6, e_3] &= -\lambda_4 e_4, & [e_6, e_4]& = \lambda_4 e_3,
	\end{align*}
	where $\lambda_1, \lambda_2, \lambda_3, \lambda_4 \in \mathbb{R}$ satisfy $\delta = \lambda_1 \lambda_4 - \lambda_2 \lambda_3 \neq 0$ (a necessary condition for $\tilde{\lambda}_1$ and $\tilde{\lambda}_2$ to be mutually distinct non-zero characters spanning $\mathfrak{h}^\ast$). Put $\G=\G_{(\lambda_1,\lambda_2,\lambda_3,\lambda_4)}$. On the one hand, it is straightforward to verify that the map  
	\[
	\G_{\lambda_1,\lambda_2,\lambda_3,\lambda_4}  \longrightarrow \G_{(1,0,0,1)}, \quad (x, y) \mapsto (x, Ay), \quad x \in \mathfrak{a},\ y \in \h,
	\]
	where  
	\[
	A :=
	\begin{pmatrix}
		\lambda_1 & \lambda_2 \\
		\lambda_3 & \lambda_4
	\end{pmatrix}
	\]
	is the required Lie algebra isomorphism.

	On the other hand, for the Lie algebra $\mathfrak{g}_{(1,0,0,1)}=(\A_1\oplus\A_2)\rtimes\h$, where $\A_1=\langle e_1,e_2\rangle$, $\A_2=\langle e_3,e_4\rangle$ are abelian ideals and $\h=\langle e_5,e_6\rangle$, the corresponding Maurer-Cartan equations are:
	\[
	\partial e^1 = e^{25}, \quad \partial e^2 = -e^{15}, \quad \partial e^3 = e^{46}, \quad \partial e^4 = -e^{36}, \quad \partial e^5=\partial e^6 = 0.
	\]
	Using these relations, we compute:
	\begin{align*}
		B^2(\G_{(1,0,0,1)}) &= \langle e^{15}, e^{25}, e^{36} \rangle = e^5 \wedge \mathfrak{a}_1^\ast \oplus e^6 \wedge \mathfrak{a}_2^\ast, \\
		Z^2(\G_{(1,0,0,1)}) &= B^2(\G_{(1,0,0,1)}) \oplus \langle e^{12}, e^{34}, e^{56} \rangle.
	\end{align*}
	Consider the family  of automorphism of $\G_{(1,0,0,1)}$ given by
	\begin{small}
	\begin{equation}\label{groupauto}
		\phi_{\varepsilon_1,\varepsilon_2}
		=\begin{pmatrix}
			0&0&\varepsilon_1 x_{{24}}&0&0&x_{{16}}
			\\ 0&0&0&x_{{24}}&0&x_{{26}}\\ 
			\varepsilon_2 x_{{42}}&0&0&0&x_{{35}}&0\\ 0&x_{{42}}&0&0&x_{{45}}&0\\ 0&0&0&0&0&\varepsilon_1\\ 
			0&0&0&0&\varepsilon_2&0
		\end{pmatrix},\quad\varepsilon_1,\varepsilon_2=\pm1,~~x_{{24}}x_{{42}}\neq0.
	\end{equation}
	\end{small}
	
	We can show that for every symplectic form $\omega = \omega_0 + \partial \alpha$ on $\G_{(1,0,0,1)}$, where $\omega_0 \in Z^2(\G_{(1,0,0,1)}) \backslash B^2(\G_{(1,0,0,1)}) $ and $\alpha \in \mathfrak{g}^\ast$, we have $\phi^\ast_{(\epsilon_1,\epsilon_2,\epsilon_2,\epsilon_1)}(\omega) = \omega_0$. Indeed, let $\partial\alpha = \alpha_{15} e^{15} + \alpha_{25} e^{25} + \alpha_{36} e^{36} + \alpha_{46} e^{46}$, where $\alpha_{ij} \in \mathbb{R}$, and 
	\[
	\omega_0 = \omega_{12} e^{12} + \omega_{34} e^{34} + \omega_{56} e^{56} \quad \text{with} \quad \omega_{12}\omega_{34}\omega_{56} \neq 0.
	\]
	
	By applying $\phi^*_{(\epsilon_1,\epsilon_2)}$ with the given values $x_{45} = -\tfrac{\omega_{36}}{\omega_{34}}, ~ x_{35} = \tfrac{\omega_{46}}{\omega_{34}}, ~ x_{26} = -\tfrac{\omega_{15}}{\omega_{12}}, ~ x_{16} = \tfrac{\omega_{25}}{\omega_{12}},
	$ we obtain the desired result. Furthermore, every symplectic form $\omega$ has non-zero coefficients for the basis elements 
	$e^{12}$, $e^{34}$, and $e^{56}$ in $\langle e^{12}, e^{34}, e^{56} \rangle$; that is, 
	\[
	\omega = \omega_{12} e^{12} + \omega_{34} e^{34} +\omega_{56} e^{56} \quad \text{with} \quad \omega_{12}\omega_{34}\omega_{56} \neq 0.
	\]
	Finally, by applying the automorphism \(\phi_{(\epsilon_1,\epsilon_2)}\), with $x_{4 5} =0, x_{3 5}=0, x_{2 6} = 0$, and  $x_{1 6} = 0$, we successively prove that
	\begin{align*}  
		\omega_{\delta_1,\delta_2,\eta} &= \phi_{(1,1)}^* \omega = \delta_1 e^{12} + \delta_2 e^{34} + \eta e^{56}, \quad \delta_1,\delta_2 = \pm 1, \quad \eta \in \mathbb{R}^\ast, \\  
		\omega_{\delta_1,\delta_2,-\eta} &= \phi_{(1,1)}^* \omega_{\delta_1,\delta_2,\eta}, \\  
		\omega_{1,-1,\eta} &= \phi_{(1,-1)}^* \omega_{1,1,\eta}, \\  
		\omega_{-1,1,\eta} &= \phi_{(-1,1)}^* \omega_{1,1,\eta}, \\  
		\omega_{-1,-1,\eta} &= \phi_{(1,-1)}^* \omega_{1,-1,\eta}.  
	\end{align*}  
	Consequently, the symplectic form \(\omega_{\delta_1,\delta_2,\eta}\) is symplectomorphically isomorphic to \(\omega_{1,1,\eta}\), where \(\eta \in \mathbb{R}^{\ast+}\).

	The symplectic Lie algebras \((\G, \omega_{1,1,\eta})\) and \((\G, \omega_{1,1,\eta^\prime})\) with $\eta\neq\eta^\prime$ are pairwise non-symplectomorphic. Indeed, consider the symplectic connection of \((\G, \omega_{1,1,\eta})\).
	\[
	\nabla_x y = \frac{1}{3}\left([x, y] + x \cdot y\right).
	\]
It is easy to verify that the eigenvalues of the associated Ricci endomorphism 	are given by $\{i\frac{2}{9\eta},0,-i\frac{2}{9\eta}\}$.
\end{proof}
It is well known that an irreducible Lie algebra admits a K\"ahler structure. For the Lie algebra $\G_{(1,0,0,1)}$, the complex structure $\J$ defined by
\[
\J(e_1) = -e_2, \quad \J(e_3) = -e_4, \quad \J(e_5) = -e_6
\]
defines a K\"ahler structure compatible with the symplectic form $\omega_\lambda$. Moreover, Theorem~1.5 in \cite{Mi} implies that any Riemannian metric on an irreducible Lie algebra is necessarily flat.

Recall that a para-K\"ahler Lie algebra is a pseudo-Riemannian Lie algebra $(\G, \langle\cdot, \cdot\rangle)$ endowed with an isomorphism $\K: \G \to \G$ satisfying the following conditions:
\begin{enumerate}
	\item $\K^2 = \mathrm{Id}_\G$;
	\item $\K$ is skew-symmetric with respect to $\langle\cdot, \cdot\rangle$;
	\item $\K$ is invariant under the Levi-Civita product.
\end{enumerate}
A hyper-para-K\"ahler Lie algebra is a para-K\"ahler Lie algebra $(\G, \langle\cdot,\cdot\rangle, \K)$ endowed with a complex structure $\J$, i.e.,  $J^2 =-\mathrm{Id}_\G$, such that, $\J\K = -\K\J$, $\J$ is skew-symmetric with respect to $\langle\cdot,\cdot\rangle$ and $\J$ is invariant with respect to the Levi-Civita product.

Together with these definitions, we therefore have
\begin{pr}
	Let $\G_{(1,0,0,1)}$ be the unique six-dimensional irreductible symplectic Lie algebra.  Then, the pair $(\K,\langle\;,\;\rangle)$, defined by
	\[\K(e_1)=-e_3,\;\K(e_2)=e_4,\;\K(e_5)=-e_6,\] 
	\[\langle e_1,e_4\rangle=\langle e_2, e_3\rangle=-1\esp\langle e_5, e_5\rangle=-\langle e_6, e_6\rangle=\eta,\]
	is a para-Kähler Lie structure compatible with $\omega_\lambda$. Moreover, the pair $(\J,\K)$, constitutes a hyper-para-Kähler structure compatible with $\omega_\eta$.
\end{pr}
\begin{proof}
	This is a straightforward proof that follows directly from the previous definitions.
\end{proof}
\subsection{Eight-dimensional irreducible symplectic Lie algebras}
As a next step, we shall study the classification of eight-dimensional irreducible symplectic Lie algebras. 
\begin{pr}\label{Classi8}
	Let $\G=\h\ltimes\mathfrak{a}$ be an eight-dimensional irreducible symplectic Lie algebra. Then, $\G$ is isomorphic to one of the following  Lie algebras:
		\begin{enumerate}		
			\item $\G^{(a,b)}$, with  $ab\not=0$ $:$
		\begin{align*}
	[e_7, e_1] &= -e_2,& [e_7, e_2] &= e_1, &[e_8, e_3] &= -e_4, &[e_8, e_4]& = e_3,\\
[e_7, e_5] &= -a e_6,& [e_7, e_6]& = a e_5, &[e_8, e_5]& = -b e_6,& [e_8, e_6] &= b e_5.
		\end{align*}
		\item  $\G^{(a,0)}$, with   $a>0$  $:$ 
		\begin{align*}
			[e_7, e_1]& = -e_2,& [e_7, e_2]& = e_1,& [e_8, e_3]& = -e_4,\\
			[e_8, e_4]& = e_3,&[e_7, e_5]& = -a e_6,& [e_7, e_6] &= a e_5.
		\end{align*}
	\end{enumerate}
\end{pr}
\begin{proof}
	Let $\G_{\lambda_1,\lambda_2,\lambda_3,\lambda_4,\lambda_5,\lambda_6}$ be an eight-dimensional irreducible  Lie algebra. Then, there exists a basis $\{e_1,e_2,e_3,e_4,e_5,e_6,e_7,e_8\}$ where $\h=\langle e_7,e_8\rangle$ and $\mathfrak{a}=\langle e_1,e_2,e_3,e_4,e_5,e_6\rangle$ such that
	\begin{align}\label{Lie8}
		\begin{split}
			&[e_7, e_1] = -\lambda_1e_2,~~ [e_7, e_2] = \lambda_1 e_1,~~ [e_8, e_1] = -\lambda_2 e_2, ~~[e_8, e_2] = \lambda_2 e_1,\\
			&[e_7, e_3] = -\lambda_3 e_4 ,~~ [e_7, e_4] = \lambda_3e_3,~~ [e_8, e_3] = -\lambda_4e_4,~~ [e_8, e_4] = \lambda_4 e_3,\\
			&[e_7, e_5] = -\lambda_5 e_6 ,~~ [e_7, e_6] = \lambda_5e_5,~~ [e_8, e_5] = -\lambda_6e_6,~~ [e_8, e_6] = \lambda_6 e_5.
		\end{split}
	\end{align}

	The necessary  condition for the linear functionals $\tilde{\lambda}_1, \tilde{\lambda}_2, \tilde{\lambda}_3 \in \h^*$ to be distinct, non-zero characters, and to span $\h^*$ is that the matrix
	\[A=
	\begin{pmatrix}
		\lambda_1 & \lambda_3 & \lambda_5 \\
		\lambda_2 & \lambda_4 & \lambda_6
	\end{pmatrix}
	\]
	has rank $2$. Equivalently, at least one of the following $2 \times 2$ submatrix determinants must be nonzero:
	\[
	\delta_1 = \det\begin{pmatrix} \lambda_1 & \lambda_3 \\ \lambda_2 & \lambda_4 \end{pmatrix}, \quad
	\delta_2 = \det\begin{pmatrix} \lambda_1 & \lambda_5 \\ \lambda_2 & \lambda_6 \end{pmatrix}, \quad
	\delta_3 = \det\begin{pmatrix} \lambda_3 & \lambda_5 \\ \lambda_4 & \lambda_6 \end{pmatrix}.
	\]
	
	Put $\G_{\lambda_1,\lambda_2,\lambda_3,\lambda_4,\lambda_5,\lambda_6}=\G_{\delta_1,\delta_2,\delta_3}$. Without loss of generality, we assume $\delta_1 \neq 0$.  This simplification is justified by the following symmetry: for any $\delta_i \neq 0$, there exists a permutation $\sigma \in S_3$ such that the Lie algebra $\G_{\delta_1, \delta_2, \delta_3}$ is isomorphic to $\G_{\sigma(\delta_1,\delta_2,\delta_3)}=\G_{\delta_i, \delta_{i_1}, \delta_{i_2}}$, $i_1,i_2\in\{1,2,3\}, i_1\neq  i_2\neq i$, where $\sigma$ maps $\delta_i$ to the first position. 
	
	Consider the Lie algebra $\G_{\delta_1,\delta_2,\delta_3} = (\A_{1} \oplus \A_{2} \oplus \A_{3}) \rtimes \h$, parametrized by $\delta_1, \delta_2, \delta_3$. The permutation symmetry is realized by the following isomorphisms \[\Perm_{\sigma}:\G_{\delta_1,\delta_2,\delta_3}\to\G_{\sigma(\delta_1, \delta_2, \delta_3)}\]
	\[
	\Perm_\sigma(\A_i) = \A_{\sigma(i)} \quad \text{for } i = 1,2,3,\qquad \Perm_\sigma^2=\mathrm{Id}, \qquad \Perm_\sigma(\h) = \h.
	\]
	\begin{enumerate}
		\item \textbf{Case $\delta_1 \neq 0$:} $\Perm_\sigma$ acts as the identity:
		\[
		\Perm_\sigma(\A_{i}) = \A_{i},\quad i=1,2,3, \quad \Perm_\sigma(\h) = \h.
		\]
		
		\item \textbf{Case $\delta_1 = 0$, $\delta_2 \neq 0$:} $\Perm_\sigma$ acts as $(2\,3)$:
		\[
		\Perm_\sigma(\A_{1}) = \A_{1}, \quad \Perm_\sigma(\A_{2}) = \A_{3}, \quad \Perm_\sigma(\A_{3}) = \A_{2}, \quad \Perm_\sigma(\h) = \h.
		\]
		and acts on the basis vectors as
		\[\Perm_\sigma(e_1)=e_1,\;\Perm_\sigma(e_2)=e_2,\;\Perm_\sigma(e_3)=e_5,\;\Perm_\sigma(e_4)=e_6,\;\Perm_\sigma(e_7)=e_7,\;\Perm_\sigma(e_8)=e_8.\]
		This exchanges $\delta_1 \leftrightarrow \delta_2$ and fixes $\delta_3$.
		
		\item \textbf{Case $\delta_1 = \delta_2 = 0$, $\delta_3 \neq 0$:} $\Perm_\sigma$ acts as $(1\,3)$:
		\[
		\Perm_\sigma(\A_{1}) = \A_{3}, \quad \Perm_\sigma(\A_{2}) = \A_{2}, \quad \Perm_\sigma(\A_{3}) = \A_{1}, \quad \Perm_\sigma(\h) = \h.
		\]
		and acts on the basis vectors as
		\[\Perm_\sigma(e_1)=e_5,\;\Perm_\sigma(e_2)=e_6,\;\Perm_\sigma(e_3)=e_3,\;\Perm_\sigma(e_4)=e_4,\;\Perm_\sigma(e_7)=e_7,\;\Perm_\sigma(e_8)=e_8.\]
		This maps $\delta_3 \mapsto \delta_1$ and fixes $\delta_2$.
	\end{enumerate}
	
	In each case, $\Perm_\sigma:\G_{\delta_1,\delta_2,\delta_3}\to\G_{\sigma(\delta_1, \delta_2, \delta_3)}$ is an isomorphism  that permutes the ideals $\A_{i}$ while stabilizing $\h$, and its action on the parameters $\delta_i$ reflects the underlying permutation of their defining matrices.

	Using $(\ref{Lie8})$ it is easily verified that then the map
	
	\[
	\G_{(\lambda_1,\lambda_2,\lambda_3,\lambda_4,\lambda_5,\lambda_6)}  \longrightarrow \G_{(1,0,0,1,\frac{\delta_2}{\delta_1},\frac{\delta_3}{\delta_1})}, \quad (x, y) \mapsto (x, Ay), \quad x \in \mathfrak{a},\ y \in \h,
	\]
	where  
	\[
	A :=
	\begin{pmatrix}
		\lambda_1 & \lambda_2 \\
		\lambda_3 & \lambda_4
	\end{pmatrix}
	\]
	is the required Lie algebra isomorphism. Let $a=\frac{\delta_2}{\delta_1}$, $b=\frac{\delta_3}{\delta_1}$, and denoted $\G^{(a,b)}=\G_{
		(1,0,0,1,\frac{\delta_2}{\delta_1},\frac{\delta_3}{\delta_1})}$.

	Now let us show that  $\G^{(a,0)}\cong\G^{(0,a)}$ and $\G^{(a,0)}\cong\G^{(-a,0)}$.		
		On the one hand using the following isomorphism
	\[\Perm_\sigma(e_1)=e_3,\;\Perm_\sigma(e_2)=e_4,\;\Perm_\sigma(e_3)=e_1,\;\Perm_\sigma(e_4)=e_2,\]\[\Perm_\sigma(e_5)=e_5,\;\Perm_\sigma(e_6)=e_6\;\Perm_\sigma(e_7)=e_8,\;\Perm_\sigma(e_8)=e_7,\]
	we can show that $\G^{(a,0)}\cong\G^{(0,a)}$ and on the other hand
	the following isomorphism
	\[\Perm_\sigma(e_i)=e_i,\;i\in\{1,2,3,4,7,8\},\;\Perm_\sigma(e_5)=e_6,\;\Perm_\sigma(e_6)=e_5,\]
	we can show that $\G^{(a,0)}\cong\G^{(-a,0)}$. This completes the proof.
	 
\end{proof}

A basic and
straightforward observation is:
\begin{Le}
Let $(\G, \omega)$ be a $2n$-dimensional irreducible symplectic Lie algebra. For any $\alpha \in \G^\ast$, define $\omega_\alpha = \omega + \md \alpha$. If $\omega_\alpha$ is non-degenerate, then $(\G, \omega_\alpha)$ is also a $2n$-dimensional irreducible symplectic Lie algebra.
\end{Le}
\begin{proof}
Let $(\G, \omega)$ be an irreducible symplectic Lie algebra. Let $\alpha \in \G^\ast$, and set $\omega_\alpha = \omega + \md\alpha$. Suppose that $\omega_\alpha$ is non-degenerate, then $\omega_\alpha$ is a symplectic form on $\G$.

Now, if there exists an isotropic ideal $\mathfrak{j}$ of $(\G, \omega_\alpha)$, then
\begin{align*}
\omega_\alpha(\mathfrak{j}, \mathfrak{j}) &= \omega(\mathfrak{j}, \mathfrak{j}) - \alpha\big([\mathfrak{j}, \mathfrak{j}]\big) = \omega(\mathfrak{j}, \mathfrak{j}) = 0.
\end{align*}
Due to the fact that $\mathfrak{j}$ is an isotropic ideal of $(\G, \omega)$, then  $\mathfrak{j}$ is abelian. This shows that $\mathfrak{j}$ is an isotropic ideal of $(\G, \omega)$, contradicting the fact that $(\G, \omega)$ is an irreducible symplectic Lie algebra. Therefore, $(\G, \omega_\alpha)$ is an irreducible symplectic Lie algebra.
\end{proof}

\begin{pr}\label{Classi8forms}
	Let $\G$ be an eight-dimensional irreducible symplectic Lie algebra. Then, $\G$ is symplectomorphically isomorphic to exactly one of the following Lie algebras equipped with a symplectic structure:
\begin{align*}
	\G^{a,b}:&~~&\omega_1&=e^{12}+e^{34}+e^{56}+\lambda e^{78}&&\\
	&~~&\omega_2&=e^{12}+e^{34}-e^{56}+\lambda e^{78}&\lambda\in\R^{\ast},\;a,b\neq0,1&\\
\G^{a,0}:&~~&\omega_1&=e^{12}+e^{34}+e^{56}+\lambda e^{78}&&\\
&~~&\omega_2&=e^{12}+e^{34}-e^{56}+\lambda e^{78}&&\\
	 &~~&\omega_3&=e^{12}+e^{34}+e^{56}+\mu e^{67}+\lambda e^{78}&\\
	 &~~&\omega_4&=e^{12}+e^{34}-e^{56}+\mu e^{67}+\lambda e^{78}&\lambda\in\R^\ast,~~
	 \mu\in\R^{\ast+},~~a\in\R^{\ast+},~~a\neq1\\
	\G^{1,0}:&~~&\omega_1&=e^{12}+e^{34}+e^{56}+\lambda e^{78}&&\\
	 &~~&\omega_2&=e^{12}+e^{34}-e^{56}+\lambda e^{78}&\lambda\in\R^\ast	& 
\end{align*}

\end{pr}
\begin{proof}
Every eight-dimensional irreducible symplectic Lie algebra is isomorphic to $\G^{a,b}$, $\G^{a,0}$, the family of Lie algebras classified in Proposition~$\ref{Classi8}$. The Maurer-Cartan equations for $\G^{a,b}$, expressed in terms of its structure constants, are given by:
\begin{align*}
    \partial e^1 &= -e^{27}, 
    &\partial e^2 &= e^{17}, &\partial e^3 &= -e^{28}, \\
    \partial e^4 &= e^{38}, &\partial e^5 &= -a e^{67} - b e^{68},&\partial e^6 &= a e^{57} + b e^{58}, \\
    \partial e^7 &= 0,&\partial e^8& = 0.
\end{align*}
Now define an element $\omega\in\bigwedge^2\big(\G^{(a,0)}\big)^\ast$, i.e., \[\omega=\mathop{\resizebox{1.3\width}{!}{$\sum$}}\limits_{1\leq i<j\leq8}\omega_{ij}e^{ij},\quad\omega_{ij}\in\R.\]
It is straightforward to verify that $\omega\in Z^2(\G^{(a,0)})$, i.e., $\md\omega=0$ is equivalent 
\begin{align*}
\omega&=\omega_{12}e^{12}+\omega_{34}e^{34}+\omega_{78}e^{78}+\omega_{15}(e^{15}+ae^{26})+\omega_{25}(-ae^{16}+e^{25})\\
&~~~~+\omega_{17}e^{17}+\omega_{27}e^{27}+\omega_{38}e^{38}+\omega_{48}e^{48}+\omega_{56}e^{56}+\omega_{57}e^{57}+\omega_{67}e^{67},
\end{align*}
subject to the condition $a^2\omega_{2 5} - \omega_{2 5}=-a^2\omega_{1 5} + \omega_{1 5}=0$. Since $a>0$, we consider two cases: $a=1$ or $a\neq1$.
On the other hand, consider the Lie algebra $\G^{(a,b)}$ with parameters satisfying $ab \neq 0$. In a similar manner, we can show that any $\omega \in Z^2(\G^{(a,b)})$ has the following form:
\begin{align*}
\omega &= \omega_{12}e^{12}+\omega_{17}e^{17}+\omega_{27}e^{27}+\omega_{34}e^{34}+\omega_{38}e^{38}+\omega_{48}e^{48}+\omega_{56}e^{56}\\
&~~~~+\omega_{57}e^{57}+\tfrac{b}{a}\omega_{57}e^{58}+\omega_{67}e^{67}+\tfrac{b}{a}\omega_{67}e^{68}+\omega_{78}e^{78}.
\end{align*}

Let $\Omega(\G)$ denote the space of all symplectic forms on $\G$. We now proceed to classify the orbit space of the action of $\mathrm{Aut}(\G)$ on $\Omega(\G)$. It is straightforward to observe that if $\omega_1, \omega_2 \in \Omega(\G)$ lie in the same $\mathrm{Aut}(\G)$-orbit, then they are symplectomorphically isomorphic. More precisely, two symplectic forms $\omega_1, \omega_2 \in \Omega(\G)$ belong to the same $\mathrm{Aut}(\G)$-orbit if and only if there exists an automorphism $\Phi \in \mathrm{Aut}(\G)$ such that $\Phi^* \omega_2 = \omega_1$.

Due to the reliance on standard computations, we defer the detailed proof to Appendix~$\ref{Appen}$.
\end{proof}

\section{Generalized symplectic oxidation of Lie algebras}\label{se2}
Let $(\G,\omega)$  be a $2n$-dimensional symplectic Lie algebra, and let $\mathfrak{j}=\langle \xi\rangle$  be a one-dimensional isotropic ideal of $(\G,\omega)$. Let us denote $\ol{\G}=(\R \ell\oplus \R \xi)^\perp$. Then, we have 
$\G=\R \ell\oplus \ol{\G} \oplus \R \xi$ and 
$(\R \xi)^\perp=\ol{\G}\oplus \R \xi$. For any $x,y\in (\R \ell\oplus \R \xi)^\perp= (\R \xi)^\perp \cap (\R \ell)^\perp$, we have  $[x,y]\in (\R \xi)^\perp$. Therefore, the non-zero brackets are given by
\begin{align*}
[x,y]&=\ol{[x,y]}+\varphi(x,y)\xi,& \forall x,y\in \ol{\G},\\
[\xi,x]&= -\mu(x)\xi,&  \forall x\in \ol{\G},\\ 
[\xi, \ell]&= t\xi,& t\in\R,\\
[\ell,x]&=\mu_1(x)\ell+\mathrm{D}(x)+\lambda(x)\xi,&\forall x\in \ol{\G},
\end{align*}
where, $\varphi$ is a $2$-cochain on $\ol{\G}$,  $\mathrm{D}$ is an endomorphism of $\ol{\G}$,  $\mu_1$, $\mu$ and $\lambda$ are  linear forms of $\ol{\G}$. 

Since $\omega$ is symplectic, then
\begin{align*}
\mathop{\resizebox{1.3\width}{!}{$\sum$}}\limits_{\mathrm{cycl}}\om([[x,y],\ell])&=-\varphi(x,y)-\om(\mathrm{D}(y),x)+\om(\mathrm{D}(x),y).
\end{align*}
Therefore, $\varphi=\om_\mathrm{D}=\overline{\omega}_\mathrm{D}$ and also
\begin{align*}
\mathop{\resizebox{1.3\width}{!}{$\sum$}}\limits_{\mathrm{cycl}}\om([[x,\xi],\ell])&=-\mu_1(x)+\mu(x).
\end{align*}
Thus, $\mu_1=\mu$. Consequently, the brackets  become
\begin{align}\label{Brackestgenera}
\begin{split}
[x,y]&=\ol{[x,y]}+\overline{\omega}_\mathrm{D}(x,y),\xi\hspace{3cm} \forall x,y\in \ol{\G},\\
[\xi,x]&= -\mu(x)\xi,\hspace{4.1cm}\quad  \forall x\in \ol{\G},\\ 
[\xi,\ell]&= t\xi,\hspace{5.4cm}t\in\R,\\
[\ell,x]&=\mu(x)\ell+\mathrm{D}(x)+\lambda(x)\xi,\hspace{2.55cm}\forall x\in \ol{\G}.
\end{split}
\end{align} 
\begin{pr}\label{prgeneralized}
	Let $(\G,\om)$  be a symplectic Lie algebra which reduces to $(\ol{\G},\ol{\om})$ with respect to the one-dimensional isotropic ideal $\mathfrak{j}=\langle\xi\rangle$. Then the data $(\mathrm{D},\mu,\lambda,t)$ verify
\begin{enumerate}

\item $\mathop{\resizebox{1.3\width}{!}{$\sum$}}\limits_{\mathrm{cycl}}\ol\om_\mathrm{D}(\ol{[x,y]},z)-\ol\om_\mathrm{D}(x,y)\mu(z)=0$
\item $\Delta \mathrm{D}-\mu\otimes\mathrm{D} = 0,$
\item $t\ol{\om}_{\mathrm{D}}-\ol{\om}_{\mathrm{D},\mathrm{D}}=\md\lambda-2\lambda\otimes\mu,$
\item $\mu\circ \mathrm{D}=t\mu$ and $\md\mu=0$.
\end{enumerate}
\end{pr}
\begin{proof}
For all $x$, $y$ and $z\in\ol\G$, we have 	
\begin{align*}
\mathop{\resizebox{1.3\width}{!}{$\sum$}}\limits_{\mathrm{cycl}}[[x,y],z]&=\mathop{\resizebox{1.3\width}{!}{$\sum$}}\limits_{\mathrm{cycl}}[\ol{[x,y]}+\ol\om_\mathrm{D}(x,y)\xi,z]\\
&=\mathop{\resizebox{1.3\width}{!}{$\sum$}}\limits_{\mathrm{cycl}}\ol{[\ol{[x,y]},z]}+\mathop{\resizebox{1.3\width}{!}{$\sum$}}\limits_{\mathrm{cycl}}\left(\ol\om_\mathrm{D}(\ol{[x,y]},z)-\ol\om_\mathrm{D}(x,y)\mu(z)\right)\xi\\
&=\mathop{\resizebox{1.3\width}{!}{$\sum$}}\limits_{\mathrm{cycl}}\left(\ol\om_\mathrm{D}(\ol{[x,y]},z)-\ol\om_\mathrm{D}(x,y)\mu(z)\right)\xi.
\end{align*}
Then\begin{equation}
	\mathop{\resizebox{1.3\width}{!}{$\sum$}}\limits_{\mathrm{cycl}}\left(\ol\om_\mathrm{D}(\ol{[x,y]},z)-\ol\om_\mathrm{D}(x,y)\mu(z)\right)=0,\quad\text{for all}\quad x,y,z\in\ol\G.
\end{equation}
For all $x,y\in\ol\G$, we have
\begin{align*}
\mathop{\resizebox{1.3\width}{!}{$\sum$}}\limits_{\mathrm{cycl}}[[x,\ell],y]&=-[\mu(x)\ell+\mathrm{D}(x)+\lambda(x)\xi,y]-[\ol{[x,y]}-\overline{\omega}_\mathrm{D}(x,y)\xi,\ell]+[\mu(y)\ell+\mathrm{D}(y)+\lambda(y)\xi,x]\\
&=-\mu(x)\mu(y)\ell-\mu(x)\mathrm{D}(y)-\mu(x)\lambda(y)\xi-\ol{[\mathrm{D}(x),y]}-\overline{\omega}_\mathrm{D}(\mathrm{D}(x),y)\xi+\lambda(x)\mu(y)\xi\\
&+\mu(\ol{[x,y]})\ell+\mathrm{D}(\ol{[x,y]})+\lambda(\ol{[x,y]})\xi+t\overline{\omega}_\mathrm{D}(x,y)\xi\\
&+\mu(x)\mu(y)\ell+\mu(y)\mathrm{D}(x)+\mu(y)\lambda(x)\xi+\ol{[\mathrm{D}(y),x]}+\overline{\omega}_\mathrm{D}(\mathrm{D}(y),x)\xi-\lambda(y)\mu(x)\xi.
\end{align*}
So
\[
\begin{cases}
\mu(\ol{[x,y]})=0\\
(\Delta \mathrm{D})(x,y)-\mu(y)\mathrm{D}(x)-\mu(x)\mathrm{D}(y) = 0,\\
t\overline{\omega}_\mathrm{D}(x,y)-\overline{\omega}_\mathrm{D}(x,\mathrm{D}(y))-\overline{\omega}_\mathrm{D}(\mathrm{D}(x),y)+\lambda(\ol{[x,y]})+2\lambda(x)\mu(y)
-2\lambda(y)\mu(x) = 0
\end{cases}
\]
We thus have
\[
\begin{cases}
\md\mu=0\\
\Delta\mathrm{D}-\mu\otimes\mathrm{D} = 0,\\
t\overline{\omega}_\mathrm{D}-\overline{\omega}_{\mathrm{D},\mathrm{D}}-\md\lambda+2\lambda\otimes\mu = 0.
\end{cases}
\]
 On the one hand,
\begin{align*}
\mathop{\resizebox{1.3\width}{!}{$\sum$}}\limits_{\mathrm{cycl}}[[x,\ell],\xi]&=-[\mu(x)\ell+\mathrm{D}(x)+\lambda(x)\xi,\xi]\\
&=\big(t\mu(x)-\mu(\mathrm{D}(x))\big)\xi.
\end{align*}
On the other hand,
\begin{align*}
\mathop{\resizebox{1.3\width}{!}{$\sum$}}\limits_{\mathrm{cycl}}[[x,\xi],y]&=-\mu(x)\mu(y)\xi+\mu(x)\mu(y)\xi-\mu(\ol{[x,y]})\xi\\
&=-\mu(\ol{[x,y]})\xi.
\end{align*}
\end{proof}
\begin{Def}
The symplectic Lie algebra $(\G, \omega) =(\G_{\mathrm{D},\lambda,\mu,t},\omega)$ is called generalized symplectic oxidation of $(\overline{\G}, \overline{\omega})$ with respect to the data $\mathrm{D}$, $\lambda$, $\mu$ and $t$.
\end{Def}

\begin{remark}\label{remarkcentral}
\begin{enumerate}
	\item If $\xi$ is a central element, that is, $\mu=0$ and $t=0$. It follows that
	\[\om(\xi,[x,y])=-\om(x,[y,\xi])-\om(y,[\xi,x])=0.\]
Then, $[\G,\G]\subset \langle\xi\rangle^\perp=\ol\G\oplus \xi$  and the brackets  become
\begin{align*}
	[x,y]&=\ol{[x,y]}+\overline{\om}_\mathrm{D}(x,y)\xi,& \quad\text{for all}\quad x,y\in \ol{\G},\\
	[\ell,x]&=\mathrm{D}(x)+\lambda(x)\xi,&\quad\text{for all}\quad  x\in \ol{\G}.
\end{align*} 
Where $\mathrm{D}$ becomes a derivation of $\overline{\G}$ and $\lambda\in\overline{\G}^*$ is a $1$-form  satisfying $\overline{\om}_{\mathrm{D},\mathrm{D}}=-\md\lambda$.  This special case of generalized symplectic oxidation essentially coincides with the "symplectic oxidation", or "double extension" as
developed in \textsc{\cite{B-C}} and \textsc{\cite{M-R}}.
	\item If $(\R\xi)^\perp$ is an ideal, which is equivalent to $\mu=0$. In this case, $(\overline{\G},\overline{\om})$ is a normal symplectic reduction
	of $(\G,\om)$ with respect to the isotropic ideal $\langle\xi\rangle$, see \textsc{\cite{B-C}}.
	We also say that $(\G,\om)$ is a normal symplectic oxidation of $(\overline{\G},\overline{\omega})$ with respect to the data $(\mathrm{D},\lambda,t)$.
\end{enumerate}
\end{remark}

Let $(\G, \omega)$ be a $2n$-dimensional symplectic Lie algebra, $\mathfrak{j}=\langle\xi\rangle$
a one-dimensional isotropic ideal and $(\overline{\G},\overline{\omega})$ the symplectic reduction with respect to $\mathfrak{j}$. By choosing $\ell$ in the complement of $\mathfrak{j}^{\perp_\omega}$ with $\omega(\xi,\ell) = 1$, we obtain an isotropic direct sum decomposition
\begin{align*}
\G&=\langle\ell\rangle\oplus\mathfrak{j}^{\perp_\omega}=\langle\ell\rangle\oplus\overline{\G}\oplus\mathfrak{j}.
\end{align*}
The adjoint action $\mathrm{ad}(\xi)$ restricts to an endomorphism
\[
\mathrm{D}_\ell \in \mathrm{End}(\overline{\G}), \quad \overline{\G} := \mathfrak{j}^{\perp_\omega}/\mathfrak{j},
\]
where $\mathrm{D}_\ell = \mathrm{ad}(\ell)|_{\overline{\G}}$. With respect to the above decomposition, we obtain the expressions given in~$(\ref{Brackestgenera})$, which satisfy the hypotheses of  Proposition~$\ref{prgeneralized}$.

Put the following for all $x \in \overline{\G}$:
\begin{enumerate}
    \item $\mu_\omega(x) = -\omega\big(\ell, [\xi, x]\big)$,
    \item $\lambda_\omega(x) = -\omega\big(\ell, [\ell, x]\big)$,
    \item $t_\omega = -\omega\big(\ell, [\xi, \ell]\big)$.
\end{enumerate}

The generalized symplectic oxidation construction (Proposition~\ref{prgeneralized}) 
reverses symplectic reduction by one-dimensional ideals.

\begin{pr}\label{Inverse}
	Let $(\G,\om)$ be a $2n$-dimensional symplectic Lie algebra  which is reduced to $(\ol{\G},\ol{\om})$  with respect to a one-dimensional ideal $\mathfrak{j}=\langle\xi\rangle$. Then 
	\begin{enumerate}
	\item $\mu_\omega=\mu$,
	\item  $\lambda_\omega=\lambda$,
	\item $t_\omega=t$.
	
	\end{enumerate}
Furthermore, $(\G,\om)$ is a generalized symplectic oxidation of $(\ol{\G},\ol{\om})$ with respect to $\mathrm{D}_\ell$, $\mu_\omega$, $\lambda_\omega$ and $t_\omega$.	
\end{pr}
\begin{proof}
	Since $[\xi, x] = \mu(x)\xi$, it follows that
\[
\omega\big(\ell, [\xi, x]\big) = \omega\big(\ell, \mu(x)\xi\big) = -\mu(x).
\]
Thus, $1.$ holds.

Additionally, since $[\ell, x] = \mu(x)\ell + \mathrm{D}(x) + \lambda(x)\xi$, we immediately obtain
\[
\omega\big(\ell, [\ell, x]\big) = -\lambda(x).
\]
Hence, $2.$ holds. Since $[\xi, \ell] = t\xi$, we evidently have
$\omega(\ell, [\xi, \ell]) = -t$. Therefore, $3.$ holds. The remaining assertion follows as a direct consequence of Proposition~$\ref{prgeneralized}$.

\end{proof}

\begin{pr}\label{solvanilpoLevi}
Let $(\G_{(\mathrm{D},\lambda,\mu,t)},\omega)$ be a generalized symplectic oxidation of $(\ol{\G},\ol{\om})$ with respect to $\mathrm{D},\lambda,\mu$ and $t$. Then, the following equivalences hold for $\G_{(\mathrm{D},\lambda,\mu,t)}$
\begin{enumerate}
\item It is a nilpotent Lie algebra if and only if
\begin{enumerate}
\item $\overline{\G}$ is nilpotent,
\item $\mu\equiv0$,
\item $\mathrm{D}$ is a nilpotent endomorphism.\\
Furthermore, if $\overline{\G}$ is $k$-step nilpotent, and $\mathrm{D}$ is a $p$-nilpotent endomorphism with $k \geq p$ $($resp. $k < p)$, then $\G$ is at most $(k+2)$-step $($resp. $(p+2)$-step$)$ nilpotent.
 \end{enumerate}
\item It is a solvable Lie algebra if and only if $\overline{\G}$ is solvable. Moreover, if $\overline{\G}$ is $k$-step solvable, then $\G$ is at most $(k+2)$-step solvable.
\item It has nontrivial Levi-Malcev decomposition if and only if $\overline{\G}$ has nontrivial Levi-Malcev decomposition, i.e., $\overline{\G}=\mathfrak{rad}(\overline{\G})\rtimes\mathfrak{s}$. Moreover, $\G=\langle\ell\rangle\oplus\mathfrak{rad}(\overline{\G})\oplus\langle\xi\rangle\rtimes\mathfrak{s}
$.

\end{enumerate}

\end{pr}
\begin{proof}
Suppose that $\G = \G_{(\mathrm{D}, \lambda, \mu, t)}$ is a nilpotent Lie algebra of nilpotency degree $k$. It follows from Engel's characterization Theorem that $\G$ is nilpotent if and only if, for each $x \in \G$, the operator $\mathrm{ad}_x$ is nilpotent. Then,  for all $x\in\overline{\G}$, we have $\mathrm{ad}_x^k\xi=\mu^k(x)\xi$ and $\mathrm{ad}_\xi^k\ell=t^k\xi$ for all $k\in\mathbb{N}^\ast$. This implies that $t=0$ and $\mu(x)=0$  for all $x\in\overline{\G}$. We also have $\mathrm{ad}_xy=\overline{\mathrm{ad}}_xy+\overline{\omega}_{\mathrm{D}}\left(x,y\right)\xi$. By induction, we obtain  $\mathrm{ad}^k_xy=\overline{\mathrm{ad}}^k_xy+\overline{\omega}_{\mathrm{D}}\big(x,\overline{\mathrm{ad}}^{k-1}_xy\big)\xi$~ for all $x,y\in\overline{\G}$, and $k\geq1$. This  show that $\overline{\G}$ is nilpotent. For all $x\in\overline{\G}$, we have $\mathrm{ad}_\ell x=\mathrm{D}(x)+\lambda\left(x\right)\xi$, and thus, for all $k\in\mathbb{N}^\ast$, $\mathrm{ad}^k_\ell x=\mathrm{D}^k(x)+\lambda\left(\mathrm{D}^{k-1}(x)\right)\xi$. Therefore, $\mathrm{D}$ is a nilpotent endomorphism. Finally, for all $x \in \overline{\G}$, we have $\mathrm{ad}_x\ell = -\mathrm{D}(x) - \lambda(x)\xi$, and for all $k \geq 2$, $\mathrm{ad}_x^k \ell = -\overline{\mathrm{ad}}^{k-1}_x \big(\mathrm{D}(x)\big)-\overline{\omega}_{\mathrm{D}}\Big(x,\overline{\mathrm{ad}}^{k-2}_x\big(\mathrm{D}(x)\big)\Big)\xi$; moreover, if $\overline{\G}$ is $(k-2)$-step nilpotent, this equality holds. Conversely, if $\overline{\G}$ is a $k$-step nilpotent Lie algebra with $\mu \equiv 0$ and $\mathrm{D}$ a nilpotent endomorphism, then $\G_{(\mathrm{D},0,\lambda,0)}$ is nilpotent, which is exactly the symplectic oxidation of a nilpotent Lie algebra. This shows the first equivalence.

Suppose now that $\G=\G_{(\mathrm{D},\mu,\lambda,t)}$ is solvable. We discuss now derived series $\mathcal{D}^i(\G)$ of ideals in  $\G$. On the one hand, we have
\begin{align*}
\mathcal{D}^1(\G)&=\langle\mathcal{D}^1(\overline{\G}),\mu(\overline{\G})\ell\oplus\mathrm{D}(\overline{\G}),\xi\rangle.
\end{align*}
On the other hand, for all $x,y\in\overline{\G}$, we have
\begin{align*}
[\mu(x)\ell+\mathrm{D}(x),\mu(y)\ell+\mathrm{D}(y)]&=\mu(x)\mu\big(\mathrm{D}(y)\big)\ell+\mu(x)\mathrm{D}^2(y)+\mu(x)\lambda\big(\mathrm{D}(y)\big)\xi\\
&~~~~-\mu(y)\mu\big(\mathrm{D}(x)\big)\ell-\mu(y)\mathrm{D}^2(x)-\mu(y)\lambda\big(\mathrm{D}(x)\big)\xi\\
&~~~~+\overline{[\mathrm{D}(x),\mathrm{D}(y)]}+\overline{\omega}_{\mathrm{D}}\big(\mathrm{D}(x),\mathrm{D}(y)\big)\xi\\
&=t\mu(x)\mu(y)\ell+\mu(x)\mathrm{D}^2(y)+\mu(x)\lambda\big(\mathrm{D}(y)\big)\xi\\
&~~~~-t\mu(y)\mu(x)\ell-\mu(y)\mathrm{D}^2(x)-\mu(y)\lambda\big(\mathrm{D}(x)\big)\xi\\
&~~~~+\overline{[\mathrm{D}(x),\mathrm{D}(y)]}+\overline{\omega}_{\mathrm{D}^\ast\mathrm{D}^2}\big(x,y\big)\xi\\
&=(\mu\otimes\mathrm{D}^2)(x,y)+\overline{[\mathrm{D}(x),\mathrm{D}(y)]}+(\mu\otimes\lambda\circ\mathrm{D})(x,y)\xi
\end{align*}
Then, the second derived algebra decomposes as:
\begin{align*}
\mathcal{D}^2(\G)&=\langle\mathcal{D}^2(\overline{\G}),\overline{[\mathrm{D}(\overline{\G}), \mathrm{D}(\overline{\G})]}\oplus\mu\otimes\mathrm{D}^2(\overline{\G},\overline{\G}),\xi\rangle.
\end{align*}
We  will denote the space $\overline{[\mathrm{D}(\overline{\G}), \mathrm{D}(\overline{\G})]} 
\oplus \mu \otimes \mathrm{D}^2(\overline{\G},\overline{\G})$ by $\mathcal{V}\subset\overline{\G}$ 
to simplify the computation of the third derived series $\mathcal{D}^3$.  We obtain, 
\begin{align*}
[\mathcal{D}^2(\overline{\G})\oplus\mathcal{V}\oplus\langle\xi\rangle,\mathcal{D}^2(\overline{\G})\oplus\mathcal{V}\oplus\langle\xi\rangle]&=\langle\mathcal{D}^3(\overline{\G}),\mathcal{D}^2(\overline{\G}),\mathcal{D}^1(\overline{\G}),\xi\rangle.
\end{align*}
Thus,
\begin{align*}
\mathcal{D}^3(\G)&=\langle\mathcal{D}^3(\overline{\G}),\mathcal{D}^2(\overline{\G}),\mathcal{D}^1(\overline{\G}),\xi\rangle.
\end{align*}
Therefore, for all $k\geq4$, we have
\begin{align*}
\mathcal{D}^{k}(\G)&=\langle\mathcal{D}^{k}(\overline{\G}),\mathcal{D}^{k-1}(\overline{\G}),\mathcal{D}^{k-2}(\overline{\G}),\kappa\xi\rangle.
\end{align*}
where $\kappa \in \R$ depends on the degree of solvability, such that $\kappa = 0$ if $\overline{\G}$ is at least $(k-2)$-step solvable.   
If in addition $\G$ is $k$-step solvable, then $\overline{\G}$ is a $(k-2)$-step solvable Lie algebra.

Assume now that $\G$ has nontrivial Levi-Malcev decomposition, i.e., $\G=\mathfrak{r}\rtimes\mathfrak{s}$, with $\mathfrak{s}$ is the Levi factor, and $\mathfrak{r}$ its radical part.  The following relations hold: \[[\mathfrak{r},\mathfrak{r}]\subset\mathfrak{r}\quad[\mathfrak{r},\mathfrak{s}]\subset\mathfrak{r},\quad[\mathfrak{s},\mathfrak{s}]\subset\mathfrak{s}.\]

 Note that $\mathrm{Tr}(\mathrm{ad}_s) = 0$ for all $s \in \mathfrak{s}$. Indeed, since $\mathfrak{s}$ is perfect (i.e., $\mathfrak{s} = [\mathfrak{s}, \mathfrak{s}]$), any $s \in \mathfrak{s}$ can be written as $s = [x, y]$ for some $x, y \in \mathfrak{s}$. Therefore,
\[
\mathrm{Tr}(\mathrm{ad}_s) = \mathrm{Tr}(\mathrm{ad}_{[x,y]}) = \mathrm{Tr}([\mathrm{ad}_x, \mathrm{ad}_y]) = 0,
\]
where the last equality follows from the general property that the trace of any commutator is zero. On one hand, suppose that either the endomorphism $\mathrm{D}$ is non-nilpotent or $t \neq 0$. We have the adjoint actions:
\begin{align*}
\mathrm{ad}_\ell \xi &= -t \xi, \\
\mathrm{ad}_\ell x &= \mu(x)\ell + \mathrm{D}(x) + \lambda(x)\xi.
\end{align*}
By induction, we obtain for all $k \in \mathbb{N}^*$:
\begin{align*}
\mathrm{ad}_\ell^k \xi &= (-1)^k t^k \xi, \\
\mathrm{ad}_\ell^k x &=\mu\big(\mathrm{D}^{k-1}(x)\big)\ell+ \mathrm{D}^k(x) + \left(\mathop{\resizebox{1.3\width}{!}{$\sum^{k-1}$}}\limits_{j=0} (-t)^j \lambda\big(\mathrm{D}^{k-1-j}(x)\big)\right)\xi.
\end{align*}

This implies that $\ell \notin \mathfrak{s}$, since $\mathfrak{s}$ consists of elements with nilpotent adjoint action. Then, $\ell \in \mathfrak{r}$ (the radical), which shows that $\mathfrak{s} \subset \overline{\G}$. Since $\overline{\G}$ is symplectic, it admits a non-trivial Levi decomposition $\overline{\G} = \mathfrak{s} \ltimes \mathfrak{rad}(\overline{\G})$. Consequently, the full Lie algebra decomposes as
\[
\G = \mathfrak{s} \ltimes \mathfrak{r}, \quad \text{where} \quad \mathfrak{r} = \langle \ell \rangle \oplus \mathfrak{rad}(\overline{\G}) \oplus \langle \xi \rangle.
\]
Consider now the case where both $t = 0$ and $\mathrm{D}$ is a nilpotent endomorphism. Then the adjoint operator $\mathrm{ad}_\ell$ is nilpotent.

 Note that, $\G=(\langle\ell\rangle\bowtie
\overline{\G})\ltimes\langle\xi\rangle$, where $\overline{\G}$ is a symplectic Lie algebra. First, observe that since $\langle \xi \rangle$ is a one-dimensional ideal of $\G$, it follows that $\langle \xi \rangle \subset \mathfrak{r}$. On the other hand, we have
\[[\langle\ell\rangle,\mathfrak{r}]=\mu\big(\mathfrak{r}\big)\ell+\mathrm{D}\big(\mathfrak{r}\big)+\lambda\big(\mathfrak{r}\big)\xi\subset\mathfrak{r}.\]
If $\mu \not\equiv0$, then $\langle\ell\rangle\subset\mathfrak{r}$. Now, assume that $\mu \equiv 0$. Then, $\G_{(\mathrm{D},0,\lambda,0)}$ is a symplectic oxidation of $\overline{\G}$. 

Suppose that $\langle \ell \rangle \not\subseteq \mathfrak{r}$, and consider $\mathfrak{r}' = \langle \ell \rangle \oplus \mathfrak{r}$. Since $\mathfrak{r}$ is a maximal solvable ideal, $\mathfrak{r}'$ is also a solvable ideal of $\G_{(\mathrm{D},0,\lambda,0)}$, because $\langle \ell \rangle$ is a nilpotent subalgebra of $\G_{(\mathrm{D},0,\lambda,0)}$. This contradicts the maximality of $\mathfrak{r}$. Therefore, $\langle \ell \rangle \subset \mathfrak{r}$, and consequently $\mathfrak{s} \subset \overline{\G}$. Since $\overline{\G}$ is symplectic, it follows that $\overline{\G}$ has a nontrivial Levi-Malcev decomposition, i.e.,$
\overline{\G} = \mathfrak{s} \ltimes \mathfrak{rad}(\overline{\G}).
$

Conversely, suppose that $\overline{\G}$ has a nontrivial Levi-Malcev decomposition, and let $\G_{(\mathrm{D},\mu,\lambda,t)}$ be its generalized symplectic oxidation. Since $\overline{\G}$ is symplectic and $\G_{(\mathrm{D},\mu,\lambda,t)}$ is also symplectic, it follows that $\G_{(\mathrm{D},\mu,\lambda,t)}$ has a nontrivial Levi-Malcev decomposition.

\end{proof}

\begin{pr}\label{diagram}
	Let  $(\overline{\G}_1,\overline{\omega}_1)$ and $(\overline{\G}_2,\overline{\omega}_2)$ be  isomorphic symplectic Lie algebras. Then their generalized symplectic oxidations   are symplectomorphically isomorphic.
	\end{pr}
	\begin{proof}
	Let $(\G_{(\mathrm{D}_1,\mu_1,\lambda_1,t_1)}, \om_1)$ and $(\G_{(\mathrm{D}_2,\mu_2,\lambda_2,t_2)}, \om_2)$ be two $2n$-dimensional generalized symplectic oxidations of $(\overline{\G}_1,\overline{\omega}_1)$ and $(\overline{\G}_2,\overline{\omega}_2)$, respectively. Suppose $\Phi: (\overline{\G}_1,\overline{\omega}_1) \longrightarrow (\overline{\G}_2,\overline{\omega}_2)$ is a symplectomorphism of Lie algebras. On the one hand, for every pair of $1$-forms $\lambda_2, \mu_2 \in \overline{\G}_2^*$, there exist $\lambda_1, \mu_1 \in \overline{\G}_1^*$ such that:
\[
(\Phi^*)^{-1} \lambda_1 = \lambda_2 \quad \text{and} \quad (\Phi^*)^{-1} \mu_1 = \mu_2.
\]

On the other hand, if $\mathrm{D}_2: \overline{\G}_2\longrightarrow \overline{\G}_2$ is an endomorphism, then there exists an endomorphism $\mathrm{D}_1 \in \mathrm{End}(\overline{\G}_1)$ such that:
\[
\mathrm{D}_2 = \Phi \circ \mathrm{D}_1 \circ \Phi^{-1}.
\]

We extend $\Phi$ to a symplectomorphism between the generalized symplectic oxidations of $(\overline{\G}_1,\overline{\omega}_1)$ and $(\overline{\G}_2,\overline{\omega}_2)$ by defining:
\[
\tilde{\Phi}|_{\overline{\G}_1} = \Phi, \quad \tilde{\Phi}(\xi_1) = \xi_2, \quad \tilde{\Phi}(\ell_1) = \ell_2.
\]

We summarize the above construction with the following diagram:
\[
\begin{tikzcd}
\big( \G_{(\mathrm{D}_1, \mu_1, \lambda_1, \xi_1,t)}, \omega_1 \big)
\arrow[r, densely dotted, "\widetilde{\Phi}"]
\arrow[d, leftrightarrow]
& \big( \G_{(\mathrm{D}_2, \mu_2, \lambda_2, \xi_2,t)}, \omega_2 \big)
\arrow[d, leftrightarrow] \\
(\overline{\G}_1, \overline{\omega}_1)
\arrow[r, "\Phi"]
& (\overline{\G}_2, \overline{\omega}_2)
\end{tikzcd}
\]
Therefore, 
\begin{align*}
(\tilde{\Phi}^\ast\omega_2)(x+\xi_1,y+\ell_1)&=\omega_2(\Phi(x)+\xi_2,\Phi(y)+\ell_2)\\
&=\omega_2(\Phi(x),\Phi(y))+\omega_2(\xi_2,\ell_2)\\
&=\overline{\omega}_1(x,y)+\eta(\xi_2^\ast\wedge\ell_2^\ast)(\xi_2,\ell_2)\\
&=\overline{\omega}_1(x,y)+\eta((\tilde{\Phi}^\ast)^{-1}(\xi_1^\ast)\wedge(\tilde{\Phi}^\ast)^{-1}(\ell_1^\ast))(\tilde{\Phi}(\xi_1),\tilde{\Phi}(\ell_1))\\
&=\overline{\omega}_1(x,y)+\eta (\xi_1^\ast\wedge\ell_1^\ast)(\xi_1,\ell_1)\\
&=\omega_1(x+\xi_1,y+\ell_1).
\end{align*}

Using $(\ref{Brackestgenera})$, we have:
\begin{align*}
\tilde{\Phi}\big([x,y]_1\big)&=\Phi\big(\overline{[x,y]}_1\big)+\overline{\omega}_1(x,y)\xi_2,\\
[\tilde{\Phi}(x),\tilde{\Phi}(y)]_2&=\overline{[\Phi(x),\Phi(y)]}_2+\overline{\omega}_2(\Phi(x),\Phi(y))\xi_2,
\end{align*}
moreover,
\begin{align*}
\tilde{\Phi}\big([\xi_1,x]_1\big)&=-\mu_1(x)\xi_2=-(\Phi^\ast\mu_2)(x)\xi_2=-\mu_2(\Phi(x))\xi_2=[\tilde{\Phi}(\xi_1),\tilde{\Phi}(x)]_2,
\end{align*}
and
\begin{align*}
\tilde{\Phi}\big([\xi_1,\ell_1]_1\big)&=t\xi_2=[\xi_2,\ell_2]_2=[\tilde{\Phi}(\xi_1),\tilde{\Phi}(\ell_1)]_2.
\end{align*}
Finally,
\begin{align*}
\tilde{\Phi}\big([\ell_1,x]_1\big)&=\mu_1(x)\ell_2+\Phi\circ \mathrm{D}_1(x)+\lambda_1(x)\xi_2\\
&=(\Phi^\ast\mu_2)(x)\ell_2+(\mathrm{D}_2\circ \Phi)(x)+(\Phi^\ast\lambda_2)(x)\xi_2\\
&=\mu_2(\Phi(x))+\mathrm{D}_2(\Phi(x))+\lambda_2(\Phi(x))\xi_2\\
&=[\ell_2,\Phi(x)]_2=[\tilde{\Phi}(\ell_2),\tilde{\Phi}(x)]_2.
\end{align*}
Consequently, $\tilde{\Phi}$ is a Lie algebra isomorphism.
	\end{proof}

\section{Classification of eight-dimensional  non completely reducible symplectic Lie algebras}\label{se3}
We begin by proving the following result,  which adapts Proposition~$\ref{solvanilpoLevi}$ exclusively to eight-dimensional non completely reducible symplectic Lie algebras.

\begin{pr}\label{solvability}
Every  eight-dimensional non completely reducible symplectic Lie algebra is  solvable.
\end{pr}
\begin{proof}
Let $(\G,\omega)$ be  an eight-dimensional non completely reducible symplectic Lie algebra. If $(\G,\omega)$ is irreducible, i.e., if $\G$ does not have a non-trivial, isotropic ideal $\mathfrak{j}$, then $\G$ is $2$-step solvable. Otherwise, let $\mathfrak{j}=\langle\xi\rangle$ be an  isotropic ideal of $(\G,\omega)$. In this case, $(\G,\omega)$ is symplectically reduced to a six-dimensional irreducible symplectic Lie algebra $(\overline{\G}=\h\ltimes\mathfrak{a},\overline{\omega})$. Thus, the Lie brackets of $\G$ can  be expressed as given in $(\ref{Brackestgenera})$ under the conditions of Proposition $\ref{prgeneralized}$. On the one hand, we have $\mathfrak{a}\subset\mathfrak{nil}(\G)$. Indeed, by $(\ref{Brackestgenera})$, for all $x\in\mathfrak{a}$, $y\in\overline{\G}$,  we have
\begin{align*}
\mathrm{ad}_xy&=\overline{\mathrm{ad}}_xy+\overline{\omega}_\mathrm{D}(x,y)\xi,\\
\mathrm{ad}^2_xy&=\overline{\mathrm{ad}}^2_xy+\overline{\omega}_\mathrm{D}(x,\overline{\mathrm{ad}}_xy)\xi+\overline{\omega}_\mathrm{D}(x,y)\mu(x)\xi,\\
\mathrm{ad}^3_xy&=\mu(x)\overline{\omega}_\mathrm{D}(x,\overline{\mathrm{ad}}_xy)\xi=0,
\end{align*}
since, $\mu(\mathfrak{a})=0$, and $\overline{\mathrm{ad}}_{\mathfrak{a}}$ is a $2$-step nilpotent endomorphism. On the other hand, we have
\begin{align*}
\mathrm{ad}_x\xi&=\mu(x)\xi=0,\\
\mathrm{ad}_x\ell&=\mu(x)\ell+\mathrm{D}(x)+\lambda(x)\xi=\mathrm{D}(x)+\lambda(x)\xi,\\
\mathrm{ad}^2_x\ell&=\overline{[x,\mathrm{D}(x)]}+\overline{\omega}_\mathrm{D}(x,\mathrm{D}(x))\xi+\lambda(x)\mu(x)\xi,\\
\mathrm{ad}^3_x\ell&=\overline{\mathrm{ad}}^2_x\big(\mathrm{D}(x)\big)+\overline{\omega}_\mathrm{D}(x,\overline{\mathrm{ad}}_x\big(\mathrm{D}(x)\big))\xi+\overline{\omega}_\mathrm{D}(x,\mathrm{D}(x))\mu(x)\xi,\\
\mathrm{ad}^4_x\ell&=\overline{\omega}_\mathrm{D}(x,\overline{\mathrm{ad}}_x\big(\mathrm{D}(x)\big))\mu(x)\xi=0.
\end{align*}
Note that $\mathrm{ad}_\xi$ is a $2$-step nilpotent endomorphism, then, $\mathfrak{k}=\mathfrak{a}\oplus\langle\xi\rangle\subset
\mathfrak{nil(\G)}$. Suppose now that $\G$ has a non-triavial Levi-Malcev decomposition, i.e., $\G=\mathfrak{s}\ltimes\mathfrak{rad}(\G)$, with $\dim\mathfrak{s}=3$ and $\dim\mathfrak{rad}(\G)=5$. From the previous discussion, we conclude that $\mathfrak{rad(\G)}=\mathfrak{nil(\G)}$. Let $x=x_\mathfrak{k}+x_\mathfrak{s}\in\G$, where $x_\mathfrak{k}\in\mathfrak{k}$ and $x_\mathfrak{s}\in\mathfrak{s}$, we have
\begin{align*}
\mathrm{Tr}(\mathrm{ad}_x)&=\mathrm{Tr}(\mathrm{ad}_{x_\mathfrak{k}})+\mathrm{Tr}(\mathrm{ad}_{x_\mathfrak{s}})=0.
\end{align*}
Therefore, since $\G$ is symplectic and unimodular, it must be solvable \textsc{\cite{C}}. This contradicts the assumption that $\G$ admits a non-trivial Levi-Malcev decomposition. Consequently, $\G$ is a  solvable Lie algebra.
\end{proof}

\subsection{Central symplectic oxidation}
As mentioned before in Remark $\ref{remarkcentral}$, the central case is characterized by $\mu = 0$ and $t = 0$. The Lie brackets are given by:
\begin{align}\label{LieCentral}
\begin{split}
	[x,y]&=\ol{[x,y]}+\overline{\om}_\mathrm{D}(x,y)\xi, \hspace{0.9cm}\text{for all}\quad x,y\in \ol{\G},\\
	[\ell,x]&=\mathrm{D}(x)+\lambda(x)\xi,\hspace{1.5cm}\text{for all}~~  x\in \ol{\G},
	\end{split}
\end{align} 
where $\mathrm{D}$ is a derivation of $\overline{\G}$ and $\lambda\in\overline{\G}^*$ is a $1$-form  satisfying $\overline{\om}_{\mathrm{D},\mathrm{D}}=-\md\lambda$.

The symplectic Lie algebra $(\G, \omega) = (\overline{\G}_{(\mathrm{D},\overline{\omega}_\mathrm{D},\lambda)},\omega)$ is called central symplectic oxidation of
$(\overline{\G},\overline{\omega})$ with respect to the data $\mathrm{D}$ and $\lambda$.

Now, we define the isomorphism classes between two symplectic oxidations by considering the general case. We introduce a  bracket on the  algebra $\G_{\ell,\xi}=\langle\ell\rangle\ltimes
\overline{\G}\oplus\langle\xi\rangle$ defined as follows: 
\begin{align}\label{Bragene}
\begin{split}
	[x,y]&=\ol{[x,y]}+\varphi(x,y)\xi, \hspace{1.15cm}\text{for all}\quad x,y\in \ol{\G},\\
	[\ell,x]&=\mathrm{D}(x)+\lambda(x)\xi,\hspace{1.5cm}\text{for all}~~  x\in \ol{\G},\\
	[\ell,\xi]&=w+t\xi,\hspace{2.48cm}w\in \mathfrak{z}(\overline{\G})\cap\ker\big(\iota_x\varphi\big),~~t\in\R,
	\end{split}
\end{align} 
where, $\mathrm{D}\in\mathrm{End}(\overline{\G})$, $\varphi\in Z^2(\overline{\G})$,  $\iota_x\varphi:\overline{\G}\longrightarrow\R$, is given by $\iota_x\varphi(y)=\varphi(x,y)$ for all $y\in\overline{\G}$,  $\lambda\in\overline{\G}^\ast$,  $\varphi(\cdot,\cdot)w=\Delta\mathrm{D}$ and $t\varphi-\varphi_\mathrm{D}=\md\lambda$.

Later, we will show that the bracket defined above $(\ref{Bragene})$ indeed gives a Lie algebra structure on $\G_{\ell,\xi}$. We will call the algebra $\G_{\ell,\xi}=\G_{(\mathrm{D}, \varphi, \lambda, w, t)}$ the \textit{non-central oxidation} of $\overline{\G}$ with respect to $(\mathrm{D}, \varphi, \lambda, w, t)$. By restriction to the central case ($w = 0$ and $t = 0$), we obtain isomorphism classes between two central symplectic oxidations. We begin with the following characterization:

\begin{theo}\label{isombetewnoncentral}
Two  non-central  oxidations $\G_{(\mathrm{D}_1, \varphi, \lambda_1, w_1, t_1)}$ and $\G_{(\mathrm{D}_2, \varphi, \lambda_2, w_2, t_2)}$    of $\overline{\G}$,  are  isomorphic if and only if there exist an  automorphism of central extension  $\Psi:\overline{\G}\oplus\langle\xi\rangle
\longrightarrow\overline{\G}\oplus\langle\xi\rangle$, $u\in\overline{\G}$, $v \in \mathfrak{z}(\overline{\G})\cap\ker\big(\iota_{\Psi|_{\overline{\G}}(\overline{\G})}\varphi\big)$,  $\beta\in\overline{\G}^\ast$, and $\theta,a\in\R^\ast$ such that
\begin{enumerate}
\item $\mathrm{D}_2=\frac{1}{\theta}\Big(\big(\Psi|_{\overline{\G}}\circ \mathrm{D}_1+v\lambda_1-w_2\theta\beta\big)\circ\Psi|_{\overline{\G}}^{-1}-\mathrm{ad}_u\big)$,
\item $\Psi^\ast|_{\overline{\G}}\Big(\lambda_2+\frac{1}{\theta}\iota_u\varphi\Big)=\frac{1}{\theta}\Big(\beta\circ \mathrm{D}_1+a\lambda_1-t_2\theta\beta\Big)$,
\item$\Psi|_{\overline{\G}}(w_1)= \theta \mathrm{D}_2(v)+a\theta w_2-t_1v$,
\item $\varphi(u,v)=\beta(w_1)+at_1-\theta\lambda_2(v)-a\theta t_2$.
\end{enumerate}
With, $a\varphi-\Psi|_{\overline{\G}}^\ast\big(\varphi\big)=\md\beta$.
\end{theo}
\begin{proof}
Let $\overline{\G}$ be a Lie algebra. We Extend the Lie algebra $\overline{\G}$ by a derivation. For this, consider the vector space defined by
\begin{align*}
\G_\ell=\langle\ell\rangle\ltimes
\overline{\G}.
\end{align*}
Let $\mathrm{D}\in\mathrm{Der}(\overline{\G})$ be a derivation on $\overline{\G}$. Define a Lie algebra structure on $\G_{\ell}$ as the semidirect sum,
\begin{align}
[x,y]&=\overline{[x,y]},\hspace{2.05cm} \text{for all}\quad x,y\in\overline{\G},\\
[\ell,x]&=\mathrm{D}(x),\hspace{2.1cm} \text{for all}\quad x\in\overline{\G}.
\end{align}
Suppose that $\G_{\ell_1}$ and $\G_{\ell_2}$ are isomorphic. Then, there exists a Lie algebra isomorphism $\Phi:\G_{\ell_1}\longrightarrow\G_{\ell_2}$. Since $\overline{\G}$ is a  maximal ideal of $\G_{\ell}$, then $\Phi(\overline{\G})=\overline{\G}$,  and we have
\begin{align*}
\Phi([x,y]_1)&=
\Phi\big(\overline{[x,y]}\big)=\Phi|_{\overline{\G}}\Big(\overline{[x,y]}\Big)\\
[\Phi(x),\Phi(y)]_2&=[\Phi|_{\overline{\G}}(x),\Phi|_{\overline{\G}}(y)]_2=
\overline{[\Phi|_{\overline{\G}}(x),\Phi|_{\overline{\G}}(y)]}.
\end{align*} 
This shows that $\Phi|_{\overline{\G}}\in\mathrm{Aut}(\overline{\G})$. On the other hand,
\begin{align*}
\Phi([\ell_1,x]_1&=\Phi(\mathrm{D}_1(x))=
\Phi|_{\overline{\G}}(\mathrm{D}_1(x))\\
[\Phi(\ell_1),\Phi(x)]_2)&=[u+\theta \ell_2,\Phi|_{\overline{\G}}(x)]_2=\overline{[u,\Phi|_{\overline{\G}}(x)]}+\theta \mathrm{D}_2\big(\Phi|_{\overline{\G}}(x)\big).
\end{align*} 
Therefore,
\[\mathrm{D}_2=\frac{1}{\theta}\big(\Phi|_{\overline{\G}}\circ \mathrm{D}_1\circ\Phi|_{\overline{\G}}^{-1}-\mathrm{ad}_u\big),\]
for some $u\in\overline{\G}$, and $\theta\in\R^\ast$.

Let $\langle\xi\rangle$ be a one-dimensional trivial Lie algebra, and $\varphi\in\Lambda^2\overline{\G}$  a $2$-cocycle. Define a Lie brackets on the direct sum $\G_{\xi}=\overline{\G}\oplus\langle\xi\rangle$ by
\begin{align}
[x,y]&=\overline{[x,y]}+\varphi(x,y)\xi,\hspace{2cm}\text{for all}~~x,y\in\overline{\G},\\
[x,\xi]&=0,\hspace{4cm}\text{for all}~~x\in\overline{\G}.
\end{align}
The vector space  $\G_{\xi}$ is called the central
extension of $\overline{\G}$ by the $2$-cocycle $\varphi$. Let $\mathrm{D}_\xi\in\mathrm{Der}(\G_{\xi})$ be a derivation on $\G_{\xi}$. Then, for all $x,y\in\overline{\G}$, we have
\begin{align*}
\mathrm{D}_\xi([x,y])&=
\mathrm{D}_\xi(\overline{[x,y]})+\varphi(x,y) \mathrm{D}_\xi\xi\\
&=\mathrm{D}_\xi|_{\overline{\G}}(\overline{[x,y]})+\lambda(\overline{[x,y]})\xi+\varphi(x,y)w+t\varphi(x,y)\xi,
\end{align*}
and
\begin{align*}
[\mathrm{D}_\xi(x),y]&=[\mathrm{D}_\xi|_{\overline{\G}}(x)+\lambda(x)\xi,y]\\
&=\overline{[\mathrm{D}_\xi|_{\overline{\G}}(x),y]}+\varphi\big(\mathrm{D}_\xi|_{\overline{\G}}(x),y\big)\xi,\\
[x,\mathrm{D}_\xi(y)]&=\overline{[x,\mathrm{D}_\xi|_{\overline{\G}}(y)]}+\varphi\big(x,\mathrm{D}_\xi|_{\overline{\G}}(y)\big)\xi,
\end{align*}
moreover,
\begin{align*}
0&=\mathrm{D}_\xi([x,\xi])=[\mathrm{D}_\xi(x),\xi]+[x,\mathrm{D}_\xi(\xi)]\\
&=[x,w+t\xi]\\
&=\overline{[x,w]}+\varphi(x,w)\xi.
\end{align*}
Therefore, every derivation $\mathrm{D}_\xi$ on $\G_\xi$ has the following form
\begin{align}\label{Derxi}
\begin{split}
\mathrm{D}_\xi(x)&=\mathrm{D}(x)+\lambda(x)\xi,\quad\text{for all}~~x\in\overline{\G},\\
\mathrm{D}_\xi(\xi)&=w+t\xi,\qquad t\in\R,w\in \mathfrak{z}(\overline{\G})\cap\ker(\iota_x\varphi),
\end{split}
\end{align}
where, $\mathrm{D}\in\mathrm{End}(\overline{\G})$, $\iota_x\varphi:\overline{\G}\longrightarrow\R$, $\iota_x\varphi(y)=\varphi(x,y)$ for all $y\in\overline{\G}$,  $\lambda\in\overline{\G}^\ast$,  $\varphi(\cdot,\cdot)w=\delta(\mathrm{D})$ and $t\varphi-\varphi_\mathrm{D}=\md\lambda$. We extend, $\G_\xi$ by the derivation $\mathrm{D}_\xi$ given in $(\ref{Derxi})$ as described previously. On the one hand, suppose that there exists a Lie algebra isomorphism $\Psi:\G_{\ell_1,\xi}=\langle\ell_1\rangle\ltimes
\G_{\xi}\longrightarrow\G_{\ell_2,\xi}=\langle\ell_2\rangle\ltimes
\G_{\xi}$. From the previous discussion, the restriction $\Psi|_{\G_\xi}$ lies in $\mathrm{Aut}(\G_{\xi})$, and the restriction $\Psi|_{\langle \ell \rangle}$ acts as
\[
\Psi|_{\langle \ell_1 \rangle}(\ell_1) = u + \theta \ell_2,
\]
where $u \in \G_{\xi}$ and $\theta \in \R^\ast$. Now let $\Upsilon$ be an automorphism of $\G_\xi$, then $\Upsilon$ can be expressed as
\begin{align*}
\Upsilon(x)&=\Upsilon|_{\overline{\G}}(x)+\beta(x)\xi,\hspace{0.5cm}~~\text{for all}~~x\in\overline{\G},~~\beta\in
\overline{\G},\\
\Upsilon(\xi)&=v+a\xi,\hspace{2.7cm}~~v\in\overline{\G},~a\in\R.
\end{align*}
Thus,
\begin{align}\label{restri1}
\begin{split}
\Upsilon\big([x,y]\big)&=\Upsilon\big(
\overline{[x,y]}\big)+\varphi(x,y)\Upsilon(\xi)\\
&=\Upsilon|_{\overline{\G}}\big(
\overline{[x,y]}\big)+\beta\big(
\overline{[x,y]}\big)\xi+\varphi(x,y)v+a\varphi(x,y)\xi,
\end{split}
\end{align}
and
\begin{align}\label{restri2}
\begin{split}
[\Upsilon(x),\Upsilon(y)]&=[\Upsilon|_{\overline{\G}}(x)+\beta(x)\xi,\Upsilon|_{\overline{\G}}(y)+\beta(y)\xi]\\
&=[\Upsilon|_{\overline{\G}}(x),\Upsilon|_{\overline{\G}}(y)]\\
&=\overline{[\Upsilon|_{\overline{\G}}(x),\Upsilon|_{\overline{\G}}(y)]}+\varphi\big(\Upsilon|_{\overline{\G}}(x),\Upsilon|_{\overline{\G}}(y)\big)\xi.
\end{split}
\end{align}
Moreover, for all $x\in\overline{\G}$, we have
\begin{align}\label{Autgxi0}
\begin{split}
\Upsilon\big([\xi,x]\big)&=0,\\
[\Upsilon(\xi),\Upsilon(x)]&=[v+a\xi,\Upsilon|_{\overline{\G}}(x)+\beta(x)\xi]=\overline{[v,\Upsilon|_{\overline{\G}}(x)]}+\varphi(v,\Upsilon|_{\overline{\G}}(x))\xi.
\end{split}
\end{align}
Therefore, $\Upsilon\in\mathrm{Aut}(\G_\xi)$ if and only if
\begin{align}\label{Autgxi1}
\Upsilon|_{\overline{\G}}\big(\overline{[x,y]}\big)+\varphi(x,y)v&=\overline{
[\Upsilon|_{\overline{\G}}(x),\Upsilon|_{\overline{\G}}(y)]},\hspace{0.5cm}\text{for all}~~ x,y\in\overline{\G},~v\in \mathfrak{z}(\overline{\G})\cap\ker\big(\iota_{\Upsilon|_{\overline{\G}}(\overline{\G})}\varphi\big),
\end{align}
and
\begin{align}\label{Autgxi2}
a\varphi-\Upsilon|_{\overline{\G}}^\ast\big(\varphi\big)=\md\beta,\hspace{0.5cm}~~a\in\R^\ast,~~\beta\in\overline{\G}^\ast.
\end{align}
Note that, $\Upsilon|_{\overline{\G}}$ is an automorphism of $\overline{\G}$ if and only if $\mathfrak{z}(\overline{\G})=\{0\}$.

Let $\mathrm{D}^{(1)}_{\xi},\mathrm{D}^{(2)}_{\xi}\in\mathrm{Der}(\G_\xi)$, then there exist, $\theta\in\R^\ast$ and $u\in\G_\xi$ such that 
\begin{align*}
\mathrm{D}^{(2)}_{\xi}=\frac{1}{\theta}\big(\Psi|_{\G_\xi}\circ \mathrm{D}^{(1)}_\xi\circ\Psi|_{\G_\xi}^{-1}-\mathrm{ad}_u\big).
\end{align*}
For all $x\in\overline{\G}$, we have
\begin{align}\label{C1}
\begin{split}
\theta \mathrm{D}^{(2)}_{\xi}\Big(\Psi|_{\G_\xi}(x)\Big)&=\theta \mathrm{D}^{(2)}_{\xi}\Big(\Psi|_{\overline{\G}}(x)+\beta(x)\xi\Big)\\
&=\theta \mathrm{D}_2\Big(\Psi|_{\overline{\G}}(x)\Big)+\theta\lambda_2\Big(\Psi|_{\overline{\G}}(x)\Big)\xi+\theta\beta(x)\big(w_2+t_2\xi\big),
\end{split}
\end{align}
and
\begin{align}\label{C2}
\Psi|_{\G_\xi}\circ \mathrm{D}^{(1)}_\xi(x)-\mathrm{ad}_u\circ\Psi|_{\G_\xi}(x)&=\Psi|_{\G_\xi}\big(\mathrm{D}_1(x)+\lambda_1(x)\xi\big)-\mathrm{ad}_u\circ\Psi|_{\overline{\G}}(x)\nonumber\\
&=\Psi|_{\overline{\G}}\big(\mathrm{D}_1(x)\big)+\beta\big(\mathrm{D}_1(x)\big)\xi+\lambda_1(x)(v+a\xi)-\mathrm{ad}_u\circ\Psi|_{\overline{\G}}(x)\\
&=\Psi|_{\overline{\G}}\big(\mathrm{D}_1(x)\big)+\beta\big(\mathrm{D}_1(x)\big)\xi+\lambda_1(x)(v+a\xi)-\overline{\mathrm{ad}}_u\circ\Psi|_{\overline{\G}}(x)-\varphi\big(u,\Psi|_{\overline{\G}}(x)\big)\xi.\nonumber
\end{align}
Therefore,
\begin{align}\label{C11}
\mathrm{D}_2&=\frac{1}{\theta}\Big(\big(\Psi|_{\overline{\G}}\circ \mathrm{D}_1+v\lambda_1-w_2\theta\beta\big)\circ\Psi|_{\overline{\G}}^{-1}-\overline{\mathrm{ad}}_u\big),
\end{align}
and
\begin{align}\label{C21}
\Psi^\ast|_{\overline{\G}}\Big(\lambda_2+\frac{1}{\theta}\iota_u\varphi\Big)&=\frac{1}{\theta}\Big(\beta\circ \mathrm{D}_1+a\lambda_1-t_2\theta\beta\Big).
\end{align}
We also have
\begin{align}\label{C3}
\begin{split}
\theta \mathrm{D}^{(2)}_\xi\Psi|_{\G_\xi}(\xi)&=\theta \mathrm{D}^{(2)}_\xi\big(v+a\xi\big)=\theta \mathrm{D}^{(2)}_\xi(v)+a\theta \mathrm{D}^{(2)}_\xi(\xi),\\
&=\theta \mathrm{D}_2(v)+\theta\lambda_2(v)\xi+a\theta(w_2+t_2\xi),
\end{split}
\end{align}
and
\begin{align}\label{C4}
\begin{split}
\Psi|_{\G_\xi}\circ \mathrm{D}^{(1)}_\xi(\xi)-\mathrm{ad}_u\circ\Psi|_{\G_\xi}(\xi)&=\Psi|_{\G_\xi}\big(w_1+t_1\xi\big)-\mathrm{ad}_u(v)-\varphi(u,v)\xi\\
&=\Psi|_{\overline{\G}}(w_1)+\beta(w_1)\xi+t_1(v+a\xi)-\varphi(u,v)\xi.
\end{split}
\end{align}
As a consequence,
\begin{align}\label{C31}
\theta \mathrm{D}_2(v)+a\theta w_2=\Psi|_{\overline{\G}}(w_1)+t_1v,
\end{align}
and
\begin{align}\label{C41}
\theta\lambda_2(v)+a\theta t_2=\beta(w_1)+t_1 a-\varphi(u,v).
\end{align}
\end{proof}
\begin{co}\label{Covarphi}
Two  non-central  oxidations, $\G_{(\mathrm{D}_1, \varphi, \lambda_1, w_1, t_1)}$ and $\G_{(\mathrm{D}_2, \varphi, \lambda_2, w_2, t_2)}$    of $\overline{\G}$, with $\mathfrak{z}(\overline{\G})=\{0\}$,  are  isomorphic if and only if there exist an automorphism  $\Phi:\overline{\G}
\longrightarrow\overline{\G}$, $u\in\overline{\G}$ and $\theta,a\in\R^\ast$ such that
\begin{enumerate}
\item $\mathrm{D}_2=\frac{1}{\theta}\Big(\big(\Phi|_{\overline{\G}}\circ \mathrm{D}_1\big)\circ\Phi|_{\overline{\G}}^{-1}-\mathrm{ad}_u\big)$,
\item $\Phi^\ast|_{\overline{\G}}\Big(\lambda_2+\frac{1}{\theta}\iota_u\varphi\Big)=\frac{1}{\theta}\Big(\beta\circ \mathrm{D}_1+a\lambda_1-t_2\theta\beta\Big)$,
\item $t_1=\theta t_2$.
\end{enumerate}
\end{co}
\begin{proof}
Suppose that $\mathfrak{z}(\overline{\G})=\{0\}$. In this case, and from the bracket defined in $(\ref{Bragene})$, we obtain $w_1=w_2=0$. On the one hand, suppose that there exists a Lie algebra isomorphism $\Psi:\G_{(\mathrm{D}_1, \varphi, \lambda_1, 0, t_1)}\longrightarrow\G_{(\mathrm{D}_2, \varphi, \lambda_2, 0, t_2)}$.   According to Proposition \ref{isombetewnoncentral}, under condition $(\ref{Autgxi1})$, $\Psi$ can be expressed as follows:
\begin{align}
\begin{split}
\Psi(x)&=\Phi(x)+\beta(x)\xi,\hspace{0.98cm}~\text{for all}~~x\in\overline{\G},\\
\Psi(\xi)&=a\xi,\hspace{3.31cm}a\in\R^\ast,\\
\Psi(\ell_1)&=u+\theta\ell_2,\hspace{1.7cm}u\in\overline{\G},~\theta\in\R^\ast,
\end{split}
\end{align}
where, $\Phi\in\mathrm{Aut}(\overline{\G})$, $\beta\in\overline{\G}^\ast$ and $a\varphi-\Phi^\ast\varphi=\md\beta$. Therefore, conditions 1, 2, and 3 hold as restrictions of the general case given in Proposition \ref{isombetewnoncentral}.
\end{proof}
We now characterize the isomorphism between two non-central oxidations arising from distinct $2$-cocycles $\varphi_1, \varphi_2 \in Z^2(\overline{\G})$.

\begin{Le}\label{extegene}
Let $\G_{\ell_1}$ and $\G_{\ell_2}$ be two  extended Lie algebras of $(\overline{\G}_1,\mathrm{D}_1)$ and $(\overline{\G}_2,\mathrm{D}_2)$, respectively. Then, $\G_{\ell_1}$ and $\G_{\ell_2}$ are isomorphic if and only if there exist a Lie algebra isomorphism  $\Psi:\overline{\G}_1\longrightarrow\overline{\G}_2$, $u\in\overline{\G}_2$ and $\theta\in\R^\ast$ such that
\begin{equation*}
\mathrm{D}_2=\frac{1}{\theta}\big(\Psi\circ \mathrm{D}_1\circ\Psi^{-1}-\mathrm{ad}_u\big).
\end{equation*}

\end{Le}
\begin{proof}
Suppose that there exists a Lie algebra isomorphism $\Phi:\G_{\ell_1}\longrightarrow\G_{\ell_2}$. Since $\overline{\G}_1$ is a maximal ideal of $\G_{\ell_1}$ then $\Phi(\overline{\G}_1)$ is a maximal ideal of $\G_{\ell_2}$, this  implies that, $\Phi(\overline{\G}_1)=\overline{\G}_2$. For all $x,y\in\overline{\G}_1$, we have
\begin{align*}
\Phi([x,y]_1)&=\Phi(\overline{[x,y]}_1)=\Phi|_{\overline{\G}_1}(\overline{[x,y]}_1),\\
[\Phi|_{\overline{\G}_1}(x),\Phi|_{\overline{\G}_1}(y)]_2&=\overline{[\Phi|_{\overline{\G}_1}(x),\Phi|_{\overline{\G}_1}(y)]}_2.
\end{align*}
and
\begin{align*}
\Phi([\ell_1,x]_1)&=\Phi|_{\overline{\G}_1}(\mathrm{D}_1(x)),\\
[\Phi(\ell_1),\Phi(x)]_2&=[u+\theta \ell_2,\Phi|_{\overline{\G}_1}(x)]_2\\
&=\overline{[u,\Phi_{\overline{\G}_1}(x)]}_2+\theta \mathrm{D}_2\big(\Phi|_{\overline{\G}_1}(x)\big).
\end{align*}
Therefore, $\Psi:=\Phi_{\overline{\G}_1}$ is a Lie algebra isomorphism, and
\begin{align*}
\mathrm{D}_2=\frac{1}{\theta}\big(\Psi\circ \mathrm{D}_1\circ\Psi^{-1}-\mathrm{ad}_u\big).
\end{align*}
\end{proof}

\begin{theo}\label{Theogeneral}
Two  non-central  oxidations $\G_{(\mathrm{D}_1, \varphi_1, \lambda_1, w_1, t_1)}$ and $\G_{(\mathrm{D}_2, \varphi_2, \lambda_2, w_2, t_2)}$    of $\overline{\G}$,  are  isomorphic if and only if there exist a  Lie algebra isomorphism of central extensions  $\Psi:\overline{\G}_{\varphi_1}
\longrightarrow\overline{\G}_{\varphi_2}$, $u\in\overline{\G}$, $v\in \mathfrak{z}(\overline{\G})\cap\ker\big(\iota_{\Psi|_{\overline{\G}}(\overline{\G})}\varphi_2\big)$, $\beta\in\overline{\G}^\ast$, and $\theta,a\in\R^\ast$ such that
\begin{enumerate}
\item $\mathrm{D}_2=\frac{1}{\theta}\Big(\big(\Psi|_{\overline{\G}}\circ \mathrm{D}_1+v\lambda_1-w_2\theta\beta\big)\circ\Psi|_{\overline{\G}}^{-1}-\mathrm{ad}_u\big)$,
\item $\Psi^\ast|_{\overline{\G}}\Big(\lambda_2+\frac{1}{\theta}\iota_u\varphi_2\Big)=\frac{1}{\theta}\Big(\beta\circ \mathrm{D}_1+a\lambda_1-t_2\theta\beta\Big)$,
\item$\Psi|_{\overline{\G}}(w_1)= \theta \mathrm{D}_2(v)+a\theta w_2-t_1v$,
\item $\varphi_2(u,v)=\beta(w_1)+at_1-\theta\lambda_2(v)-a\theta t_2$.
\end{enumerate}
With, $a\varphi_1-\Psi|_{\overline{\G}}^\ast\big(\varphi_2\big)=\md\beta$.
\end{theo}
\begin{proof}
Similar to the proof of Theorem $\ref{isombetewnoncentral}$. Let $\Phi=\G_{(\mathrm{D}_1, \varphi_1, \lambda_1, w_1, t_1)}\longrightarrow\G_{(\mathrm{D}_2, \varphi_2, \lambda_2, w_2, t_2)}$ be a Lie algebra isomorphism. According to Lemma $\ref{extegene}$, $\Psi=\Phi|_{\overline{\G}_{\varphi_1}}:\overline{\G}_{\varphi_1}\longrightarrow\overline{\G}_{\varphi_2}$ is  a Lie algebra isomorphism   which can be written as:
\begin{align*}
\Psi(x)&=\Psi|_{\overline{\G}}(x)+\beta(x)\xi,\hspace{2.1cm}\text{for all }x\in\overline{\G},~\beta\in\overline{\G}^\ast,\\
\Psi(\xi)&=v+a\xi,\hspace{4.1cm}v\in\overline{\G},~a\in\R^\ast.
\end{align*}
Similar to the steps  $(\ref{restri1})-(\ref{Autgxi0})$, and their consequences $(\ref{Autgxi1})$ and $(\ref{Autgxi2})$,  we obtain that $\Psi\in\mathrm{Isom}(\G_{\varphi_1}\to\G_{\varphi_2})$ if and only if
\begin{align}\label{GAutgxi1}
\Psi|_{\overline{\G}}\big(\overline{[x,y]}\big)+\varphi_1(x,y)v&=\overline{
[\Psi|_{\overline{\G}}(x),\Psi|_{\overline{\G}}(y)]},\hspace{0.5cm}\text{for all}~~ x,y\in\overline{\G},~v\in \mathfrak{z}(\overline{\G})\cap\ker\big(\iota_{\Psi|_{\overline{\G}}(\overline{\G})}\varphi_2\big),
\end{align}
and
\begin{align}\label{GAutgxi2}
a\varphi_1-\Psi|_{\overline{\G}}^\ast\big(\varphi_2\big)=\md\beta,\hspace{0.5cm}~~a\in\R^\ast,~~\beta\in\overline{\G}^\ast.
\end{align}
In the same way, using the steps $(\ref{C1})$, $(\ref{C2})$ and their consequences $(\ref{C11})$ and $(\ref{C21})$, we find the following conditions:
 \begin{align}\label{GC11}
\mathrm{D}_2&=\frac{1}{\theta}\Big(\big(\Psi|_{\overline{\G}}\circ \mathrm{D}_1+v\lambda_1-w_2\theta\beta\big)\circ\Psi|_{\overline{\G}}^{-1}-\overline{\mathrm{ad}}_u\big),
\end{align}
and
\begin{align}\label{GC21}
\Psi^\ast|_{\overline{\G}}\Big(\lambda_2+\frac{1}{\theta}\iota_u\varphi_2\Big)&=\frac{1}{\theta}\Big(\beta\circ \mathrm{D}_1+a\lambda_1-t_2\theta\beta\Big).
\end{align}
Moreover, in a way analogous to the  steps $(\ref{C3})$, $(\ref{C4})$ and their consequences $(\ref{C31})$ and $(\ref{C41})$, we therefore have
\begin{align}\label{GC31}
\theta \mathrm{D}_2(v)+a\theta w_2=\Psi|_{\overline{\G}}(w_1)+t_1v,
\end{align}
and
\begin{align}\label{GC41}
\theta\lambda_2(v)+a\theta t_2=\beta(w_1)+t_1 a-\varphi_2(x,v).
\end{align}
\end{proof}
Similarly, we obtain the following refinement of Corollary $\ref{Covarphi}$.
\begin{co}\label{Cgeral}
Two  non-central  oxidations, $\G_{(\mathrm{D}_1, \varphi_1, \lambda_1, w_1, t_1)}$ and $\G_{(\mathrm{D}_2, \varphi_2, \lambda_2, w_2, t_2)}$    of $\overline{\G}$, with $\mathfrak{z}(\overline{\G})=\{0\}$,  are  isomorphic if and only if there exist an automorphism  $\Phi:\overline{\G}
\longrightarrow\overline{\G}$, $u\in\overline{\G}$ and $\theta,a\in\R^\ast$ such that
\begin{enumerate}
\item $\Phi^\ast\Big(\lambda_2+\frac{1}{\theta}\iota_u\varphi_2\Big)=\frac{1}{\theta}\Big(\beta\circ \mathrm{D}_1+a\lambda_1-t_2\theta\beta\Big)$,
\item $\mathrm{D}_2=\frac{1}{\theta}\big(\Phi\circ \mathrm{D}_1\circ\Phi^{-1}-\mathrm{ad}_u\big)$,
\item $t_1=\theta t_2$.
\end{enumerate}
With, $a\varphi_1-\Phi^\ast\big(\varphi_2\big)=\md\beta$.
\end{co}

\begin{pr}\label{CharaCentral}
Two  central symplectic oxidations $(\G_{(\mathrm{D}_1,0,\lambda_1,0)},\om_1)$ and $(\G_{(\mathrm{D}_2,0,\lambda_2,0)},\om_2)$    of $(\overline{\G},\overline{\omega}_1)$ and $(\overline{\G},\overline{\omega}_2)$, respectively,  are symplectomorphically isomorphic if and only if there exist an automorphism $\Phi:\overline{\G}\longrightarrow\overline{\G}$, $u\in\overline{\G}$ and $\theta\in\R^\ast$, $\beta\in\overline{\G}^\ast$ such that
\begin{enumerate}
\item $\Phi^\ast\overline{\omega}_2=\overline{\omega}_1$
\item $\beta=\frac{1}{\theta}\Phi^\ast\big(\iota_u\overline{\omega}_2\big)$
\item $\lambda_2=\frac{1}{\theta}\big(\Phi^\ast\big)^{-1}\Big(\beta\circ \mathrm{D}_1-\iota_u\overline{\omega}_{_2 \mathrm{D}_2}+\frac{1}{\theta}\lambda_1\Big)$,
\item $\mathrm{D}_2=\frac{1}{\theta}\big(\Phi\circ \mathrm{D}_1\circ\Phi^{-1}-\mathrm{ad}_u\big)$.
\end{enumerate}

\end{pr}
\begin{proof}
Let $\Psi\in\mathrm{Isom}(\G_{(\mathrm{D}_1,0,\lambda_1,0)}\to\G_{(\mathrm{D}_2,0,\lambda_2,0)})$ be a Lie algebra isomorphism. According to Theorem $\ref{Theogeneral}$ and Corollary $\ref{Cgeral}$, $\Psi$ is given by
\begin{align*}
\Psi(x)&=\Phi(x)+\beta(x)\xi,\hspace{1cm}\text{for all }x\in\overline{\G},~\beta\in\overline{\G}^\ast,\\
\Psi(\xi)&=a\xi,\hspace{4.25cm}a\in\R^\ast,\\
\Psi(\ell_1)&=u+\theta \ell_2,\hspace{2.64cm}u\in\overline{\G},~\theta\in\R^\ast,
\end{align*} 
where, $\Phi\in\mathrm{Aut}(\overline{\G})$, and $a\overline{\omega}_{\mathrm{D}_1}-\Phi^\ast(\overline{\omega}_{\mathrm{D}_2})=\md\beta$. For all $x,y\in\overline{\G}$, we have
\begin{align*}
\big(\Psi^\ast\omega_2\big)(x,y)&=\omega_2\big(\Psi(x),\Psi(y)\big)\\
&=\omega_2\big(\Phi(x)+\beta(x)\xi,\Phi(y)+\beta(y)\xi\big)\\
&=\overline{\omega}_2(\Phi(x),\Phi(y)),
\end{align*}
and
\begin{align*}
\omega_1(x,y)&=\overline{\omega}_1(x,y).
\end{align*}
Therefore, $\Phi^\ast\overline{\omega}_2=\overline{\omega}_1$.
On the other hand, for all $x\in\overline{\G}$, we have
\begin{align*}
\big(\Psi^\ast\omega_2\big)(x,\xi)&=\omega_1(x,\xi)=0,
\end{align*}
and
\begin{align*}
\big(\Psi^\ast\omega_2\big)(x,\ell_1)&=\omega_2\big(\Phi(x)+\beta(x)\xi,u+\theta\ell_2\big)\\
&=\omega_2\big(\Phi(x),u\big)+\theta\beta(x)\omega_2\big(\xi,\ell_2\big)\\
&=\overline{\omega}_2\big(\Phi(x),u\big)+\theta\beta(x).
\end{align*}
Then, $\beta=\frac{1}{\theta}\Phi^\ast\big(\iota_u\overline{\omega}_2\big)$, since, $\omega_1(x,\ell_1)=0$.  Moreover,

\begin{align*}
\big(\Psi^\ast\omega_2\big)(\xi,\ell_1)&=\omega_2\big(a\xi,u+\theta\ell_2\big)\\
&=a\theta\omega_2(\xi,\ell_2)=a\theta.
\end{align*}
Thus, $a=\frac{1}{\theta}$, since $\omega_1(\xi,\ell_1)=1$.

The remaining conditions 3.~and 4.,  are a consequence of Theorem $\ref{Theogeneral}$ and Corollary $\ref{Cgeral}$.
\end{proof}

In the following, we describe the scheme  of the classification of eight-dimensional non completely reducible symplectic Lie algebra  with  one-dimensional central ideal.

Let $(\G,\omega)$ be an eight-dimensional non completely reducible symplectic Lie algebra,  $\mathfrak{j}=\langle e_7\rangle$ a one-dimensional central ideal and $(\overline{\G},\overline{\omega})$ the symplectic reduction with respect to $\mathfrak{j}$, which is irreducible. The Lie brackets of $\G$ can be written as in~\eqref{Brackestgenera} under the conditions of Proposition~\ref{prgeneralized}. The central case corresponds to $t = 0$ and $\mu = 0$, so the Lie brackets of $\G$ reduce to those in~\eqref{LieCentral} under the condition given below.

 According to Proposition $\ref{diagram}$ and the fact that there exists a unique irreducible symplectic Lie algebra in dimension six, up to isomorphism, one can take  $(\overline{\G},\overline{\omega})=(\G_{(1,0,0,1)},\overline{\omega})$, where $(\G_{(1,0,0,1)},\overline{\omega})$ is the six-dimensional irreducible Lie given by Proposition $\ref{Classi6}$.  According to Proposition~\ref{CharaCentral}, the central oxidations $(\G_{(\mathrm{D},0,\lambda,0)}, \omega_1)$ and $(\G_{(\mathrm{D},0,\lambda,0)}, \omega_2)$ are actually symplectomorphic, since both are symplectomorphic to the central oxidation $(\G_{(D,0,\lambda,0)}, \omega_\eta + e^7 \wedge e^8)$. This follows from the fact that any symplectic form on $\G_{(1,0,0,1)}$ is symplectomorphic to the unique symplectic form $\omega_\eta$ on $\G_{(1,0,0,1)}$ (see Proposition~\ref{Classi6}). Moreover,  the Lie bracket structure of $\G$ is given by
 \begin{align}
\begin{split}
	[x,y]&=\ol{[x,y]}+\overline{\om}_{\eta,\mathrm{D}}(x,y)e_7, \hspace{0.9cm}\text{for all}\quad x,y\in \ol{\G},\\
	[e_8,x]&=\mathrm{D}(x)+\lambda(x)e_7,\hspace{1.7cm}\text{for all}\hspace{0.65cm}x\in\ol{\G}.
	\end{split}
\end{align} 
where $\mathrm{D}$ is a derivation of $\overline{\G}$ and $\lambda\in\overline{\G}^*$ is a $1$-form  satisfying $\overline{\om}_{\eta,\mathrm{D},\mathrm{D}}=-\md\lambda$.

We have the following:

\begin{Le}\label{Derivation}
Every derivation on $\G_{(1,0,0,1)}$ has the following form$:$
\begin{align*}
\mathrm{D}e_1&=d_{11}e_1-d_{12}e_2,~~\mathrm{D}e_2=d_{12}e_1+d_{11}e_2,~~\mathrm{D}e_3=d_{33}e_3-d_{34}e_4,~~\\
\mathrm{D}e_4&=d_{34}e_3+d_{33}e_4,~~\mathrm{D}e_5=d_{15}e_1+d_{25}e_2,~~\mathrm{D}e_6=d_{36}e_3+d_{46}e_4,
\end{align*}
where, $d_{ij}\in\R$. Moreover, any derivation  $\mathrm{D}$   satisfying $\omega_{\eta,\mathrm{D},\mathrm{D}}=-\md\lambda$ is an inner derivation on $\G_{(1,0,0,1)}$.
\end{Le}
\begin{proof}
Let $\{e_1,\ldots,e_6\}$ be a basis $\G_{(1,0,0,1)}$. Putting $\{e^1,\ldots,e^6\}$, for the dual basis. On the one hand, it is straightforward to verify that any derivation $\mathrm{D}$ on $\G_{(1,0,0,1)}$ has the following form
\begin{align}\label{Deri}
\begin{split}
\mathrm{D}e_1&=d_{11}e_1-d_{12}e_2,~~\mathrm{D}e_2=d_{12}e_1+d_{11}e_2,~~\mathrm{D}e_3=d_{33}e_3-d_{34}e_4,~~\\
\mathrm{D}e_4&=d_{34}e_3+d_{33}e_4,~~\mathrm{D}e_5=d_{15}e_1+d_{25}e_2,~~\mathrm{D}e_6=d_{36}e_3+d_{46}e_4,
\end{split}
\end{align}
where, $d_{ij}\in\R$. On the other hand, let $\lambda\in\G_{(1,0,0,1)}^\ast$, i.e., $\lambda=\lambda_1e^1+\lambda_2e^2+\lambda_3e^3+\lambda_4e^4+\lambda_5e^5+\lambda_6e^6$. The fact that, $\omega_{\eta,\mathrm{D},\mathrm{D}}=-\md\lambda$, then the derivation $\mathrm{D}$ given in $(\ref{Deri})$ becomes
\begin{align*}
\mathrm{D} e_1 &= -d_{12} e_2,       & \mathrm{D} e_2 &= d_{12} e_1,        & \mathrm{D} e_3 &= -d_{34} e_4, \\
\mathrm{D} e_4 &= d_{34} e_3,        & \mathrm{D} e_5 &= d_{15} e_1 + d_{25} e_2, & \mathrm{D} e_6 &= d_{36} e_3 + d_{46} e_4,
\end{align*}
and
\begin{align*}
\lambda= - d_{12}d_{25}e^1+d_{12}d_{15}e^2 - d_{34}d_{46}e^3 + d_{34}d_{36}e^4+\lambda_5e^5+\lambda_6 e^6,
\end{align*}
where, $d_{ij},\lambda_5,\lambda_6\in\R$. Now, let $x_0\in\G_{(1,0,0,1)}$, where $x_0=d_{25}e_1-d_{15}e_2+d_{46}e_3-d_{36}e_4+ d_{12}e_5+d_{34}e_6$, then one can check that  $\mathrm{ad}_{x_0}=\mathrm{D}$. This shows that $\mathrm{D}$ is an inner derivation on $\G_{(1,0,0,1)}$.
\end{proof}
The next step is to introduce the construction scheme for the isomorphism group 
\[
\mathrm{Isom}\big(\overline{\G}_{(\mathrm{D}_1,\overline{\omega}_{\eta,\mathrm{D}_1},\lambda_1)}\longrightarrow\overline{\G}_{(\mathrm{D}_2,\overline{\omega}_{\eta,\mathrm{D}_2},\lambda_2)}\big)
\]
of central oxidations. It follows from Theorem~$\ref{Theogeneral}$,~$\ref{isombetewnoncentral}$ that any isomorphism $\Psi$  of central oxidations  has  the  following  form
\begin{align*}
\Psi(x)&=\Psi|_{\overline{\G}}(x)+\beta(x)\xi,\hspace{2.1cm}\text{for all }x\in\overline{\G},~\beta\in\overline{\G}^\ast,\\
\Psi(\xi)&=v+a\xi,\hspace{4.1cm}v\in \mathfrak{z}(\overline{\G})\cap\ker\big(\iota_{\Psi|_{\overline{\G}}(\overline{\G})}\omega_{\eta,\mathrm{D}}\big),~a\in\R^\ast,\\
\Psi(\ell_1)&=u+\theta\ell_2,\hspace{3.98cm}u\in\overline{\G},~\theta\in\R^\ast.
\end{align*}
With, $a\omega_{\eta,\mathrm{D}}-\Phi^\ast\big(\omega_{\eta,\mathrm{D}}\big)=\md\beta$. In the view of Theorem~$\ref{Theogeneral}$, Corollary~$\ref{Cgeral}$, Proposition~$\ref{CharaCentral}$,  and the fact that $Z(\G_{1,0,0,1})=\{0\}$, we therefore have
\begin{align*}
\Psi(x)&=\Phi(x)+\beta(x)\xi,\hspace{2.1cm}\text{for all }x\in\overline{\G},~\beta\in\overline{\G}^\ast,\\
\Psi(\xi)&=a\xi,\hspace{4.4cm}~a\in\R^\ast,\\
\Psi(\ell_1)&=u+\theta\ell_2,\hspace{3.85cm}u\in\overline{\G},~\theta\in\R^\ast.
\end{align*}
Where, $\Phi\in\mathrm{Aut}(\G_{(1,0,0,1)})$.

Let $\Phi\in\mathrm{Aut}(\G_{(1,0,0,1)})$ be an automorphism of $\G_{(1,0,0,1)}=(\mathfrak{a}_1\oplus\mathfrak{a}_2)\rtimes\h$, where $\mathfrak{a}_1$ and $\mathfrak{a}_2$ are minimal ideals of $\G_{(1,0,0,1)}$  of dimension $2$ and $\h$ is a $2$-dimensional subalgebra of $\G_{(1,0,0,1)}$. Then, $\Phi(\mathfrak{a}_j)$ is also a minimal ideal of $\G_{(1,0,0,1)}$, thus $\Phi(\mathfrak{a}_1)=\mathfrak{a}_1$ and $\Phi(\mathfrak{a}_2)=\mathfrak{a}_2$ or $\Phi(\mathfrak{a}_1)=\mathfrak{a}_2$ and $\Phi(\mathfrak{a}_2)=\mathfrak{a}_1$. Therefore, by the definition of Lie algebra automorphism, the following is easily verified:
\begin{Le}\label{Autg1001}
The automorphism group $\mathrm{Aut}(\G_{(1,0,0,1)})$ is the set
\begin{small}
\begin{equation}
\left\lbrace
		\Phi_{_1(\varepsilon_1,\varepsilon_2)}=\begin{pmatrix}
			0&0&\varepsilon_1 x_{{24}}&0&0&x_{{16}}
			\\ 0&0&0&x_{{24}}&0&x_{{26}}\\ 
			\varepsilon_2 x_{{42}}&0&0&0&x_{{35}}&0\\ 0&x_{{42}}&0&0&x_{{45}}&0\\ 0&0&0&0&0&\varepsilon_1\\ 
			0&0&0&0&\varepsilon_2&0
		\end{pmatrix},~~
		\Phi_{_2(\varepsilon_1,\varepsilon_2)}=\begin{pmatrix}
			x_{11}&x_{12}&0&0&x_{{15}}&0\\
 -\varepsilon_2x_{12}&\varepsilon_2x_{11}&0&0&x_{{25}}&0\\ 
0&0&x_{{33}}&x_{{34}}&0&x_{{36}}\\ 0&0&-\varepsilon_1x_{{34}}&\varepsilon_1x_{{33}}&0&x_{{46}}\\ 
0&0&0&0&\varepsilon_2&0\\ 
0&0&0&0&0&\varepsilon_1
		\end{pmatrix}		
		\right\rbrace,
\end{equation}
\end{small}
where,  $\Phi_{_1(\varepsilon_1,\varepsilon_2)}$ is given by the conditions $\varepsilon_1,\varepsilon_2=\pm1,~x_{{24}}x_{{42}}\neq0$, and  $\Phi_{_2(\varepsilon_1,\varepsilon_2)}$ is given by the conditions $\varepsilon_1,\varepsilon_2=\pm1,~(x_{11}^2+x_{21})(x_{33}^2+x_{34}^2)\neq0$.

\end{Le}
Based on what has preceded, it is sufficient to compute $\beta\in\G_{(1,0,0,1)}^\ast$ in order to have automorphisms of central oxidation of $\G_{(1,0,0,1)}$. We have the following:
\begin{Le}\label{Isogroup}
 The isomorphism group 
$
\mathrm{Isom}\big(\overline{\G}_{(\mathrm{D}_1,\overline{\omega}_{\eta,\mathrm{D}},\lambda_1)}\longrightarrow\overline{\G}_{(\mathrm{D}_2,\overline{\omega}_{\eta,\mathrm{D}},\lambda_2)}\big)
$
of central oxidations is given by  the following set
\begin{small}
\begin{equation*}
\mathrm{Isom}\big(\overline{\G}_{(\mathrm{D}_1,\overline{\omega}_{\eta,\mathrm{D}},\lambda_1)}\longrightarrow\overline{\G}_{(\mathrm{D}_2,\overline{\omega}_{\eta,\mathrm{D}},\lambda_2)}\big)=
\left\lbrace\left(
\begin{array}{ccccc|c|c}
&&&&&0&u_1 \\
&&&&&\cdot&\cdot \\
&&\mathrm{Aut}(\G_{(1,0,0,1)})&&&\cdot&\cdot\\
&&&&&\cdot&\cdot\\
&&&&&0&u_{6}\\\hline
\beta(e_1)&\cdot\cdot\cdot&
\cdot\cdot\cdot&\cdot\cdot\cdot&\beta(e_6)&a&u_\xi\\\hline
0&\cdot\cdot\cdot&
\cdot\cdot\cdot&\cdot\cdot\cdot&0&0&\theta
\end{array}
\right),~~a\theta\neq0\right\rbrace
\end{equation*}
\end{small}
where, $u_1,\ldots,u_6\in\G_{(1,0,0,1)}$, $u_\xi\in\R$. The mappings $\beta_1$ and $\beta_2$ are defined as follows:
\begin{align*}
\beta_1(e_1) &= -ad_{15} + d_{36}\varepsilon_2x_{4 2}, & \beta_1(e_2) &= -ad_{25} + d_{46}x_{4 2}, & \beta_1(e_3) &= -ad_{36} + d_{15}\varepsilon_1x_{2 4}, \\
\beta_1(e_4)&=-ad_{46} + d_{25}x_{2 4},&\beta_1(e_5)&=x,
&\beta_1(e_6)&=y,
\end{align*}
and
\begin{align*}
\beta_2(e_1) &=(x_{1 1} - a)d_{15} - \varepsilon_2x_{1 2}d_{25}, & \beta_2(e_2) &= ((\varepsilon_2x_{1 1} - a)d_{25} + d_{15}x_{1 2}), & \beta_2(e_3) &= (x_{3 3} - a)d_{36} - \varepsilon_1x_{3 4}d_{46}, \\
\beta_2(e_4)&=(\varepsilon_1x_{3 3} - a)d_{46} +x_{34}d_{36},&\beta_2(e_5)&=x,&\beta_2(e_6)&=y.
\end{align*}
Here, \(\beta_1\) and \(\beta_2\) correspond to the automorphisms \(\Phi_{1(\varepsilon_1, \varepsilon_2)}\) and \(\Phi_{2(\varepsilon_1, \varepsilon_2)}\), respectively, via the condition
\[
a \omega_{\eta, \mathrm{D}} - \Phi^* \omega_{\eta, \mathrm{D}} = d\beta_j.
\] The terms $d_{ij}$ are the   derivation coefficients given in Lemma~$\ref{Derivation}$,  $x_{ij}$ are the automorphism coefficients given in Lemma $\ref{Autg1001}$, $\varepsilon_1,\varepsilon_2=\pm1$, and $x,y\in\R$.
\end{Le}

We now provide the class of eight-dimensional non completely reducible symplectic Lie algebras  that admits a one-dimensional central ideal.
\begin{pr}\label{ClassiCentral}
Let $(\G,\omega)$ be an eight-dimensional non completely reducible symplectic Lie algebra  with  one-dimensional central ideal. Then, $(\G,\omega)$ is symplectomorphically isomorphic to one of the following symplectic Lie algebras$:$
\begin{align*}
		\G_{8,1}:&~~[e_1, e_5] = e_2, ~~\quad 
		[e_2, e_5] = -e_1, \quad
		[e_3, e_6] = e_4, \quad 
		[e_4, e_6] = -e_3,\quad [e_8, e_5] = e_7,\\
		&~~[e_8, e_6] =\alpha e_7,\quad \alpha\in\R^+,\\
		&\hspace{0.254cm}\omega=e^{12}+e^{34}+\eta e^{56}+e^{78},\quad\eta\in\R^{\ast+}.
	\\&\\
		\G_{8,2}:&~~[e_1, e_5] = e_2, ~~\quad 
		[e_2, e_5] = -e_1, \quad
		[e_3, e_6] = e_4, \quad 
		[e_4, e_6] = -e_3,\quad [e_8, e_6] = e_7,\\
&\hspace{0.254cm}\omega=e^{12}+e^{34}+\eta e^{56}+e^{78},\quad\eta\in\R^{\ast+}\\
		&\\
		\G_{8,3}:&~~[e_1, e_5] = e_2, ~~\quad 
		[e_2, e_5] = -e_1, \quad
		[e_3, e_6] = e_4, \quad 
		[e_4, e_6] = -e_3,\\
		&\hspace{0.254cm}\omega=e^{12}+e^{34}+\eta e^{56}+e^{78},\quad\eta\in\R^{\ast+}.
		\end{align*}
\end{pr}
\begin{proof}
Let $(\G,\omega)$ be an eight-dimensional non completely reducible symplectic Lie algebra,  $\mathfrak{j}=\langle e_7\rangle$ a one-dimensional central ideal and $(\G_{(1,0,0,1)},\omega_\eta)$ the symplectic reduction with respect to $\mathfrak{j}$.  It follows from Lemma~$\ref{Derivation}$,  that the Lie brackets of $\G$ are given by
\begin{align*}
[e_1,e_5] &= e_2 + d_{25}e_7,          & [e_2,e_5] &= -e_1 - d_{15}e_7,         & [e_3,e_6] &= e_4 + d_{46}e_7, \\
[e_4,e_6] &= -e_3 - d_{36}e_7,         & [e_8,e_1] &= -d_{12}e_2 - d_{12}d_{25}e_7, & [e_8,e_2] &= d_{12}e_1 + d_{12}d_{15}e_7, \\
[e_8,e_3] &= -d_{34}e_4 - d_{34}d_{46}e_7, & [e_8,e_4] &= d_{34}e_3 + d_{34}d_{36}e_7, & [e_8,e_5] &= d_{15}e_1 + d_{25}e_2 + \lambda_5e_6, \\
[e_8,e_6] &= d_{36}e_3 + d_{46}e_4 + \lambda_6e_7.
\end{align*}
According to Lemma~$\ref{Derivation}$, every derivation $\mathrm{D}\in\mathrm{Der}(\G_{(1,0,0,1)})$  satisfying $\omega_{\eta,\mathrm{D},\mathrm{D}}=-\md\lambda$ is an inner derivation. Thus, there exists $x_0\in\G$ such that $\mathrm{D}=\mathrm{ad}_{x_0}$, and we have
\begin{align*}
\mathrm{D}^\prime&=\frac{1}{\theta}\big(\Phi_{(1,1)}\circ \mathrm{D}\circ \Phi^{-1}_{(1,1)}-\mathrm{ad}_{u_0}\big)=0,
\end{align*}
where $u_0=\Phi_{(1,1)}(x_0)$. On the other hand, we have
\begin{align*}
\lambda^\prime&=\frac{1}{\theta}(\Phi^\ast_{(1,1)})^{-1}\Big(\beta\circ \mathrm{D}+a\lambda-\Phi^\ast_{(1,1)}\iota_{u_0}\omega_{\eta,\mathrm{D}}\Big)\\
&=-\frac{a(d_{36}^2+d_{46}^2-\lambda_6)}{\theta} e^5-\frac{a(d_{36}^2+d_{46}^2-\lambda_5)}{\theta}e^6\\
&=\lambda^\prime_1e^5+\lambda^\prime_2e^6.
\end{align*}
For a suitable choice of $\theta$, we have the following cases: $(\lambda^\prime_1,\lambda^\prime_2)=(1,\lambda^\prime_2)$, or $(\lambda^\prime_1,\lambda^\prime_2)=(0,1)$ or $(\lambda^\prime_1,\lambda^\prime_2)=(0,0)$. Next, by applying the automorphism $\Phi_{(1,-1)}$ to $\lambda^\prime$, we find that $\lambda_2^\prime$ can be chosen in $\R^{\ast+}$. Using Lemma~\ref{Isogroup} and Proposition~$\ref{CharaCentral}$, we deduce that $(\G, \omega)$ is symplectomorphically isomorphic to one of the symplectic Lie algebras listed above: $\G_{8,1}$, $\G_{8,2}$, or $\G_{8,3}$, under the choice $x_{45} = 0,\quad x_{35} = 0,\quad x_{26} = 0,\quad x_{16} = 0,\quad x_{42} = 1,\quad x_{24} = 1,\quad x_{76} = \frac{\eta}{\theta}d_{12},\quad x_{75} = -\frac{\eta}{\theta}d_{34}$.

Finally, we conclude that the Lie algebras in Proposition~\ref{ClassiCentral} are classified up to isomorphism by Lemma~\ref{Coho} via the explicit isomorphism $\Psi : \G_{\ell,\varphi_1} \to \G_{\ell,\varphi_2}$ defined by:
\begin{align*}
    \Psi(x) &= x + \beta(x)\xi,  &\text{for all } x \in \overline{\G}, \\
    \Psi(\xi) &= \xi, \\
    \Psi(\ell) &= \ell,
\end{align*}
where $\beta$ is the cohomological data from Lemma~\ref{Coho}.

\end{proof}

Let $\mathrm{D}^{(1)}_\xi, \mathrm{D}^{(2)}_\xi \in \mathrm{Der}(\G_\varphi)$ be two derivations of the central extension. We set $\mathrm{D}^{(2)}_\xi = {^\ast}\mathrm{D}^{(1)}_{\xi}$ if there exists an automorphism $\Phi \in \mathrm{Aut}(\G_\varphi)$ such that 
\[
\Phi \circ \mathrm{D}^{(1)}_\xi \circ \Phi^{-1} = \mathrm{D}^{(2)}_\xi.
\]
\begin{Le}\label{Coho}
Let $\G_{\ell,\varphi_1}$ and $\G_{\ell,\varphi_2}$ be two  extended Lie algebras of $($central$)$ normal extensions $\G_{\varphi_1}$ and $\G_{\varphi_2}$, respectively, with respect to $\mathrm{D}_\xi$ and ${^\ast}\mathrm{D}_\xi$. Then, $\G_{\ell,\varphi_1}$ and $\G_{\ell,\varphi_2}$ are isomorphic if and only if  the cohomology class $[\varphi_1-\varphi_2]\in H^2(\overline{\G})$ vanishes.
\end{Le}

\begin{proof}
Recall that the Lie brackets of $\G_{\ell,\varphi_i}$ are given by
\begin{align}
[x,y]&=\overline{[x,y]}+\varphi_i(x,y)\xi,\hspace{1cm}\text{for all }x,y\in\overline{\G}\label{Bracoho1}\\
[\ell,x]&=\mathrm{D}^{(i)}_\xi(x),\hspace{2.34cm}\text{for all }x\in\overline{\G},\label{Bracoho2}
\end{align}
where, $\mathrm{D}^{(i)}_\xi\in\mathrm{Der}(\G_{\varphi_i})$.
 Suppose that $\Psi:\G_{\ell_1,\varphi_1}\longrightarrow\G_{\ell_2,\varphi_2}$ is an isomorphism of Lie algebras. According to Theorem~$\ref{Theogeneral}$, $\Psi|_{\G_{\varphi_1}}=\Phi:\G_{\varphi_1}\longrightarrow\G_{\varphi_2}$ is an isomorphism of (central) normal extensions which can be written as:
\begin{align*}
\Phi(x)&=\Phi|_{\overline{\G}}(x)+\beta(x)\xi,\hspace{2.1cm}\text{for all }x\in\overline{\G},~\beta\in\overline{\G}^\ast,\\
\Phi(\xi)&=v+a\xi,\hspace{4.1cm}v\in \mathfrak{z}(\overline{\G})\cap\ker\big(\iota_{\Phi|_{\overline{\G}}(\overline{\G})}\varphi_2\big),~a\in\R^\ast,
\end{align*}
where,
\begin{align}\label{CondiCoho}
a\varphi_1-\Phi|_{\overline{\G}}^\ast\big(\varphi_2\big)=\md\beta,\hspace{0.5cm}~~a\in\R^\ast,~~\beta\in\overline{\G}^\ast.
\end{align}
In particular,  it is enough to consider $\Phi$ as follows:
\begin{align*}
\Phi(x)&=x+\beta(x)\xi,\hspace{2.1cm}\text{for all }x\in\overline{\G},~\beta\in\overline{\G}^\ast,\\
\Phi(\xi)&=\xi.
\end{align*}
It follows from~$(\ref{CondiCoho})$ that $\varphi_1$ and $\varphi_2$ have the same cohomology class. Conversely, assume that there exists $\beta\in\overline{\G}^\ast$ such that $\varphi_1-\varphi_2=\md\beta$. Define $\Psi$ as 
\begin{align*}
\Psi(x)&=\Phi(x)+\beta(x)\xi,&\text{for all } x\in\overline{\G},\\
\Psi(\xi)&=\xi,\\
\Psi(\ell)&=\ell,
\end{align*}
where, $\Phi\in\mathrm{Aut}(\overline{\G})$. Set $\mathrm{D}^{(2)}_{\xi}=\Psi_\ast\mathrm{D}^{(1)}_{\xi}$. Clearly, $\Psi$ is bijective. To verify that $\Psi$ is a Lie algebra morphism, it suffices to establish the compatibility condition
\[
\Psi\big([u, w]\big) = \big[\Psi(u), \Psi(w)\big] \quad \text{for all } u,w \in \G_{\ell,\varphi},
\]
which follows immediately from Equations~\eqref{Bracoho1} and \eqref{Bracoho2}.
\end{proof}

\subsection{Normal symplectic oxidation}
By Remark~$\ref{remarkcentral}$, if $(\mathbb{R}\xi)^\perp$ is an ideal, which is equivalent to $\mu=0$, then $(\overline{\G}, \overline{\omega})$ is a normal symplectic reduction of $(\G, \omega)$ with respect to the isotropic ideal $\langle \xi \rangle$ (see \cite{B-C}). We also say that $(\G, \omega)$ is a normal symplectic oxidation of $(\overline{\G}, \overline{\omega})$ with respect to the data $(\mathrm{D}, \lambda, t)$. This situation coincides with the case of non-central oxidation ($w = 0$ or $\mathfrak{z}(\overline{\G}) = \{0\}$ and $t\neq0$).

Together with  Theorem $\ref{Theogeneral}$, we therefore have:

\begin{pr}\label{Prnormal}
Two  normal  oxidations $\G_{(\mathrm{D}_1, \varphi_1, \lambda_1, 0, t_1)}$ and $\G_{(\mathrm{D}_2, \varphi_2, \lambda_2, 0, t_2)}$    of $\overline{\G}$,  are  isomorphic if and only if there exist a  Lie algebra isomorphism of central extensions  $\Psi:\overline{\G}_{\varphi_1}
\longrightarrow\overline{\G}_{\varphi_2}$, $u\in\overline{\G}$, $v\in \mathfrak{z}(\overline{\G})$, $\beta\in\overline{\G}^\ast$, and $\theta,a\in\R^\ast$ such that
\begin{enumerate}
\item $\mathrm{D}_2=\frac{1}{\theta}\Big(\big(\Psi|_{\overline{\G}}\circ \mathrm{D}_1+v\lambda_1\big)\circ\Psi|_{\overline{\G}}^{-1}-\mathrm{ad}_u\big)$,
\item $\Psi^\ast|_{\overline{\G}}\Big(\lambda_2+\frac{1}{\theta}\iota_u\varphi_2\Big)=\frac{1}{\theta}\Big(\beta\circ \mathrm{D}_1+a\lambda_1-t_2\theta\beta\Big)$,
\item$\mathrm{D}_2(v)=\frac{t_1}{\theta} v$,
\item $\varphi_2(u,v)=at_1-\theta\lambda_2(v)-a\theta t_2$.
\end{enumerate}
With, $a\varphi_1-\Psi|_{\overline{\G}}^\ast\big(\varphi_2\big)=\md\beta$.
\end{pr}

\begin{proof}
The proof is carried out within the general framework provided by Theorem~\ref{Theogeneral}, through a restriction to the central case (i.e., for $\mu = 0$), which is equivalent to the non-central oxidation when taking $w_1 = w_2 = 0$. Therefore, the conditions from Theorem~\ref{Theogeneral} reduce to those given in Proposition~\ref{Prnormal}.
\end{proof}

\begin{co}\label{Charactnormal}
Two  normal symplectic oxidations $(\G_{(\mathrm{D}_1,0,\lambda_1,t_1)},\om_1)$ and $(\G_{(\mathrm{D}_2,0,\lambda_2,t_2)},\om_2)$    of $(\overline{\G},\overline{\omega}_1)$ and $(\overline{\G},\overline{\omega}_2)$, respectively, with $\mathfrak{z}(\overline{\G})=\{0\}$,  are symplectomorphically isomorphic if and only if there exist an automorphism $\Phi:\overline{\G}\longrightarrow\overline{\G}$, $u\in\overline{\G}$ and $\theta\in\R^\ast$, $\beta\in\overline{\G}^\ast$ such that
\begin{enumerate}
\item $\Phi^\ast\overline{\omega}_2=\overline{\omega}_1$
\item $\beta=\frac{1}{\theta}\Phi^\ast\big(\iota_u\overline{\omega}_2\big)$
\item $\lambda_2=\frac{1}{\theta}\big(\Phi^\ast\big)^{-1}\Big(\beta\circ \mathrm{D}_1-\iota_u\overline{\omega}_{_2 \mathrm{D}_2}+\frac{1}{\theta}\lambda_1-t_1\beta\Big)$,
\item $\mathrm{D}_2=\frac{1}{\theta}\big(\Phi\circ \mathrm{D}_1\circ\Phi^{-1}-\mathrm{ad}_u\big)$.
\end{enumerate}

\end{co}
\begin{proof}
Let $\Psi\in\mathrm{Isom}(\G_{(\mathrm{D}_1,0,\lambda_1,t_1)}\to\G_{(\mathrm{D}_2,0,\lambda_2,t_2)})$ be an isomorphism of normal symplectic oxidations. According to Theorem $\ref{Theogeneral}$ and Corollary $\ref{Cgeral}$, $\Psi$ is given by
\begin{align*}
\Psi(x)&=\Phi(x)+\beta(x)\xi,\hspace{1cm}\text{for all }x\in\overline{\G},~\beta\in\overline{\G}^\ast,\\
\Psi(\xi)&=a\xi,\hspace{4.25cm}a\in\R^\ast,\\
\Psi(\ell_1)&=u+\theta \ell_2,\hspace{2.64cm}u\in\overline{\G},~\theta\in\R^\ast,
\end{align*} 
where, $\Phi\in\mathrm{Aut}(\overline{\G})$, and $a\overline{\omega}_{\mathrm{D}_1}-\Phi^\ast(\overline{\omega}_{\mathrm{D}_2})=\md\beta$. Therefore, assertions 1. and 2. follow from the proof of Proposition~\ref{CharaCentral}, while the remaining assertions 3. and 4. are consequences of Theorem~\ref{Theogeneral} and Corollary~\ref{Cgeral}.
\end{proof}

In the following, we briefly describe the classification scheme for eight-dimensional non completely reducible symplectic Lie algebras with a normal isotropic ideal. Let $(\G,\omega)$ be an eight-dimensional non completely reducible symplectic Lie algebras with a normal isotropic ideal $\mathfrak{j}=\langle e_7\rangle$. The Lie bracket structure of $\G$ is given by
 \begin{align}
\begin{split}
	[x,y]&=\ol{[x,y]}+\overline{\om}_{\eta,\mathrm{D}}(x,y)e_7, \hspace{0.9cm}\text{for all}\quad x,y\in \ol{\G},\\
	[e_8,x]&=\mathrm{D}(x)+\lambda(x)e_7,\hspace{1.7cm}\text{for all}\hspace{0.65cm}x\in \ol{\G},\\
	[e_8,e_7]&=te_7,\hspace{4.58cm}t\in\R^\ast,
	\end{split}
\end{align} 
where $\mathrm{D}$ is a derivation of $\overline{\G}$ and $\lambda\in\overline{\G}^*$ is a $1$-form  satisfying $t\omega_\mathrm{D}-\overline{\om}_{\eta,\mathrm{D},\mathrm{D}}=\md\lambda$.

We have the following:

\begin{Le}\label{DerivationNormal}
Every derivation on $\G_{(1,0,0,1)}$ satisfying $t\omega_\mathrm{D}-\overline{\om}_{\eta,\mathrm{D},\mathrm{D}}=\md\lambda$ is one of the following$:$
\begin{align*}
\mathrm{D}_1 e_1 &= -d_{12} e_2,       & \mathrm{D}_1 e_2 &= d_{12} e_1,        & \mathrm{D}_1 e_3 &= -d_{34} e_4, \\
\mathrm{D}_1 e_4 &= d_{34} e_3,        & \mathrm{D}_1 e_5 &= d_{15} e_1 + d_{25} e_2, & \mathrm{D}_1 e_6 &= d_{36} e_3 + d_{46} e_4, \\
&&&&\\
\mathrm{D}_2 e_1 &= \tfrac{t}{2} e_1 - d_{12} e_2, & \mathrm{D}_2 e_2 &= d_{12} e_1 + \tfrac{t}{2} e_2, & \mathrm{D}_2 e_3 &= \tfrac{t}{2} e_3 - d_{34} e_4, \\
\mathrm{D}_2 e_4 &= d_{34} e_3 + \tfrac{t}{2} e_4, & \mathrm{D}_2 e_5 &= d_{15} e_1 + d_{25} e_2,       & \mathrm{D}_2 e_6 &= d_{36} e_3 + d_{46} e_4,\\
&&&&\\
\mathrm{D}_3 e_1 &= \tfrac{t}{2} e_1 - d_{12} e_2, & \mathrm{D}_3 e_2 &= d_{12} e_1 + \tfrac{t}{2} e_2, & \mathrm{D}_3 e_3 &=  - d_{34} e_4, \\
\mathrm{D}_3 e_4 &= d_{34} e_3 , & \mathrm{D}_3 e_5 &= d_{15} e_1 + d_{25} e_2,       & \mathrm{D}_3 e_6 &= d_{36} e_3 + d_{46} e_4,\\
&&&&\\
\mathrm{D}_4 e_1 &= - d_{12} e_2, & \mathrm{D}_4 e_2 &= d_{12} e_1 , & \mathrm{D}_4 e_3 &= \tfrac{t}{2} e_3 - d_{34} e_4, \\
\mathrm{D}_4 e_4 &= d_{34} e_3 + \tfrac{t}{2} e_4, & \mathrm{D}_4 e_5 &= d_{15} e_1 + d_{25} e_2,       & \mathrm{D}_4 e_6 &= d_{36} e_3 + d_{46} e_4,
\end{align*}
where, $d_{ij}\in\R$, $t\in\R^\ast$. 
\end{Le}
\begin{proof}
By Lemma~$\ref{Derivation}$, every derivation on $\G_{(1,0,0,1)}$ has the following form$:$
\begin{align*}
\mathrm{D}e_1&=d_{11}e_1-d_{12}e_2,~~\mathrm{D}e_2=d_{12}e_1+d_{11}e_2,~~\mathrm{D}e_3=d_{33}e_3-d_{34}e_4,~~\\
\mathrm{D}e_4&=d_{34}e_3+d_{33}e_4,~~\mathrm{D}e_5=d_{15}e_1+d_{25}e_2,~~\mathrm{D}e_6=d_{36}e_3+d_{46}e_4.
\end{align*}
The condition $t\omega_\mathrm{D}-\overline{\om}_{\eta,\mathrm{D},\mathrm{D}}=\md\lambda$ is equivalent to $d_{11}(t-2d_{11} )=0$ and $d_{33}(t-2d_{33})=0$. Consequently, the solutions are precisely those characterized in Lemma~\ref{DerivationNormal}.
\end{proof}
Next, for each derivation we compute the corresponding isomorphism group, obtaining the following classification:
\begin{Le}\label{Isogroupnormal}
 The isomorphism groups 
$
\mathrm{Isom}\big(\overline{\G}_{(\tilde{\mathrm{D}}_1,\overline{\omega}_{\eta,\mathrm{D}_j},\lambda_1)}\longrightarrow\overline{\G}_{(\tilde{\mathrm{D}}_2,\overline{\omega}_{\eta,\mathrm{D}_j},\lambda_2)}\big)
,~~j=1,2,3,4$,
of normal oxidations are given by the following sets:
\begin{small}
\begin{equation*}
\mathrm{Isom}\big(\overline{\G}_{(\tilde{\mathrm{D}}_1,\overline{\omega}_{\eta,\mathrm{D}_j},\lambda_1)}\longrightarrow\overline{\G}_{(\tilde{\mathrm{D}}_2,\overline{\omega}_{\eta,\mathrm{D}_j},\lambda_2)}\big)=
\left\lbrace\left(
\begin{array}{ccccc|c|c}
&&&&&0&u_1 \\
&&&&&\cdot&\cdot \\
&&\mathrm{Aut}(\G_{(1,0,0,1)})&&&\cdot&\cdot\\
&&&&&\cdot&\cdot\\
&&&&&0&u_{6}\\\hline
\beta^{\mathrm{D}_j}_k(e_1)&\cdot\cdot\cdot&
\cdot\cdot\cdot&\cdot\cdot\cdot&
\beta^{\mathrm{D}_j}_k(e_6)&a&u_\xi\\\hline
0&\cdot\cdot\cdot&
\cdot\cdot\cdot&\cdot\cdot\cdot&0&0&\theta
\end{array}
\right),~~a\theta\neq0\right\rbrace
\end{equation*}
\end{small}
where, $u_1,\ldots,u_6\in\G_{(1,0,0,1)}$, $u_\xi\in\R$.
The mappings $\beta^{\mathrm{D}_j}_k$, $k=1,2$,  are defined as follows$:$
\begin{align*}
\beta^{\mathrm{D}_1}_1(e_1) &= -ad_{15} + d_{36}\varepsilon_2x_{4 2}, & \beta^{\mathrm{D}_1}_1(e_2) &= -ad_{25} + d_{46}x_{4 2}, & \beta^{\mathrm{D}_1}_1(e_3) &= -ad_{36} + d_{15}\varepsilon_1x_{2 4}, \\
\beta^{\mathrm{D}_1}_1(e_4)&=-ad_{46} + d_{25}x_{2 4},&\beta^{\mathrm{D}_1}_1(e_5)&=x,
&\beta^{\mathrm{D}_1}_1(e_6)&=y,
\end{align*}
and
\begin{align*}
\beta^{\mathrm{D}_1}_2(e_1) &=(x_{1 1} - a)d_{15} - \varepsilon_2x_{1 2}d_{25}, & \beta^{\mathrm{D}_1}_2(e_2) &= ((\varepsilon_2x_{1 1} - a)d_{25} + d_{15}x_{1 2}), & \beta^{\mathrm{D}_1}_2(e_3) &= (x_{3 3} - a)d_{36} - \varepsilon_1x_{3 4}d_{46}, \\
\beta^{\mathrm{D}_1}_2(e_4)&=(\varepsilon_1x_{3 3} - a)d_{46} +x_{34}d_{36},&\beta^{\mathrm{D}_1}_2(e_5)&=x,&\beta^{\mathrm{D}_1}_2(e_6)&=y.
\end{align*}
\begin{align*}
\beta^{\mathrm{D}_2}_1(e_1) &= d_{36}\varepsilon_2x_{4 2} + tx_{3 5}x_{4 2} - ad_{15}, & \beta^{\mathrm{D}_2}_1(e_2) &= (d_{46} + t\varepsilon_2x_{4 5})x_{4 2} - ad_{25}, & \beta^{\mathrm{D}_2}_1(e_3) &= d_{15}\varepsilon_1x_{2 4} + tx_{1 6}x_{2 4} - ad_{36}, \\
\beta^{\mathrm{D}_2}_1(e_4)&=(d_{25}+ t\varepsilon_1x_{2 6})x_{2 4} - ad_{46},&\beta^{\mathrm{D}_2}_1(e_5)&=x,
&\beta^{\mathrm{D}_2}_1(e_6)&=y,
\end{align*}
and
\begin{align*}
\beta^{\mathrm{D}_2}_2(e_1) &=(d_{15} + t\varepsilon_2x_{1 5})x_{1 1} - (d_{25}\varepsilon_2 + tx_{2 5})x_{1 2} - ad_{15}, & \beta^{\mathrm{D}_2}_2(e_2) &=(d_{15} + t\varepsilon_2x_{1 5})x_{1 2} + (d_{25}\varepsilon_2 + tx_{2 5})x_{1 1} - ad_{25},\\
 \beta^{\mathrm{D}_2}_2(e_3) &= (d_{36} + t\varepsilon_1x_{3 6})x_{3 3} - (d_{46}\varepsilon_1 + tx_{4 6})x_{3 4} - ad_{36},&\beta^{\mathrm{D}_2}_2(e_4)&=(d_{36}+ t\varepsilon_1x_{3 6})x_{3 4}  + (d_{46}\varepsilon_1 + tx_{4 6})x_{3 3} - ad_{46}, \\
\beta^{\mathrm{D}_2}_2(e_5)&=x,&\beta^{\mathrm{D}_2}_2(e_6)&=y.
\end{align*}
For the derivations $ \mathrm{D}_3$ and $ \mathrm{D}_4$, the automorphism $\Phi_{1(\varepsilon_1,\varepsilon_2)}$ cannot be considered because it does not satisfy condition $(\ref{CondiIsom})$. Alternatively, there exists no $\beta \in \G_{(1,0,0,1)}$ such that condition $(\ref{CondiIsom})$ holds. We therefore have
\begin{align*}
\beta^{\mathrm{D}_3}_2(e_1) &=(d_{15} + t\varepsilon_2x_{1 5})x_{1 1} - (d_{25}\varepsilon_2 + tx_{2 5})x_{1 2} - ad_{15}, & \beta^{\mathrm{D}_3}_2(e_2) &=(d_{15} + t\varepsilon_2x_{1 5})x_{1 2} + (d_{25}\varepsilon_2 + tx_{2 5})x_{1 1} - ad_{25},\\
 \beta^{\mathrm{D}_3}_2(e_3) &= (x_{3 3} - a)d_{36} - \varepsilon_1x_{3 4}d_{46},&\beta^{\mathrm{D}_3}_2(e_4)&=(\varepsilon_1x_{3 3} - a)d_{46} + x_{3 4}d_{36}, \\
\beta^{\mathrm{D}_3}_2(e_5)&=x,&\beta^{\mathrm{D}_3}_2(e_6)&=y.
\end{align*}
and
\begin{align*}
\beta^{\mathrm{D}_4}_2(e_1) &=(x_{1 1} - a)d_{15} - d_{25}\varepsilon_2x_{1 2}, & \beta^{\mathrm{D}_4}_2(e_2) &=(\varepsilon_2x_{1 1} - a)d_{25} +d_{15}x_{1 2},\\
 \beta^{\mathrm{D}_4}_2(e_3) &= (d_{36} + t\varepsilon_1x_{3 6})x_{3 3} - (d_{46}\varepsilon_1 + tx_{4 6})x_{3 4} - ad_{36},&\beta^{\mathrm{D}_4}_2(e_4)&=(d_{36}+ t\varepsilon_1x_{3 6})x_{3 4}  + (d_{46}\varepsilon_1 + tx_{4 6})x_{3 3} - ad_{46}, \\
\beta^{\mathrm{D}_4}_2(e_5)&=x,&\beta^{\mathrm{D}_4}_2(e_6)&=y.
\end{align*}

Here, $\beta^{\mathrm{D}_j}_1$ and $\beta^{\mathrm{D}_j}_2$ correspond to the automorphisms $\Phi_{1(\varepsilon_1, \varepsilon_2)}\) and \(\Phi_{2(\varepsilon_1, \varepsilon_2)}$, respectively, via the condition
\begin{align}\label{CondiIsom}
a \omega_{\eta, \mathrm{D}_j} - \Phi^\ast_{k(\varepsilon_1,\varepsilon_2)} \omega_{\eta, \mathrm{D}_j}&= \md\beta^{\mathrm{D}_j}_k,\quad k=1,2.
\end{align}
 The terms $d_{ij}$ are the   derivation coefficients given in Lemma~$\ref{Derivation}$,  $x_{ij}$ are the automorphism coefficients given in Lemma $\ref{Autg1001}$, $\varepsilon_1,\varepsilon_2=\pm1$, and $x,y\in\R$.
\end{Le}
\begin{remark}
Note that the isomorphism group 
\[
\mathrm{Isom}\left(\overline{\G}_{(\tilde{\mathrm{D}}_1,\overline{\omega}_{\eta,\mathrm{D}_j},\lambda_1)} \longrightarrow \overline{\G}_{(\tilde{\mathrm{D}}_2,\overline{\omega}_{\eta,\mathrm{D}_j},\lambda_2)}\right)
\]
changes in the cases where $j = 2, 3, 4$, depending on the sign of $a \in \R^\ast$. This occurs due to the condition $ a\omega_{\eta,\mathrm{D}_j} - \Phi^\ast\omega_{\eta,\mathrm{D}_j} = \md\beta$ $($see Theorem~$\ref{isombetewnoncentral})$. This yields the following conditions:
\begin{itemize}
    \item For the automorphisms $\Phi_{_1(1,-1)}$ and $\Phi_{_1(-1,-1)}:$
    \[
    x_{24}^2 + a = 0 \quad \text{and} \quad x_{42}^2 + a = 0.
    \]
    
    \item For the automorphisms $\Phi_{_1(1,1)}$ and $\Phi_{_1(-1,1)}:$
    \[
    x_{24}^2 - a = 0 \quad \text{and} \quad x_{42}^2 - a = 0.
    \]
\end{itemize}
and
\begin{itemize}
    \item For the automorphisms $\Phi_{_2(1,1)}$ and $\Phi_{_2(1,-1)}:$
    \[
    x_{1 1}^2 + x_{1 2}^2 - a = 0 \quad \text{and} \quad x_{33}^2 + x_{34}^2 - a = 0.
    \]
    
    \item For the automorphisms $\Phi_{_2(-1,1)}$ and $\Phi_{_2(-1,-1)}:$
    \[
    x_{1 1}^2 + x_{1 2}^2 + a = 0 \quad \text{and} \quad x_{33}^2 + x_{34}^2 + a = 0.
    \]
\end{itemize}
\end{remark}

Now we present a class of non completely reducible symplectic Lie algebras that admit a normal isotropic ideal.

\begin{pr}\label{ClassifiNormal}
Let $(\G,\omega)$ be an eight-dimensional non completely reducible symplectic Lie algebra  with  normal isotropic ideal. Then, $(\G,\omega)$ is symplectomorphically isomorphic to one of the following symplectic Lie algebras$:$
\begin{align*}
		\G_{8,1}:&~~[e_1, e_5] = e_2, ~~\quad 
		[e_2, e_5] = -e_1, \quad
		[e_3, e_6] = e_4, \quad 
		[e_4, e_6] = -e_3,\quad [e_8, e_7] =e_7\\
		&\hspace{0.254cm}\omega=e^{12}+e^{34}+\eta e^{56}+\alpha e^{58}+\beta e^{68}+e^{78},\quad\eta\in\R^{\ast+},~\alpha,\beta\in\R.
\\&\\
		\G_{8,2}:&~~[e_1, e_5] = e_2, ~~\quad 
		[e_2, e_5] = -e_1, \quad
		[e_3, e_6] = e_4, \quad 
		[e_4, e_6] = -e_3,\quad [e_1,e_2]=te_7,\\
&~~[e_3,e_4]=te_7,\hspace{0.42cm}[e_8, e_1] =\tfrac{1}{2}e_1,\hspace{0.36cm} [e_8, e_2] =\tfrac{1}{2}e_2,\hspace{0.2cm}[e_8, e_3] =\tfrac{1}{2}e_3,\hspace{0.33cm}[e_8,e_4]=\tfrac{1}{2}e_4, \\
		&~~[e_8, e_7] =e_7\\
		&\hspace{0.254cm}\omega=e^{12}+e^{34}+\eta e^{56}+\alpha e^{58}+\beta e^{68}-\tfrac{1}{t}e^{78},\quad\eta\in\R^{\ast+},~~t\in\R^\ast,~\alpha,\beta\in\R.
\\&\\
		\G_{8,3}:&~~[e_1, e_5] = e_2, ~~\quad 
		[e_2, e_5] = -e_1, \quad
		[e_3, e_6] = e_4, \quad 
		[e_4, e_6] = -e_3,\quad [e_1,e_2]=te_7,\\
&~~[e_8, e_1] =\tfrac{1}{2}e_1,\hspace{0.33cm} [e_8, e_2] =\tfrac{1}{2}e_2,\hspace{0.36cm}[e_8, e_7] =e_7 \\
		&\hspace{0.254cm}\omega=e^{12}+e^{34}+\eta e^{56}+\alpha e^{58}+\beta e^{68}-\tfrac{1}{t}e^{78},\quad\eta\in\R^{\ast+},~~t\in\R^\ast,~\alpha,\beta\in\R.
		\\&\\
		\G_{8,4}:&~~[e_1, e_5] = e_2, ~~\quad 
		[e_2, e_5] = -e_1, \quad
		[e_3, e_6] = e_4, \quad 
		[e_4, e_6] = -e_3,\quad [e_3,e_4]=te_7,\\
&~~[e_8, e_3] =\tfrac{1}{2}e_3,\hspace{0.33cm} [e_8, e_4] =\tfrac{1}{2}e_4,\hspace{0.36cm}[e_8, e_7] =e_7 \\
		&\hspace{0.254cm}\omega=e^{12}+e^{34}+\eta e^{56}+\alpha e^{58}+\beta e^{68}-\tfrac{1}{t}e^{78},\quad\eta\in\R^{\ast+},~~t\in\R^\ast,~\alpha,\beta\in\R.
	\end{align*}
	
\end{pr}
\begin{proof}
Let $(\G,\omega)$ be an eight-dimensional non completely reducible symplectic Lie algebra,  $\mathfrak{j}=\langle e_7\rangle$ a one-dimensional normal isotropic ideal and $(\G_{(1,0,0,1)},\omega_\eta)$ the symplectic reduction with respect to $\mathfrak{j}$. We begin by examining the first case. Let $\mathrm{D}_1$ be the derivation given in Lemma~$\ref{DerivationNormal}$, then  the Lie brackets of $\G$ are given by
\begin{align*}
[e_1,e_5] &= e_2 + d_{25}e_7,          & [e_2,e_5] &= -e_1 - d_{15}e_7,         & [e_3,e_6] &= e_4 + d_{46}e_7, \\
[e_4,e_6] &= -e_3 - d_{36}e_7,         & [e_8,e_1] &= -d_{12}e_2+(-d_{12}d_{25} - td_{15})e_7, & [e_8,e_2] &= d_{12}e_1 + (d_{12}d_{15} - td_{25})e_7, \\
[e_8,e_3] &= -d_{34}e_4 + (-d_{34}d_{46} - td_{36})e_7, & [e_8,e_4] &= d_{34}e_3 + (d_{34}d_{36} - td_{46})e_7, & [e_8,e_5] &= d_{15}e_1 + d_{25}e_2 + \tilde{\lambda}_5e_6, \\
[e_8,e_6] &= d_{36}e_3 + d_{46}e_4 + \tilde{\lambda}_6e_7,&[e_8,e_7]&=te_8,
\end{align*}
where, \[\lambda_1 =( -d_{12}d_{25} - td_{15})e^1+ (d_{12}d_{15} - td_{25})e^2+(-d_{34}d_{46} - td_{36})e^3+ (d_{34}d_{36} - td_{46})e^4+\tilde{\lambda}_5e^5+\tilde{\lambda}_6e^6.\]
Observe first that, $\mathrm{D}_1$ is an inner derivation on $\G_{(1,0,0,1)}$. Therefore, there exists $x_0\in\G_{(1,0,0,1)}$ such that $\mathrm{D}_1=\mathrm{ad}_{x_0}$, and we have 
\begin{align*}
D^\prime_1&=\frac{1}{\theta}\big(\Phi_{_1(1,1)}\circ \mathrm{D}_1\circ \Phi^{-1}_{_1(1,1)}-\mathrm{ad}_{u_0}\big)=0,
\end{align*}
where $u_0=\Phi_{(1,1)}(x_0)$. On the other hand, we have
\begin{align*}
\lambda^\prime_1&=\frac{1}{\theta}(\Phi^\ast_{_1(1,1)})^{-1}\Big(\beta\circ \mathrm{D}_1+a\lambda_1-\Phi^\ast_{_1(1,1)}\iota_{u_0}\omega_{\eta,\mathrm{D}_1}-t_1^\prime\theta\beta\Big)\\
&=\frac{1}{\theta}(\Phi^\ast_{_1(1,1)})^{-1}\Big(\beta\circ \mathrm{D}_1+a\lambda_1-\Phi^\ast_{_1(1,1)}\iota_{u_0}\omega_{\eta,\mathrm{D}_1}-t\beta\Big)\\
&=-d_{15}e^1-d_{25}e^2-d_{36}e^3-d_{46}e^4.
\end{align*}
since $Z(\G_{1,0,0,1})=\{0\}$ and $\theta t_1^\prime=t$ (see Proposition $\ref{Prnormal}$), and for a suitable choice of $\beta\in\G_{(1,0,0,1)}^\ast$.

Using Lemma~\ref{Isogroupnormal} and Corollary~$\ref{Charactnormal}$, we deduce that $(\G, \omega)$ is symplectomorphically isomorphic to  the symplectic Lie algebra listed above: $\G_{8,1}$, under the choice $x_{45} = 0,\quad x_{35} = 0,\quad x_{26} = 0,\quad x_{16} = 0,\quad x_{42} = 1,\quad x_{24} = 1,\quad a=-\frac{1}{t},\quad\theta=t$.

Now, let  $\mathrm{D}_2$ be the derivation given in Lemma~$\ref{DerivationNormal}$, then  the Lie brackets of $\G$ are  given  by
\begin{align*}
[e_1,e_5] &= e_2 + d_{25}e_7,          & [e_2,e_5] &= -e_1 - d_{15}e_7,         & [e_3,e_6] &= e_4 + d_{46}e_7, \\
[e_1,e_2]&=te_7,&[e_3,e_4]&=te_7,&[e_4,e_6] &= -e_3 - d_{36}e_7,\\
 [e_8,e_1] &=\tfrac{t}{2}e_1 -d_{12}e_2+( -d_{12}d_{25} + \tfrac{t}{2}d_{15})e_7, & [e_8,e_2] &= d_{12}e_1+\tfrac{t}{2}e_2 + (d_{12}d_{15} +\tfrac{t}{2}d_{25})e_7, &[e_8,e_3] &=\tfrac{t}{2}e_3 -d_{34}e_4 + (-d_{34}d_{46} +\tfrac{t }{2}d_{36})e_7, \\
 [e_8,e_4] &= d_{34}e_3+\tfrac{t}{2}e_4 + (d_{34}d_{36} +\tfrac{t}{2}d_{46})e_7, & [e_8,e_5] &= d_{15}e_1 + d_{25}e_2 + \widetilde{\lambda}_5e_6, &[e_8,e_6] &= d_{36}e_3 + d_{46}e_4 + \widetilde{\lambda}_6e_7,\\
[e_8,e_7]&=te_8,
\end{align*}

where $\lambda_2$ is given by: \[\lambda_2 =( -d_{12}d_{25} + \tfrac{t}{2}d_{15})e^1+ (d_{12}d_{15} +\tfrac{t}{2}d_{25})e^2+(-d_{34}d_{46} +\tfrac{t }{2}d_{36})e^3+ (d_{34}d_{36} +\tfrac{t}{2}d_{46})e^4+\widetilde{\lambda}_5e^5+\widetilde{\lambda}_6e^6.\]
Note that, except for the derivation~$\mathrm{D}_1$, the derivations~$\mathrm{D}_2$, $\mathrm{D}_3$, and~$\mathrm{D}_4$ are not inner derivations on~$\G_{(1,0,0,1)}$. So neither $\mathrm{D}_2$, $\mathrm{D}_3$, nor $\mathrm{D}_4$ can be isomorphic to the zero derivation, and we have
\begin{align*}
\mathrm{D}^\prime_2&=\frac{1}{\theta}\big(\Phi_{_1(1,1)}\circ \mathrm{D}_2\circ \Phi^{-1}_{_1(1,1)}-\mathrm{ad}_{u_0}\big)=\frac{1}{2}\mathbf{I}_4 \oplus \mathbf{0}_2,
\end{align*}
where $u_0=\sum u_je_j$ is given by
\begin{align*}
    u_0 &=  \sqrt{a} \, d_{46}e_1 -\sqrt{a} \, d_{36}e_2  +\sqrt{a} \, d_{25}e_3 -\sqrt{a} \, d_{15}e_4+ d_{34}e_5 + d_{12}e_6.
\end{align*}

On the other hand, we have
\begin{align*}
\lambda^\prime_2&=\frac{1}{\theta}(\Phi^\ast_{_1(1,1)})^{-1}\Big(\beta\circ \mathrm{D}_2+a\lambda_2-\Phi^\ast_{_1(1,1)}\iota_{u_0}\omega_{\eta,\mathrm{D}_2}-t_2^\prime\theta\beta\Big)\\
&=\frac{1}{\theta}(\Phi^\ast_{_1(1,1)})^{-1}\Big(\beta\circ \mathrm{D}_2+a\lambda_2-\Phi^\ast_{_1(1,1)}\iota_{u_0}\omega_{\eta,\mathrm{D}_2}-t\beta\Big)\\
&=-\frac{1}{2} d_{15}e^1-\frac{1}{2}d_{25}e^2-\frac{1}{2}d_{36}e^3-\frac{1}{2}d_{46}e^4.
\end{align*}
The last equation is determined by a canonical $1$-form $\beta \in \G_{(1,0,0,1)}^\ast$ whose action on the basis vectors satisfies:
\begin{align*}
\beta(e_5) &= -\frac{a}{t}\left(d_{36}^2 + d_{46}^2 - \widetilde{\lambda}_6\right), \\
\beta(e_6) &= -\frac{a}{t}\left(d_{15}^2 + d_{25}^2 - \widetilde{\lambda}_5\right),
\end{align*}
with $\beta(e_j)$ for $j=1,\ldots,4$ are given in Lemma~\ref{Isogroupnormal}. Here $\theta = t$ is a fixed parameter of the system. The remaining parameters vanish: $
x_{1 6} = 0, \quad x_{2 6} = 0, \quad x_{3 5} = 0, \quad x_{4 5} = 0$. Under these parameter values and by normalizing $a$ (i.e., setting $a = 1$), we apply the isomorphism to the symplectic form $\omega_\eta + e^7 \wedge e^8$ and obtain the form given in Proposition~\ref{ClassifiNormal} for the algebra $\G_{8,2}$.

Let $\mathrm{D}_3$ denote the derivation given in Lemma~\ref{DerivationNormal}. The Lie brackets of $\G$ are then given by
\begin{align*}
[e_1,e_5] &= e_2 + d_{25}e_7,          & [e_2,e_5] &= -e_1 - d_{15}e_7,         & [e_3,e_6] &= e_4 + d_{46}e_7, \\
[e_1,e_2]&=te_7,&[e_4,e_6] &= -e_3 - d_{36}e_7,&[e_8,e_1] &=\tfrac{t}{2}e_1 -d_{12}e_2+( -d_{12}d_{25} + \tfrac{t}{2}d_{15})e_7, \\
 [e_8,e_2] &= d_{12}e_1+\tfrac{t}{2}e_2 + (d_{12}d_{15} +\tfrac{t}{2}d_{25})e_7, &[e_8,e_3] &= -d_{34}e_4 + (-d_{34}d_{46} -td_{36})e_7, &[e_8,e_4] &= d_{34}e_3 + (d_{34}d_{36} -td_{46})e_7, \\
 [e_8,e_5] &= d_{15}e_1 + d_{25}e_2 + \widetilde{\lambda}_5e_6, &[e_8,e_6] &= d_{36}e_3 + d_{46}e_4 + \widetilde{\lambda}_6e_7,&
[e_8,e_7]&=te_8,
\end{align*}

where $\lambda_3$ is given by: \[\lambda_3 =( -d_{12}d_{25} + \tfrac{t}{2}d_{15})e^1+ (d_{12}d_{15} +\tfrac{t}{2}d_{25})e^2+(-d_{34}d_{46} -td_{36})e^3+ (d_{34}d_{36} -td_{46})e^4+\widetilde{\lambda}_5e^5+\widetilde{\lambda}_6e^6.\]
We have
\begin{align*}
\mathrm{D}^\prime_3&=\frac{1}{\theta}\big(\Phi_{_2(1,1)}\circ \mathrm{D}_3\circ \Phi^{-1}_{_2(1,1)}-\mathrm{ad}_{u_0}\big)=\frac{1}{2}\mathbf{I}_2 \oplus \mathbf{0}_4,
\end{align*}
where $u_0=\sum u_je_j$ is given by
\begin{align*}
    u_0 &=  d_{25}e_1 - d_{15}e_2  + d_{46}e_3 - d_{36}e_4+ d_{12}e_5 + d_{34}e_6.
\end{align*}
On the other hand, we have
\begin{align*}
\lambda^\prime_3&=\frac{1}{\theta}(\Phi^\ast_{_2(1,1)})^{-1}\Big(\beta\circ \mathrm{D}_3+a\lambda_3-\Phi^\ast_{_2(1,1)}\iota_{u_0}\omega_{\eta,\mathrm{D}_3}-t_2^\prime\theta\beta\Big)\\
&=\frac{1}{\theta}(\Phi^\ast_{_2(1,1)})^{-1}\Big(\beta\circ \mathrm{D}_3+a\lambda_3-\Phi^\ast_{_2(1,1)}\iota_{u_0}\omega_{\eta,\mathrm{D}_3}-t\beta\Big)\\
&=-\frac{1}{2} d_{15}e^1-\frac{1}{2}d_{25}e^2-d_{36}e^3-d_{46}e^4.
\end{align*}
The last equation is determined by a canonical $1$-form $\beta \in \G_{(1,0,0,1)}^\ast$ whose action on the basis vectors satisfies:
\begin{align*}
\beta(e_5) &= -\frac{d_{15}^2 + d_{25}^2 - \widetilde{\lambda}_5}{t}, \\
\beta(e_6) &= -\frac{d_{36}^2 + d_{46}^2 - \widetilde{\lambda}_6}{t}.
\end{align*}
The values of $\beta(e_j)$ for $j=1,\ldots,4$ are given in Lemma~\ref{Isogroupnormal}. Here, $\theta = t$ is a fixed parameter of the system. The remaining parameters are chosen as
\[
a = 1,\quad x_{3 3} = 1,\quad x_{3 4} = 0,\quad x_{2 5} = 0,\quad x_{1 5} = 0,\quad x_{4 6} = 0,\quad x_{3 6} = 0,\quad x_{1 1} = 1.
\]
We apply the isomorphism to the symplectic form $\omega_\eta + e^7 \wedge e^8$ and obtain the form given in Proposition~\ref{ClassifiNormal} for the algebra $\mathfrak{g}_{8,3}$

Consider now the  derivation $\mathrm{D}_4$ as  given in Lemma~\ref{DerivationNormal}. The Lie brackets of $\G$ are then given by
\begin{align*}
[e_1,e_5] &= e_2 + d_{25}e_7,          & [e_2,e_5] &= -e_1 - d_{15}e_7,         & [e_3,e_6] &= e_4 + d_{46}e_7, \\
[e_3,e_4]&=te_7,&[e_4,e_6] &= -e_3 - d_{36}e_7,&[e_8,e_1] &= -d_{12}e_2+( -d_{12}d_{25} -td_{15})e_7, \\
 [e_8,e_2] &= d_{12}e_1 + (d_{12}d_{15} -td_{25})e_7, &[e_8,e_3] &=\tfrac{t}{2}e_3 -d_{34}e_4 + (-d_{34}d_{46}+\tfrac{t}{2}d_{36})e_7, &[e_8,e_4] &= d_{34}e_3+\tfrac{t}{2}e_4 + (d_{34}d_{36} +\tfrac{t}{2} d_{46})e_7, \\
 [e_8,e_5] &= d_{15}e_1 + d_{25}e_2 + \widetilde{\lambda}_5e_6, &[e_8,e_6] &= d_{36}e_3 + d_{46}e_4 + \widetilde{\lambda}_6e_7,&
[e_8,e_7]&=te_8,
\end{align*}

where $\lambda_4$ is given by: \[\lambda_4 =( -d_{12}d_{25} -td_{15})e^1+ (d_{12}d_{15} -td_{25})e^2+(-d_{34}d_{46}+\tfrac{t}{2}d_{36})e^3+ (d_{34}d_{36} +\tfrac{t}{2} d_{46})e^4+\widetilde{\lambda}_5e^5+\widetilde{\lambda}_6e^6.\]
We have
\begin{align*}
\mathrm{D}^\prime_4&=\frac{1}{\theta}\big(\Phi_{_2(1,1)}\circ \mathrm{D}_4\circ \Phi^{-1}_{_2(1,1)}-\mathrm{ad}_{u_0}\big)=\mathbf{0}_2\oplus\frac{1}{2}\mathbf{I}_2 \oplus \mathbf{0}_2,
\end{align*}
where $u_0=\sum u_je_j$ is given by
\begin{align*}
    u_0 &=  d_{25}e_1 -d_{15}e_2  + d_{46}e_3 -d_{36}e_4+ d_{12}e_5 + d_{34}e_6.
\end{align*}
As in the previous case, we have
\begin{align*}
\lambda^\prime_4 &= \frac{1}{\theta}(\Phi^\ast_{_2(1,1)})^{-1}\Big(\beta \circ \mathrm{D}_4 +a\lambda_4 - \Phi^\ast_{_2(1,1)}\iota_{u_0}\omega_{\eta,\mathrm{D}_4} - t_2^\prime\theta\beta\Big) \\
&= \frac{1}{\theta}(\Phi^\ast_{_2(1,1)})^{-1}\Big(\beta \circ D_4 + a\lambda_4 - \Phi^\ast_{_2(1,1)}\iota_{u_0}\omega_{\eta,\mathrm{D}_4} - t\beta\Big) \\
&= -\frac{1}{2} d_{15}e^1 - \frac{1}{2}d_{25}e^2 - d_{36}e^3 - d_{46}e^4.
\end{align*}
Here, $\beta \in \G_{(1,0,0,1)}^\ast$ is a canonical $1$-form whose action on the basis vectors is chosen as
\begin{align*}
\beta(e_5) &= -\frac{d_{15}^2 + d_{25}^2 - \widetilde{\lambda}_5}{t}, \\
\beta(e_6) &= -\frac{d_{36}^2 + d_{46}^2 - \widetilde{\lambda}_6}{t}.
\end{align*}

The values of $\beta(e_j)$ for $j = 1, \ldots, 4$ are given in Lemma~\ref{Isogroupnormal}. The parameter $\theta = t$ is fixed, and the remaining parameters are assigned as follows:
\[
x_{12} = 0, \quad x_{11} = 1, \quad x_{25} = 0, \quad x_{36} = 0, \quad x_{46} = 0, \quad x_{33} = 1, \quad a = 1, \quad x_{15} = 0, \quad x_{34} = 0.
\]

We therefore conclude that the isomorphism classes of the Lie algebras in Proposition~$\ref{ClassifiNormal}$ are completely classified by the cohomological invariant from Lemma~$\ref{Coho}$, with an analogous proof to the central case established in Proposition~$\ref{ClassiCentral}$.
\end{proof}

\section{D-extension of Lie algebras}\label{D-exten} \label{se4}

Let $\overline{\G}$ be a Lie algebra, $\mu\in\overline{\G}^\ast$ and let $\mathrm{D}\in\mathrm{End}(\overline{\G})$ be an endomorphism. Consider the vector space defined by
\begin{align*}
\G_{\mu,\mathrm{D}}=
\langle \ell\rangle\oplus\overline{\G}.
\end{align*}
Define an alternating bilinear product on $\G_{\mu,\mathrm{D}}\times\G_{\mu,\mathrm{D}}\longrightarrow\G_{\mu,\mathrm{D}}$ as the
bicrossed sums, i.e., $\G_{\mu,\mathrm{D}}=\langle\ell\rangle\bowtie\overline{\G}$, by requiring that the non-zero brackets are given by
\begin{align}
[x,y]&=\overline{[x,y]},\label{Brageneralcase1}\\
[\ell,x]&=\mu(x)\ell+\mathrm{D}(x),\label{Brageneralcase2}
\end{align}
for all $x,y\in\overline{\G}$.
\begin{pr}\label{G,mu,D Liebrackets}
The alternating product $[~,~]$ as declared in $(\ref{Brageneralcase1})-(\ref{Brageneralcase2})$ 
above defines a Lie algebra $\G = (\G_{\mu,\mathrm{D}}, [~,~])$ if and only if
\begin{enumerate}
\item $\md\mu=0$,
\item $\Delta\mathrm{D} -\mu\otimes \mathrm{D}=0$,
\end{enumerate}
for all $x,y\in\overline{\G}$. 
\end{pr}
\begin{proof}
For all $x,y,z\in\overline{\G}$, we have
\begin{align*}
\mathop{\resizebox{1.3\width}{!}{$\sum$}}\limits_{\mathrm{cycl}}[x,[y,z]]=\mathop{\resizebox{1.3\width}{!}{$\sum$}}\limits_{\mathrm{cycl}}\overline{[x,\overline{[y,z]}]}=0.
\end{align*}
Moreover,
\begin{align*}
\mathop{\resizebox{1.3\width}{!}{$\sum$}}\limits_{\mathrm{cycl}}[\ell,[x,y]]&=[\ell,[x,y]]+[y,[\ell,x]]+[x,[y,\ell]]\\
&=\mu(\overline{[x,y]})\ell+\mathrm{D}\big(\overline{[x,y]}\big)-\mu(x)\mu(y)\ell-\mu(x)\mathrm{D}(y)+\overline{[y,\mathrm{D}(x)]}\\
&~~+\mu(y)\mu(x)\ell+\mu(y)\mathrm{D}(x)+\overline{[\mathrm{D}(y),x]}\\
&=\mu(\overline{[x,y]})\ell+\mathrm{D}\big(\overline{[x,y]}\big)-\overline{[\mathrm{D}(x),y]}-\overline{[x,\mathrm{D}(y)]}+\mu(y)\mathrm{D}(x)-\mu(x)\mathrm{D}(y)\\
&=\mu(\overline{[x,y]})\ell+\big(\Delta\mathrm{D}\big)(x,y)-\big(\mu\otimes \mathrm{D}\big)(x,y)\\
&=0.
\end{align*}
\end{proof}

\begin{Le}\label{over(g)ideal}
Let $\G_{\mu,\mathrm{D}}$ be a $\mathrm{D}$-extended Lie algebras via endomorphism $\mathrm{D}$. Then, the following holds$:$
\begin{enumerate}
\item $\mathrm{D}$ is a derivation on the derived algebra of $\overline{\G}$,
\item  $\mathcal{D}^2(\overline{\G})$ is an ideal of $\G_{\mu,\mathrm{D}}$,
\end{enumerate}
\end{Le}
\begin{proof}
As follows from Proposition~\ref{G,mu,D Liebrackets}, we have $\mathrm{d}\mu = 0$, i.e., 
\[
\mu\big(\overline{[x, y]}\big) = 0 \quad \text{for all } x, y \in \overline{\mathfrak{g}}.
\]
Moreover, the second condition of the preceding proposition implies that for all $x, y \in \overline{[\overline{\mathfrak{g}}, \overline{\mathfrak{g}}]}$,
\[
0 = \Delta(\mathrm{D})(x, y) + \mu(y)\mathrm{D}(x) - \mu(x)\mathrm{D}(y) = \Delta(\mathrm{D})(x, y).
\]
This yields the identity
\[
\mathrm{D}\big(\overline{[x, y]}\big) = \overline{[\mathrm{D}(x), y]} + \overline{[x, \mathrm{D}(y)]},
\]
showing that $\mathrm{D}$ is a derivation on the derived algebra of $\overline{\G}$. 

Furthermore, we have the inclusions
\[
\Big[\overline{\G},\mathcal{D}^2(\overline{\G})\Big]=\overline{\Big[\overline{\G},\mathcal{D}^2(\overline{\G})\Big]}\subset\mathcal{D}^2(\overline{\G}),
\]
and
\[
[\ell, \mathcal{D}^2(\overline{\G})] 
= \mu\big(\mathcal{D}^2(\overline{\G})\big)\ell + D\big(\mathcal{D}^2(\overline{\G})\big)
= [D\big(\mathcal{D}^1(\overline{\G})\big),\mathcal{D}^1(\overline{\G})]+[\mathcal{D}^1(\overline{\G}),D\big(\mathcal{D}^1(\overline{\G})\big)]
\subset \mathcal{D}^2(\overline{\G}).
\]
Thus, $\mathcal{D}^2(\overline{\G})$ is an ideal of $\mathfrak{g}_{\mu,\mathrm{D}}$.
\end{proof}

\begin{pr}
Let $\G_{\mu,\mathrm{D}}$ be a $\mathrm{D}$-extended Lie algebra via endomorphism $\mathrm{D}$. Then,
\begin{enumerate}
\item The subspace $\mathfrak{j}_1=\G_1\subset\overline{\G}$ is an ideal of $\G_{\mu,\mathrm{D}}$ if and only if 
\begin{enumerate}
\item[$(i)$]  $\G_1$ is an ideal of $\overline{\G}$,
\item[$(ii)$] $\G_1\subset\ker\big(\mu\big)$,
\item[$(iii)$]  $\G_1$ is $\mathrm{D}$-invariant, i.e., $\mathrm{D}(\G_1)\subset\G_1$.
\end{enumerate}
\item The subspace $\mathfrak{j}_2=\langle\ell,\G_1\rangle$ is an ideal of $\G_{\mu,\mathrm{D}}$ if and only if 
\begin{enumerate}
\item[$(i)$] $\G_1$ is an ideal of $\overline{\G}$,
\item[$(ii)$] $\G_1$ is $\mathrm{D}$-invariant.
\end{enumerate}
\item For all $i\geq2$, the derived ideals $\mathcal{D}^i(\overline{\G})$ in the derived series of $\overline{\G}$ are ideals of both $\overline{\G}$ and $\G_{\mu,\mathrm{D}}$.
\end{enumerate}
\end{pr}
\begin{proof}
Let $\G_1 \subset \overline{\G}$ be a subspace. Now, assume that $\G_1$ is an ideal of $\G_{\mu,\mathrm{D}}$. Then, $[\overline{\G}, \G_1]=\overline{[\overline{\G}, \G_1]} \subset \G_1$, which shows that $\G_1$ is an ideal of $\overline{\G}$. Moreover,
$[\ell, \G_1] = \mu(\G_1)\ell + \mathrm{D}(\G_1) \subset \G_1$, implying that $\mu(\G_1) = 0$ (i.e., $\G_1 \subset \ker\mu$) and $\mathrm{D}(\G_1) \subset \G_1$. Conversely, any subspace satisfying conditions $(i)$, $(ii)$, and $(iii)$ from part $1$ forms an ideal of $\G_{\mu,\mathrm{D}}$. On the other hand, Let $\mathfrak{j}_2 = \langle \ell, \G_1 \rangle \subset \G_{\mu,\mathrm{D}}$ be an ideal. Then $[\ell, \G_1] = \mu(\G_1)\ell + \mathrm{D}(\G_1) \subset \mathfrak{j}_2$ implies $\mathrm{D}(\G_1) \subset \G_1$ and $\mu(\G_1) = 0$, while $[\G, \G_1] \subset \mathfrak{j}_2$ shows $\G_1$ is an ideal of $\G$. Conversely, any subspace $\langle \ell, \G_1 \rangle$ satisfying $(i)$  and $(ii)$ from part $2$  is an ideal of $\G_{\mu,\mathrm{D}}$. The latter consequence $(3.)$ follows from two facts:  Each $\mathcal{D}^i(\overline{\G})$ is an ideal of $\overline{\G}$, and
   the condition $\mu|_{\overline{[\overline{\G},\overline{\G}]}} = 0$ holds on the derived subalgebra.
\end{proof}
\begin{remark}
Note that $\overline{\G}$ is an ideal of $\G_{\mu,\mathrm{D}}$ if and only if $\mu \equiv 0$. 
Moreover, the derived algebra $\overline{[\overline{\G}, \overline{\G}]}$ is an ideal of $\G_{\mu,\mathrm{D}}$ 
if and only if it is $\mathrm{D}$-invariant. 

If, in addition, $\mathrm{D}$ is a derivation on $\overline{\G}$, 
then the derived algebra $\overline{[\overline{\G}, \overline{\G}]}$ is automatically 
an ideal of $\G_{\mu,\mathrm{D}}$.

\end{remark}

\begin{pr}\label{Characterisationgeneralized}
Let $\G_{\mu_1,D_1}$ and $\G_{\mu_2,\mathrm{D}_2}$ be two $\mathrm{D}$-extended Lie algebras via endomorphisms $\mathrm{D}_1$ and $\mathrm{D}_2$ respectively. Then $\G_{\mu_1,\mathrm{D}_1}$ and $ \G_{\mu_2,\mathrm{D}_2}$ are isomorphic if and only if there exist a linear map $\Phi : \overline{\G}_1\to \overline{\G}_2$, $\theta\in\R$, $\sigma\in\overline{\G}^\ast$ and $u\in\overline{\G}$ such that
\begin{enumerate}
\item $\theta\mathrm{D}_2\circ\Phi-\mathrm{D}_2(u)\sigma=\Phi\circ\mathrm{D}_1+u\mu_1-\mathrm{ad}_u\circ\Phi$,
\item $\theta\mu_2\circ\Phi-\mu_2(u)\sigma=\theta\mu_1+\sigma\circ\mathrm{D}_1$, 
\end{enumerate}
where, 
\begin{align}\label{CharactC1}
\theta\Delta(\Phi)&=u(\sigma
\otimes\mu_1)+\sigma\otimes(\Phi\circ \mathrm{D}_1)-\sigma\otimes(\mathrm{ad}_u\circ\Phi),
\end{align}
and
\begin{align}\label{CharactC2}
-\theta\md\sigma&=\theta(\sigma\otimes\mu_1)+\sigma\otimes(\sigma\circ\mathrm{D}_1).
\end{align}
\end{pr}
\begin{proof}
Suppose that there exists a Lie algebra isomorphism $\Psi:\G_{\mu_1,\mathrm{D}_1}\longrightarrow\G_{\mu_2,\mathrm{D}_2}$. Then, for all $x,y\in\overline{\G}_1$, we have
\begin{align*}
\Psi\big([x,y]_1\big)&=
\Psi\big(\overline{[x,y]}_1\big)\\
&=\Psi|_{\overline{\G}_1}\big(\overline{[x,y]}_1\big)+\sigma\big(\overline{[x,y]}_1\big)\ell_2,
\end{align*}
and
\begin{align*}
[\Psi(x),\Psi(y)]_2&=[\Psi|_{\overline{\G}_1}\big(x)+\sigma\big(x)\ell_2,\Psi|_{\overline{\G}_1}\big(y)+\sigma\big(y)\ell_2]_2\\
&=\overline{[\Psi|_{\overline{\G}_1}(x),\Psi|_{\overline{\G}_1}(y)]}_2-\sigma(y)\mu_2\big(\Psi|_{\overline{\G}_1}(x)\big)\ell_2-\sigma(y)\mathrm{D}_2\big(\Psi|_{\overline{\G}_1}(x)\big)\\
&~~+\sigma(x)\mu_2\big(\Psi|_{\overline{\G}_1}(y)\big)\ell_2+\sigma(x)\mathrm{D}_2\big(\Psi|_{\overline{\G}_1}(y)\big).
\end{align*}
Therefore,
\begin{align}\label{G1}
\Psi|_{\overline{\G}_1}\big(\overline{[x,y]}_1\big)&=\overline{[\Psi|_{\overline{\G}_1}(x),\Psi|_{\overline{\G}_1}(y)]}_2+\sigma(x)\mathrm{D}_2\big(\Psi|_{\overline{\G}_1}(y)\big)-\sigma(y)\mathrm{D}_2\big(\Psi|_{\overline{\G}_1}(x)\big),
\end{align}
and
\begin{align}\label{G2}
\sigma\big(\overline{[x,y]}_1\big)&=\sigma(x)\mu_2\big(\Psi|_{\overline{\G}_1}(y)\big)-\sigma(y)\mu_2\big(\Psi|_{\overline{\G}_1}(x)\big).
\end{align}
We have also, for all $x\in\overline{\G}_1$,
\begin{align*}
\Psi\big([\ell_1,x]_1\big)&=
\mu_1(x)\Psi(\ell_1)+\Psi\big(\mathrm{D}_1(x)\big)\big)\\
&=\mu_1(x)(u+\theta\ell_2)+\Psi|_{\overline{\G}_1}\big(\mathrm{D}_1(x)\big)+\sigma\big(\mathrm{D}_1(x)\big)\ell_2,
\end{align*}
and
\begin{align*}
[\Psi(\ell_1),\Psi(x)]_2&=[u+\theta\ell_2,\Psi|_{\overline{\G}_1}(x)+\sigma(x)\ell_2]_2\\
&=\overline{[u,\Psi|_{\overline{\G}_1}(x)]}_2-\sigma(x)\mu_2(u)\ell_2-\sigma(x)D_2(u)+\theta\mu_2\big(\Psi|_{\overline{\G}_1}(x)\big)\ell_2+\theta D_2\big(\Psi|_{\overline{\G}_1}(x)\big).
\end{align*}
Thus,
\begin{align}\label{G3}
\mu_1(x)u+\Psi|_{\overline{\G}_1}\big(\mathrm{D}_1(x)\big)&=\overline{[u,\Psi|_{\overline{\G}_1}(x)]}_2-\sigma(x)\mathrm{D}_2(u)+\theta \mathrm{D}_2\big(\Psi|_{\overline{\G}_1}(x)\big),
\end{align}
and
\begin{align}\label{G4}
\theta\mu_1(x)+\sigma\big(\mathrm{D}_1(x)\big)&=-\sigma(x)\mu_2(u)+\theta\mu_2\big(\Psi|_{\overline{\G}_1}(x)\big).
\end{align}
Set $\Psi|_{\overline{\G}_1} = \Phi$. By substituting the term $\mathrm{D}_2(\Phi(x))$ from Equation~$(\ref{G3})$ into Equation~$(\ref{G1})$, we obtain, after simplification:
\begin{align}\label{G11}
\theta\big(\Phi\big(\overline{[x,y]}_1\big)-\overline{[\Phi(x),\Phi(y)]}_2\big)&=\Big(u(\sigma\otimes\mu_1)+\sigma\otimes(\Phi\circ\mathrm{D}_1)-\sigma\otimes(\mathrm{ad}_u\circ\Phi)\Big)(x,y).
\end{align}
In the same, by substituting the term $\mu_2(\Phi(x))$ from Equation~$(\ref{G4})$ into Equation~$(\ref{G2})$, we obtain
\begin{align*}
-\big(\theta\md\sigma\big)(x,y)=\Big(\theta(\sigma\otimes
\mu_1)+\sigma\otimes(\sigma\circ\mathrm{D}_1)\Big)(x,y).
\end{align*}
\end{proof}
Let $\overline{\G}$ be a Lie algebra and $\overline{[~,~]}$ its Lie bracket. Assume
the following set of additional data 
\begin{enumerate}
\item $\chi\in\overline{\G}^\ast$  is a $1$-form,
\item $\varphi \in C^2(\overline{\G},\mathbb{R})$ is a 2-cochain.
\end{enumerate}
Let
\begin{align}
\G_{\xi}&=\overline{\G}\oplus\langle\xi\rangle
\end{align}
be the vector space direct sum of $\overline{\G}$ with one-dimensional vector space $\langle\xi\rangle$. Define
an alternating bilinear product
\begin{align}
[x,y]&=\overline{[x,y]}+\varphi(x,y)\xi,\label{Brageneralnoncentral1}\\
[\xi,x]&=\chi(x)\xi,\label{Brageneralnoncentral2}
\end{align}
for all $x,y\in\overline{\G}$.

Similarly, we can state:
\begin{pr}\label{Extennoncentral}
The alternating product $[~,~]$ as declared in $(\ref{Brageneralnoncentral1})-(\ref{Brageneralnoncentral2})$ 
above defines a Lie algebra $\G_{\varphi,\chi} = (\G_{\xi}, [~,~])$ if and only if
\begin{enumerate}
\item $\md\chi=0$,
\item $\mathop{\resizebox{1.3\width}{!}{$\sum$}}\limits_{\mathrm{cycl}}\chi(x)\varphi(y,z)=0$,
\end{enumerate}
for all $x,y\in\overline{\G}$.
\end{pr}
\begin{proof}
For all $x,y,z\in\overline{\G}$, we have
\begin{align*}
\mathop{\resizebox{1.3\width}{!}{$\sum$}}\limits_{\mathrm{cycl}}[x,[y,z]]&=[x,[y,z]]+[z,[x,y]]+[y,[z,x]]\\
&=[x,\overline{[y,z]}+\varphi(y,z)\xi]+[z,\overline{[x,y]}+\varphi(x,y)\xi]+[y,\overline{[z,x]}+\varphi(z,x)\xi]\\
&=\overline{[x,\overline{[y,z]}]}-\chi(x)\varphi(y,z)+\overline{[z,\overline{[x,y]}]}-\chi(z)\varphi(x,y)+\overline{[y,\overline{[z,x]}]}-\chi(y)\varphi(z,x)\\
&=\mathop{\resizebox{1.3\width}{!}{$\sum$}}\limits_{\mathrm{cycl}}\overline{[x,\overline{[y,z]}]}+\mathop{\resizebox{1.3\width}{!}{$\sum$}}\limits_{\mathrm{cycl}}\chi(x)\varphi(y,z)=0
\\
&=\mathop{\resizebox{1.3\width}{!}{$\sum$}}\limits_{\mathrm{cycl}}\chi(x)\varphi(y,z)=0.
\end{align*}
Moreover,
\begin{align*}
\mathop{\resizebox{1.3\width}{!}{$\sum$}}\limits_{\mathrm{cycl}}[x,[y,\xi]]&=[x,[y,\xi]]+[\xi,[x,y]]+[y,[\xi,x]]\\
&=-[x,\chi(y)\xi]+[\xi,\overline{[x,y]}+\varphi(x,y)\xi]+[y,\chi(x)\xi]\\
&=\chi(x)\chi(y)\xi+\chi\big(\overline{[x,y]}\big)\xi-\chi(x)\chi(y)\xi\\
&=\chi\big(\overline{[x,y]}\big)\xi=0.
\end{align*}
\end{proof}
\begin{pr}\label{isoofnocentral}
Let $\G_{\varphi_1,\chi_1}$ and $\G_{\varphi_2,\chi_2}$ be Lie algebras as defined in Proposition~$\ref{Extennoncentral}$. Then $\G_{\varphi_1,\chi_1}$ and $\G_{\varphi_2,\chi_2}$ are isomorphic if and only if there exist a linear map $\Phi:\overline{\G}_1\longrightarrow\overline{\G}_2$, $v\in\overline{\G}_2$, $\varrho\in\overline{\G}_1^\ast$ and $\theta\in\R$ such that 
\begin{enumerate}
\item $\Phi^\ast\varphi_2=a\varphi_1-\varrho\otimes\Phi^\ast\chi_2-\md\varrho$,
\item $a\Phi^\ast\chi_2-\chi_2(v)\varrho+\Phi^\ast\iota_v\varphi_2=a\chi_1$,
\end{enumerate}
where,
\[\Delta(\Phi)=-v\varphi_1\quad\text{and}\quad\Phi^\ast\big(\mathrm{ad}_v\big)=v\chi_1.\]

\end{pr}
\begin{proof}
Suppose that there exists a Lie algebra isomorphism $\Psi:\G_{\varphi_1,\chi_1}\longrightarrow\G_{\varphi_2,\chi_2}$.  We know that $\Psi$ maps the one-dimensional ideal $\langle\xi\rangle$ of $\G_{\varphi_1,\chi_1}$ into the one-dimensional ideal $\mathfrak{j}$ of $\G_{\varphi_2,\chi_2}$,
then
\begin{align*}
\Psi(x)&=\Phi(x)+\varrho(x)\xi,\hspace{1cm} \text{for all } x\in\overline{\G}_1,\\
\Psi(\xi)&=v+a\xi,\hspace{1.75cm} a\in\R,~~v\in\overline{\G}_2.
\end{align*}
Thus, for all $x,y\in\overline{\G}_1$
\begin{align*}
\Psi\big([x,y]_1\big)&=\Psi
\big(\overline{[x,y]}_1+\varphi_1(x,y)\xi\big)\\
&=\Phi\big(\overline{[x,y]}_1\big)+\varrho(\overline{[x,y]}_1)\xi+\varphi_1(x,y)(v+a\xi),
\end{align*}
and
\begin{align*}
[\Psi(x),\Psi(y)]_2&=[\Phi(x)+\varrho(x)\xi,\Phi(y)+\varrho(y)\xi]_2\\
&=\overline{[\Phi(x),\Phi(y)]}_2+\varphi_2\big(\Phi(x),\Phi(y)\big)\xi+\varrho(x)\chi_2\big(\Phi(y)\big)\xi-\varrho(y)\chi_2\big(\Phi(x)\big)\xi\\
&=\overline{[\Phi(x),\Phi(y)]}_2+\big(\Phi^\ast\varphi_2\big)(x,y)\xi+\big(\varrho\otimes(\chi_2\circ\Phi)\big)(x,y)\xi.
\end{align*}
Therefore,
\begin{align}\label{CChi1}
\Phi\big(\overline{[x,y]}_1\big)+\varphi_1(x,y)v&=\overline{[\Phi(x),\Phi(y)]}_2,
\end{align}
and
\begin{align}\label{CChi2}
a\varphi_1-\Phi^\ast\varphi_2-\varrho\otimes(\chi_2\circ\Phi)&=\md\varrho.\end{align}
Furthermore, for all $x\in\overline{\G}_1$, we have
\begin{align*}
\Psi\big([\xi,x]_1\big)&=\Psi(\chi_1(x)\xi)=\chi_1(x)(v+a\xi),
\end{align*}
and
\begin{align*}
[\Psi(\xi),\Psi(x)]_2&=[v+a\xi,\Phi(x)+\varrho(x)\xi]_2\\
&=\overline{[v,\Phi(x)]}_2+\varphi_2(v,\Phi(x))\xi-\varrho(x)\chi_2(v)\xi+a\chi_2(\Phi(x))\xi.
\end{align*}
Consequently,
\begin{align}\label{chi3}
v\chi_1&=\mathrm{ad}_v\circ\Phi
\end{align}
and
\begin{align}\label{chi4}
a\chi_2\circ\Phi-\chi_2(v)\varrho&=a\chi_1-\iota_v\varphi_2\circ\Phi
\end{align}
\end{proof}
 \begin{co}\label{Classifynoncentral}
Let $\G_{\varphi_1,\chi_1}$ and $\G_{\varphi_2,\chi_2}$ be two non-central extensions of $\overline{\G}_1$ and $\overline{\G}_2$, respectively, as in Proposition~$\ref{Extennoncentral}$. Suppose $\overline{\G}_2$ admits no one-dimensional ideals.  Then, $\G_{\varphi_1,\chi_1}$ and $\G_{\varphi_2,\chi_2}$ are isomorphic if and only if there exist a Lie algebra isomorphism $\Phi \colon \overline{\G}_1 \longrightarrow \overline{\G}_2$, $a\in\R^\ast$, $\varrho\in\overline{\G}_1^\ast$ such that 
\begin{enumerate}
\item $\Phi^\ast\varphi_2=a\varphi_1-\varrho\otimes\chi_1-\md\varrho$,
\item $\Phi^\ast\chi_2=\chi_1$.
\end{enumerate}
If, in addition, $\overline{\G}_1 \equiv \overline{\G}_2$, then $\Phi$ becomes a Lie algebra automorphism.
\end{co}

Let $\overline{\G}$ be a Lie algebra and $\G_{\varphi,\chi}$ its non-central extension. The $\mathrm{D}$-extension $\G_{\varphi,\chi,\widetilde{\mathrm{D}}}$ of $\G_\xi$ has the following bracket structure:
\begin{align*}
[x,y] &= \overline{[x,y]} + \varphi(x,y)\xi, & \forall x,y \in \overline{\G}, \\
[\xi,x] &= \chi(x)\xi, & \forall x \in \overline{\G}, \\
[\ell,\xi] &= \tilde{\mu}(\xi)\ell + \widetilde{\mathrm{D}}(\xi), \\
[\ell,x] &= \tilde{\mu}(x)\ell + \widetilde{\mathrm{D}}(x), & \forall x \in \overline{\G},
\end{align*}
where, $\tilde{\mu}\in\G_\xi^\ast$, $\widetilde{\mathrm{D}}\in\mathrm{End}(\G_\xi)$ such that $\Delta\widetilde{\mathrm{D}}-\tilde{\mu}\otimes\widetilde{\mathrm{D}}=0$ and  $\tilde{\mu}\in Z^1(\G_\xi)$. 
We now examine the previous conditions imposed on $\widetilde{\mathrm{D}}$ and $\tilde{\mu}$ by expressing them on $\G_\xi=\overline{\G} \oplus \langle \xi \rangle$. Set $\widetilde{\mathrm{D}}(x) = \mathrm{D}(x) + \lambda(x)\xi$ and $\widetilde{\mathrm{D}}(\xi)=w+t\xi$, for all $x\in\overline{\G}$, $w\in\overline{\G}$, where $\mathrm{D} \in \mathrm{End}(\overline{\G})$, $\lambda \in \overline{\G}^\ast$, and  $\tilde{\mu} = \mu + \kappa \xi^\ast$, where $\kappa \in \mathbb{R}$ and $\mu \in \overline{\G}^\ast$. We obtain
\begin{enumerate}
\item The condition
$\md\tilde{\mu}=0$ is equivalent to the following $:$
\begin{align}\label{COD0}
\md\mu&=\kappa\varphi+\kappa
\chi\wedge\xi^\ast.
\end{align}
\item  $\Delta\widetilde{\mathrm{D}}-\tilde{\mu}\otimes\widetilde{\mathrm{D}}=0$ is equivalent to the following conditions
\begin{align}
\Delta\mathrm{D}-\mu\otimes\mathrm{D}-w\varphi&=0,\label{COD1}\\
t\varphi-\varphi_{\mathrm{D}}-\md\lambda+\lambda\otimes\chi+\mu\otimes\lambda&=0,\label{COD2}\\
\chi\circ\mathrm{D}+\iota_w\varphi-t\mu+\kappa\lambda&=0,\label{COD3}\\
w\chi-\mathrm{ad}_w+w\mu-\kappa\mathrm{D}&=0.\label{COD4}
\end{align}
\end{enumerate}
 The Lie brackets of the $\mathrm{D}$-extension $\G_{\varphi,\chi,\widetilde{\mathrm{D}}}$ Lie algebra of $\G_\xi$ are given by
\begin{align}\label{BracketDsym}
\begin{split}
[x,y] &= \overline{[x,y]} + \varphi(x,y)\xi, \hspace{2.3cm} \forall x,y \in \overline{\G}, \\
[\xi,x] &= \chi(x)\xi, \hspace{4cm} \forall x \in \overline{\G}, \\
[\ell,\xi] &= \kappa\ell + w+t\xi, \\
[\ell,x] &= \mu(x)\ell + \mathrm{D}(x)+\lambda(x)\xi,\hspace{1.85cm}  \forall x \in \overline{\G},
\end{split}
\end{align}
where, 
\begin{itemize}
    \item $\varphi \in C^2(\overline{\G},\mathbb{R})$ is a 2-cochain such that $\mathop{\resizebox{1.3\width}{!}{$\sum$}}\limits_{\mathrm{cycl}}\chi(x)\varphi(y,z)=0$, for all $x,y,z\in\overline{\G}$.
    \item $\chi\in Z^1(\overline{\G})$ is a $1$-cocycles.
    \item $\mathrm{D} \in \mathrm{End}(\overline{\G})$ is an endomorphism, $\lambda, \mu \in \overline{\G}^\ast$ are linear forms, $w\in\overline{\G}$, and $t, \kappa \in \mathbb{R}$ are scalars, subject to the structural conditions $(\ref{COD0})$-$(\ref{COD4})$.
\end{itemize}
Suppose now that $\overline{\G}$ is a symplectic Lie algebra and let $\overline{\omega}$ be a symplectic form on $\overline{\G}$. We let $\omega$ be the non-degenerate alternating two-form on $\G_{\varphi,\chi,\mathrm{D}}$, which is defined as $\omega=\overline{\omega}$ on $\overline{\G}$, $\omega(x,\xi)=\omega(x,\ell)=0$, for all $x\in\overline{\G}$, and $\omega(\ell,\xi)=1$.
\begin{pr}\label{Condisymp}
The form $\omega$ is symplectic for the Lie-algebra $\G_{\varphi,\chi,D}$ if and only if
\begin{enumerate}
\item $\varphi=\overline{\omega}_{\mathrm{D}}$,
\item $\iota_w\overline{\omega}=-\chi-\mu$.
\end{enumerate}
\end{pr}
\begin{proof}
It suffices to check that $\omega$ is closed. For all $x,y,z\in\overline{\G}$, we have
\begin{align*}
\mathop{\resizebox{1.3\width}{!}{$\sum$}}\limits_{\mathrm{cycl}}\omega\big(x,[y,z]\big)&=\omega\big(x,[y,z]\big)+\omega\big(z,[x,y]\big)+\omega\big(y,[z,x]\big)\\
&=\omega\big(x,\overline{[y,z]}+\varphi(y,z)\xi\big)+\omega\big(z,\overline{[x,y]}+\varphi(x,y)\xi\big)+\omega\big(y,\overline{[z,x]}+\varphi(z,x)\xi\big)\\
&=\mathop{\resizebox{1.3\width}{!}{$\sum$}}\limits_{\mathrm{cycl}}\overline{\omega}\big(x,\overline{[y,z]}\big)\\
&=0.
\end{align*}
Moreover,
\begin{align*}
\mathop{\resizebox{1.3\width}{!}{$\sum$}}\limits_{\mathrm{cycl}}\omega\big(x,[y,\xi]\big)&=\omega\big(x,[y,\xi]\big)+\omega\big(\xi,[x,y]\big)+\omega\big(y,[\xi,x]\big)\\
&=-\omega\big(x,\chi(y)\xi\big)+\omega\big(\xi,\overline{[x,y]}+\varphi(x,y)\xi\big)+\omega\big(y,\chi(x)\xi\big)\\
&=0,
\end{align*}
and
\begin{align*}
\mathop{\resizebox{1.3\width}{!}{$\sum$}}\limits_{\mathrm{cycl}}\omega\big(x,[y,\ell]\big)&=\omega\big(x,[y,\ell]\big)+\omega\big(\ell,[x,y]\big)+\omega\big(y,[\ell,x]\big)\\
&=-\omega\big(x,\mu(y)\ell+\mathrm{D}(y)+\lambda(y)\xi\big)+\omega\big(\ell,\overline{[x,y]}+\varphi(x,y)\xi\big)\\
&\hspace{0.35cm}+\omega\big(y,\mu(x)\ell+\mathrm{D}(x)+\lambda(x)\xi\big)\\
&=-\overline{\omega}\big(x,\mathrm{D}(y)\big)+\varphi(x,y)+\overline{\omega}\big(y,\mathrm{D}(x)\big)\\
&=-\overline{\omega}_{\mathrm{D}}(x,y)+\varphi(x,y).
\end{align*}
Finally,
we obtain for $x\in\overline{\G}$,
\begin{align*}
\mathop{\resizebox{1.3\width}{!}{$\sum$}}\limits_{\mathrm{cycl}}
\omega\big(x,[\xi,\ell]\big)&=\omega\big(x,[\xi,\ell]\big)+\omega\big(\ell,[x,\xi]\big)+\omega\big(\xi,[\ell,x]\big)\\
&=\omega\big(x,\kappa\ell+w+t\xi\big)-\omega\big(\ell,\chi(x)\xi\big)+\omega\big(\xi,\mu(x)\ell+\mathrm{D}(x)+\lambda(x)\xi\big)\\
&=\overline{\omega}(x,w)-\chi(x)-\mu(x).
\end{align*}
\end{proof}
\begin{theo}\label{Moreconditions}
Let $(\G_{\varphi_1,\chi_1,\widetilde{\mathrm{D}}_1},\omega_1)$ and $(\G_{\varphi_2,\chi_2,\widetilde{\mathrm{D}}_2},\omega_2)$ be two  $\widetilde{\mathrm{D}}$-extensions of $\G_\xi$. Then, $\G_{\varphi_1,\chi_1,\widetilde{\mathrm{D}}_1}$ and $\G_{\varphi_2,\chi_2,\widetilde{\mathrm{D}}_2}$ are  isomorphic if and only if there exist an endomorphism $\Phi:\overline{\G}\longrightarrow\overline{\G}$, $u,v\in\overline{\G}$, $\varrho,\sigma\in\overline{\G}^\ast$ and $a,\sigma_\xi,u_\xi,\theta\in\R$ such that
\begin{enumerate}
\item  $\theta\big(\mathrm{D}_2\circ\Phi+w_2\varrho\big)-\big(\mathrm{D}_2(u)+u_\xi w_2\big)\sigma=\Phi\circ\mathrm{D}_1+v\lambda_1+u\mu_1-\overline{\mathrm{ad}}_u\circ\Phi$,
\item $\theta\big(\Phi^\ast\lambda_2+t_2\varrho\big)-\big(\lambda_2(u)+t_2u_\xi\big) \sigma=\varrho\circ\mathrm{D}_1+a\lambda_1+u_\xi\mu_1-\iota_u\varphi_2\circ\Phi-u_\xi\Phi^\ast\chi_2+\chi_2(u)\varrho$,
\item $\theta\big(\mathrm{D}_2(v)+aw_2\big)-\sigma_\xi\big(\mathrm{D}_2(u)+u_\xi w_2\big)=\Phi(w_1)+t_1v+\kappa_1u-\overline{\mathrm{ad}}_uv$,
\item $\theta\big(\lambda_2(v)+at_2\big)-\sigma_\xi\big(\lambda_2(u)+u_\xi t_2\big)=\varrho(w_1)+t_1a+\kappa_1u_\xi-\varphi_2(u,v)-u_\xi\chi_2(v)+a\chi_2(u)$,
\item $\theta\big(\Phi^\ast\mu_2+\kappa_2\varrho\big)-(\mu_2(u)+\kappa_2u_\xi)\sigma=\theta\mu_1+\sigma\circ\mathrm{D}_1+\sigma_\xi\lambda_1$,
\item $\theta\big(\mu_2(v)+a\kappa_2\big)-\sigma_\xi(\mu_2(u)+\kappa_2u_\xi)=\theta\kappa_1+\sigma(w_1)+t_1\sigma_\xi$. 
\end{enumerate}
where,
\begin{align*}
\theta\Big(\big(\Delta\Phi\big)+v\varphi_1\Big)&=u(\sigma\otimes\mu_1)+\sigma\otimes\Phi\circ\mathrm{D}_1+v\big(\sigma\otimes\lambda_1\big)
-\sigma\otimes\overline{\mathrm{ad}}_{u}\circ\Phi,\\
\theta\Big(a\varphi_1-\Phi^\ast\varphi_2-\varrho\otimes\Phi^\ast\chi_2-\md\varrho\Big)&=u_\xi\big(\sigma\otimes\mu_1-\sigma\otimes\Phi^\ast\chi_2\big)
+\sigma\otimes\varrho\circ\mathrm{D}_1+a\big(\sigma\otimes\lambda_1\big)
+\chi_2(u)\big(\sigma\otimes\varrho\big)-\sigma\otimes\iota_u
\varphi_2\circ\Phi,\\
\theta\Big(-v\chi_1+\overline{\mathrm{ad}_v}\circ\Phi\Big)&=u\big(\kappa_1\sigma-\sigma_\xi\mu_1\big)+\big(\Phi(w_1)+t_1v)\sigma-\sigma_\xi\big(\Phi\circ\mathrm{D}_1+v\lambda_1\big)-\overline{\mathrm{ad}}_uv\sigma+\sigma_\xi\overline{\mathrm{ad}}_u\circ\Phi,\\
\theta\Big(a\Phi^\ast\chi_2-a\chi_1+\iota_v\varphi_2\circ\Phi-\chi_2(v)\varrho\Big)&=u_\xi\big(\kappa_1\sigma-\sigma_\xi\mu_1\big)+\big(\varrho(w_1)+at_1\big)\sigma+\sigma_\xi\big(\varrho\circ\mathrm{D}_1+a\lambda_1\big)\\
&\hspace{0.3cm}-\big(\varphi_2(u,v)+u_\xi\chi_2(v)-a\chi_2(u)\big)\sigma+\sigma_\xi\big(\iota_u\varphi_2
\circ\Phi+u_\xi
\Phi^\ast\chi_2-\chi_2(u)\varrho\big),\\
\theta\big(-\md\sigma+\sigma_\xi\varphi_1)&=\theta\big(\sigma
\otimes\mu_1\big)+\sigma\otimes\sigma\circ\mathrm{D}_1+\sigma_\xi\big(\sigma
\otimes\lambda_1\big),\\
\theta\Phi^\ast\mu_2+\theta\kappa_2\varrho-(\mu_2(u)+\kappa_2u_\xi)\sigma&=\theta\mu_1+\sigma\circ\mathrm{D}_1+\sigma_\xi\lambda_1.
\end{align*}
\end{theo}
\begin{proof}
Suppose that $\Upsilon:\G_{\varphi_1,\chi_1,\widetilde{\mathrm{D}}_1}\longrightarrow\G_{\varphi_2,\chi_2,\widetilde{\mathrm{D}}_2}$ is an isomorphism of Lie algebras. Then, $\Upsilon$ can be expressed as
\begin{align*}
\Upsilon(x)&=\widetilde{\Psi}(x)+\sigma(x)\ell=\Phi(x)+\varrho(x)\xi+\sigma(x)\ell,\hspace{1cm} \text{for all } x\in\overline{\G},~~\varrho,~~\sigma\in\overline{\G}^\ast,\\
\Upsilon(\xi)&=\widetilde{\Psi}(\xi)+\sigma(\xi)\ell=v+a\xi+\sigma_\xi\ell,\hspace{3.1cm} v\in\overline{\G},~~a,~~\sigma_\xi\in\R,\\
\Upsilon(\ell)&=\tilde{u}+\theta\ell=u+u_\xi\xi+\theta\ell, \hspace{4.18cm} u\in\overline{\G},~~u_\xi,~~\theta\in\R.
\end{align*}
According to Proposition~$\ref{Characterisationgeneralized}$, Equation~$(\ref{CharactC1})$, we have 
\begin{align*}
\theta\big(\Delta\widetilde{\Psi}\big)(x,y)&=\tilde{u}\big(\sigma\otimes\tilde{\mu}_1\big)(x,y)+\sigma\otimes(\widetilde{\Psi}\circ
\widetilde{\mathrm{D}}_1)(x,y)-\sigma\otimes(\mathrm{ad}_{\tilde{u}}\circ\widetilde{\Psi})(x,y),
\end{align*}
for all $x,y\in\overline{\G}$. On the one hand, for all $x,y\in\overline{\G}$, we have
\begin{align*}
\Delta\widetilde{\Psi}(x,y)&=\widetilde{\Psi}\big([x,y]_1\big)-[\widetilde{\Psi}(x),\widetilde{\Psi}(y)]_2\\
&=\widetilde{\Psi}\big(\overline{[x,y]}+\varphi_1(x,y)\xi\big)-[\Phi(x)+\rho(x)\xi,\Phi(y)+\varrho(y)\xi]_2\\
&=\Phi\big(\overline{[x,y]}\big)+\varrho\big(\overline{[x,y]}\big)\xi+\varphi_1(x,y)(v+a\xi)-\overline{[\Phi(x),\Phi(y)]}-\varphi_2\big(\Phi(x),\Phi(y)\big)\xi\\
&\hspace{0.3cm}+\varrho(y)\chi_2(\Phi(x))\xi-\varrho(x)\chi_2(\Phi(y))\xi\\
&=\big(\Delta\Phi\big)(x,y)+\varphi_1(x,y)v+\big(a\varphi_1(x,y)-\Phi^\ast\varphi_2(x,y)-\big(\varrho\otimes\Phi^\ast\chi_2\big)(x,y)-\md\varrho(x,y)\big)\xi.
\end{align*}
On the other hand, 
\begin{align*}
\tilde{u}\big(\sigma\otimes\tilde{\mu}_1\big)(x,y)&=(u+u_\xi\xi)(\sigma(x)\mu_1(y)-\sigma(y)\mu_1(x))\\
&=u\big(\sigma\otimes\mu_1\big)(x,y)+u_\xi\big(\sigma\otimes\mu_1\big)(x,y)\xi,
\end{align*}
and
\begin{align*}
\sigma\otimes(\widetilde{\Psi}\circ
\widetilde{\mathrm{D}}_1)(x,y)&=\sigma(x)\widetilde{\Psi}\big(
\widetilde{\mathrm{D}}_1(y)\big)-\sigma(y)\widetilde{\Psi}\big(
\widetilde{\mathrm{D}}_1(x)\big)\\
&=\sigma(x)\widetilde{\Psi}\big(D_1(y)+\lambda_1(y)\xi\big)-\sigma(y)\widetilde{\Psi}\big(D_1(x)+\lambda_1(x)\xi\big)\\
&=\sigma(x)\big(\Phi(\mathrm{D}_1(y))+\varrho(\mathrm{D}_1(y))\xi+\lambda_1(y)(v+a\xi)\big)\\
&\hspace{0.3cm}-\sigma(y)\big(\Phi(\mathrm{D}_1(x))+\varrho(\mathrm{D}_1(x))\xi+\lambda_1(x)(v+a\xi)\big)\\
&=\big(\sigma\otimes\Phi\circ\mathrm{D}_1\big)(x,y)+v\big(\sigma\otimes\lambda_1\big)(x,y)+\Big(\big(\sigma\otimes\varrho\circ\mathrm{D}_1\big)(x,y)+a\big(\sigma\otimes\lambda_1\big)(x,y)\Big)\xi,
\end{align*}
moreover,
\begin{align*}
\sigma\otimes(\mathrm{ad}_{\tilde{u}}\circ\widetilde{\Psi})(x,y)&=\sigma(x)\mathrm{ad}_{\tilde{u}}\big(\widetilde{\Psi}(y)\big)-\sigma(y)\mathrm{ad}_{\tilde{u}}\big(\widetilde{\Psi}(x)\big)\\
&=\sigma(x)\mathrm{ad}_{\tilde{u}}\big(\Phi(y)+\varrho(y)\xi\big)-\sigma(y)\mathrm{ad}_{\tilde{u}}\big(\Phi(x)+\varrho(x)\xi\big)\\
&=\sigma(x)\mathrm{ad}_{u+u_\xi\xi}\big(\Phi(y)\big)+\sigma(x)\varrho(y)\mathrm{ad}_{u+u_\xi\xi}\xi-\sigma(y)\mathrm{ad}_{u+u_\xi\xi}\big(\Phi(x)\big)-\sigma(y)\varrho(x)\mathrm{ad}_{u+u_\xi\xi}\xi\\
&=\sigma(x)\overline{\mathrm{ad}}_{u}\circ\Phi(y)+\sigma(x)\varphi_2\big(u,\Phi(y)\big)\xi+u_\xi\sigma(x)\chi_2\big(\Phi(y)\big)\xi-\sigma(x)\varrho(y)\chi_2(u)\xi\\
&\hspace{0.3cm}-\sigma(y)\overline{\mathrm{ad}}_{u}\circ\Phi(x)-\sigma(y)\varphi_2\big(u,\Phi(x)\big)\xi-u_\xi\sigma(y)\chi_2\big(\Phi(x)\big)\xi+\sigma(y)\varrho(x)\chi_2(u)\xi\\
&=\big(\sigma\otimes\overline{\mathrm{ad}}_{u}\circ\Phi)(x,y)+\Big(\big(\sigma\otimes\iota_u
\varphi_2\circ\Phi\big)(x,y)+u_\xi\big(\sigma\otimes\Phi^\ast\chi_2\big)(x,y)-\chi_2(u)\big(\sigma\otimes\varrho\big)(x,y)\Big)\xi.
\end{align*}
Therefore,
\begin{align*}
\theta\Big(\big(\Delta\Phi\big)+v\varphi_1\Big)&=u(\sigma\otimes\mu_1)+\sigma\otimes\Phi\circ\mathrm{D}_1+v\big(\sigma\otimes\lambda_1\big)
-\sigma\otimes\overline{\mathrm{ad}}_{u}\circ\Phi,
\end{align*}
and
\begin{align*}
\theta\Big(a\varphi_1-\Phi^\ast\varphi_2-\varrho\otimes\Phi^\ast\chi_2-\md\varrho\Big)&=u_\xi\big(\sigma\otimes\mu_1-\sigma\otimes\Phi^\ast\chi_2\big)
+\sigma\otimes\varrho\circ\mathrm{D}_1+a\big(\sigma\otimes\lambda_1\big)
+\chi_2(u)\big(\sigma\otimes\varrho\big)-\sigma\otimes\iota_u
\varphi_2\circ\Phi.
\end{align*}
Now, for all $x\in\overline{\G}$, we have
\begin{align*}
\Delta\widetilde{\Psi}(x,\xi)&=\widetilde{\Psi}\big([x,\xi]_1\big)-[\widetilde{\Psi}(x),\widetilde{\Psi}(\xi)]_2\\
&=-\widetilde{\Psi}\big(\chi_1(x)\xi\big)-[\Phi(x)+\varrho(x)\xi,v+a\xi]_2\\
&=-\chi_1(x)(v+a\xi)+\overline{\mathrm{ad}_v}\circ\Phi(x)-\varphi_2\big(\Phi(x),v\big)\xi+a\chi_2\big(\Phi(x)\big)\xi-\varrho(x)\chi_2(v)\xi\\
&=-\chi_1(x)v+\overline{\mathrm{ad}_v}\circ\Phi(x)+\Big(a\Phi^\ast\chi_2(x)-a\chi_1(x)+\iota_v\varphi_2\circ\Phi(x)-\chi_2(v)\varrho(x)\Big)\xi,
\end{align*}
and
\begin{align*}
\tilde{u}\big(\sigma\otimes\tilde{\mu}_1\big)(x,\xi)&=(u+u_\xi\xi)\big(\sigma(x)\tilde{\mu}_1(\xi)-\sigma(\xi)\tilde{\mu}_1(x)\big)\\
&=(u+u_\xi\xi)\big(\kappa_1\sigma(x)-\sigma_\xi\mu_1(x)\big)\\
&=u\big(\kappa_1\sigma(x)-\sigma_\xi\mu_1(x)\big)+u_\xi\big(\kappa_1\sigma(x)-\sigma_\xi\mu_1(x)\big)\xi.
\end{align*}
In addition,
\begin{align*}
\sigma\otimes(\widetilde{\Psi}\circ
\widetilde{\mathrm{D}}_1)(x,\xi)&=\sigma(x)\widetilde{\Psi}\big(w_1+t_1\xi\big)-\sigma(\xi)\widetilde{\Psi}\big(\mathrm{D}_1(x)+\lambda_1(x)\xi\big)\\
&=\sigma(x)\big(\Phi(w_1)+\varrho(w_1)\xi+t_1(v+a\xi)\big)-\sigma_\xi\big(\Phi(\mathrm{D}_1(x)\big)+\varrho(\mathrm{D}_1(x))\xi+\lambda_1(x)(v+a\xi)\big)\\
&=\sigma(x)\big(\Phi(w_1)+t_1v)-\sigma_\xi\big(\Phi(\mathrm{D}_1(x))+\lambda_1(x)v\big)+\Big(\sigma(x)\big(\varrho(w_1)+at_1\big)+\sigma_\xi\big(\varrho(\mathrm{D}_1(x))+a\lambda_1(x)\big)\Big)\xi,
\end{align*}
moreover,
\begin{align*}
\sigma\otimes(\mathrm{ad}_{\tilde{u}}\circ\widetilde{\Psi})(x,\xi)&=\sigma(x)\mathrm{ad}_{\tilde{u}}\circ\widetilde{\Psi}(\xi)-\sigma(\xi)\mathrm{ad}_{\tilde{u}}\circ\widetilde{\Psi}(x)\\
&=\sigma(x)\mathrm{ad}_{u+u_\xi\xi}(v+a\xi)-\sigma_\xi\mathrm{ad}_{u+u_\xi\xi}\big(\Phi(x)+\varrho(x)\xi)\big)\\
&=\sigma(x)\big(\overline{\mathrm{ad}}_uv+\varphi_2(u,v)\xi+u_\xi\chi_2(v)\xi-a\chi_2(u)\xi\big)-\sigma_\xi\big(\overline{\mathrm{ad}}_u\circ\Phi(x)+\varphi_2(u,\Phi(x))\xi\big)\\
&\hspace{0.3cm}-\sigma_\xi\big(u_\xi
\chi_2(\Phi(x))\xi-\varrho(x)\chi_2(u)\xi\big)\\
&=\sigma(x)\overline{\mathrm{ad}}_uv-\sigma_\xi\overline{\mathrm{ad}}_u\circ\Phi(x)\\
&\hspace{0.3cm}+\Big(\sigma(x)\big(\varphi_2(u,v)+u_\xi\chi_2(v)-a\chi_2(u)\big)-\sigma_\xi\big(\varphi_2(u,\Phi(x))+u_\xi
\chi_2(\Phi(x))-\varrho(x)\chi_2(u)\big)\Big)\xi.
\end{align*}
Thus,
\begin{align*}
\theta\Big(-v\chi_1+\overline{\mathrm{ad}_v}\circ\Phi\Big)&=u\big(\kappa_1\sigma-\sigma_\xi\mu_1\big)+\big(\Phi(w_1)+t_1v)\sigma-\sigma_\xi\big(\Phi\circ\mathrm{D}_1+v\lambda_1\big)-\overline{\mathrm{ad}}_uv\sigma+\sigma_\xi\overline{\mathrm{ad}}_u\circ\Phi,
\end{align*}
and
\begin{align*}
\theta\Big(a\Phi^\ast\chi_2-a\chi_1+\iota_v\varphi_2\circ\Phi-\chi_2(v)\varrho\Big)&=u_\xi\big(\kappa_1\sigma-\sigma_\xi\mu_1\big)+\big(\varrho(w_1)+at_1\big)\sigma+\sigma_\xi\big(\varrho\circ\mathrm{D}_1+a\lambda_1\big)\\
&\hspace{0.3cm}-\big(\varphi_2(u,v)+u_\xi\chi_2(v)-a\chi_2(u)\big)\sigma+\sigma_\xi\big(\iota_u\varphi_2
\circ\Phi+u_\xi
\Phi^\ast\chi_2-\chi_2(u)\varrho\big).
\end{align*}
Now, we develop Equation~$(\ref{CharactC2})$. Then,  for all $x,y\in\overline{\G}$
\begin{align*}
-\md\sigma(x,y)&=\sigma\big([x,y]_1\big)\\
&=\sigma\big(\overline{[x,y]}+\varphi_1(x,y)\xi\big)\\
&=-\md\sigma(x,y)+\varphi_1(x,y)\sigma_\xi,
\end{align*}
and
\begin{align*}
\theta\big(
\sigma\otimes\tilde{\mu}_1\big)(x,y)+\big(\sigma\otimes\sigma\circ
\widetilde{\mathrm{D}}_1\big)(x,y)&=\theta\big(\sigma(x)\tilde{\mu}_1(y)-\sigma(y)\tilde{\mu}_1(x)\big)+\sigma(x)\sigma\big(\widetilde{\mathrm{D}}_1(y)\big)-\sigma(y)\sigma\big(\widetilde{\mathrm{D}}_1(x)\big)\\
&=\theta\big(\sigma(x)\mu_1(y)-\sigma(y)\mu_1(x)\big)+\sigma(x)\big(\sigma(\mathrm{D}_1(y)+\lambda_1(y)\xi)\big)\\
&\hspace{0.3cm}-\sigma(y)\big(\sigma(\mathrm{D}_1(x)+\lambda_1(x)\xi)\big)\\
&=\theta\big(\sigma(x)\mu_1(y)-\sigma(y)\mu_1(x)\big)+\sigma(x)\big(\sigma(\mathrm{D}_1(y))+\sigma_\xi\lambda_1(y))\big)\\
&\hspace{0.3cm}-\sigma(y)\big(\sigma(\mathrm{D}_1(x))+\sigma_\xi\lambda_1(x))\big).
\end{align*}
As a result, we obtain
\begin{align*}
\theta\big(-\md\sigma+\sigma_\xi\varphi_1\big)&=\theta\big(\sigma
\otimes\mu_1\big)+\sigma\otimes\sigma\circ\mathrm{D}_1+\sigma_\xi\big(\sigma
\otimes\lambda_1\big).
\end{align*}
For all $x\in\overline{\G}$, we obtain
\begin{align*}
-\md\sigma(x,\xi)&=\sigma\big([x,\xi]_1\big)=-\sigma_\xi\chi_1(x),
\end{align*}
and
\begin{align*}
\theta\big(
\sigma\otimes\tilde{\mu}_1\big)(x,\xi)+\big(\sigma\otimes\sigma\circ
\widetilde{\mathrm{D}}_1\big)(x,\xi)&=\theta\big(\sigma(x)\tilde{\mu}_1(\xi)-\sigma(\xi)\tilde{\mu}_1(x)\big)+\sigma(x)\sigma\big(\widetilde{\mathrm{D}}_1(\xi)\big)-\sigma(\xi)\sigma\big(\widetilde{\mathrm{D}}_1(x)\big)\\
&=\theta\big(\kappa_1\sigma(x)-\sigma_\xi\mu_1(x)\big)+\sigma(x)\big(\sigma(w_1+t_1\xi)\big)-\sigma_\xi\sigma\big(\mathrm{D}_1(x)+\lambda_1(x)\xi\big)\\
&=\theta\big(\kappa_1\sigma(x)-\sigma_\xi\mu_1(x)\big)+\sigma(x)\big(\sigma(w_1)+t_1\sigma_\xi\big)-\sigma_\xi\big(\sigma(\mathrm{D}_1(x))+\sigma_\xi\lambda_1(x)\big).
\end{align*}
Hence,
\begin{align*}
-\sigma_\xi\chi_1&=\theta\big(
\kappa_1\sigma-\sigma_\xi\mu_1\big)+\big(\sigma(w_1)+t_1\sigma_\xi\big)\sigma-\sigma_\xi\big(\sigma\circ\mathrm{D}_1+\sigma_\xi\lambda_1\big).
\end{align*}
We now turn our attention to the first condition of Proposition~$\ref{Characterisationgeneralized}$, which can be expressed as
\begin{align*}
\theta\widetilde{\mathrm{D}}_2\circ\widetilde{\Psi}-\widetilde{\mathrm{D}}_2(\tilde{u})\sigma&=\widetilde{\Psi}\circ\widetilde{\mathrm{D}}_1+\tilde{u}\tilde{\mu}_1-\mathrm{ad}_{\tilde{u}}\circ\widetilde{\Psi}.
\end{align*}
For all $x \in \overline{\G}$, we have
\begin{align*}
\theta\widetilde{\mathrm{D}}_2\circ\widetilde{\Psi}(x)-\widetilde{\mathrm{D}}_2(\tilde{u})\sigma(x)&=\theta\widetilde{\mathrm{D}}_2(\Phi(x)+\varrho(x)\xi)-\big(\widetilde{\mathrm{D}}_2(u)+u_\xi\widetilde{\mathrm{D}}_2(\xi)\big)\sigma(x)\\
&=\theta\big(\mathrm{D}_2(\Phi(x))+\lambda_2(\Phi(x))\xi+\varrho(x)(w_2+t_2\xi)\big)-\big(\mathrm{D}_2(u)+\lambda_2(u)\xi+u_\xi(w_2+t_2\xi)\big)\sigma(x)\\
&=\Big(\theta\big(\mathrm{D}_2(\Phi(x))+\varrho(x)w_2\big)-\big(\mathrm{D}_2(u)+u_\xi w_2\big)\sigma(x)\Big)+\Big(\theta\big(\lambda_2(\Phi(x))+t_2\varrho(x)\big)-\big(\lambda_2(u)+t_2u_\xi\big) \sigma(x)\Big)\xi,
\end{align*}
and
\begin{align*}
\widetilde{\Psi}\circ\widetilde{\mathrm{D}}_1(x)+\tilde{u}\tilde{\mu}_1(x)-\mathrm{ad}_{\tilde{u}}\circ\widetilde{\Psi}(x)&=\widetilde{\Psi}\big(\mathrm{D}_1(x)+\lambda_1(x)\xi\big)+(u+u_\xi\xi)\mu_1(x)-\mathrm{ad}_{u+u_\xi\xi}\big(\Phi(x)+\varrho(x)\xi\big)\\
&=\Phi(\mathrm{D}_1(x))+\varrho(\mathrm{D}_1(x))\xi+\lambda_1(x)(v+a\xi)+(u+u_\xi\xi)\mu_1(x)\\
&\hspace{0.3cm}-\overline{\mathrm{ad}}_u\circ\Phi(x)-\varphi_2\big(u,\Phi(x)\big)\xi-u_\xi\chi_2\big(\Phi(x)\big)\xi+\varrho(x)\chi_2(u)\xi.
\end{align*}
Therefore,
\begin{align*}
\theta\big(\mathrm{D}_2\circ\Phi+w_2\varrho\big)-\big(\mathrm{D}_2(u)+u_\xi w_2\big)\sigma&=\Phi\circ\mathrm{D}_1+v\lambda_1+u\mu_1-\overline{\mathrm{ad}}_u\circ\Phi,
\end{align*}
and
\begin{align*}
\theta\big(\Phi^\ast\lambda_2+t_2\varrho\big)-\big(\lambda_2(u)+t_2u_\xi\big) \sigma&=\varrho\circ\mathrm{D}_1+a\lambda_1+u_\xi\mu_1-\iota_u\varphi_2\circ\Phi-u_\xi\Phi^\ast\chi_2+\chi_2(u)\varrho.
\end{align*}
Additionally, we have
\begin{align*}
\theta\widetilde{\mathrm{D}}_2\circ\widetilde{\Psi}(\xi)-\widetilde{\mathrm{D}}_2(\tilde{u})\sigma(\xi)&=\theta\widetilde{\mathrm{D}}_2(v+a\xi)-\sigma_\xi\widetilde{\mathrm{D}}_2(u+u_\xi\xi)\\
&=\theta\big(\mathrm{D}_2(v)+\lambda_2(v)\xi+a(w_2+t_2\xi)\big)-\sigma_\xi\big(\mathrm{D}_2(u)+\lambda_2(u)\xi+u_\xi(w_2+t_2\xi)\big)\\
&=\theta\big(\mathrm{D}_2(v)+aw_2\big)-\sigma_\xi\big(\mathrm{D}_2(u)+u_\xi w_2\big)+\Big(\theta\big(\lambda_2(v)+at_2\big)-\sigma_\xi\big(\lambda_2(u)+u_\xi t_2\big)\Big)\xi,
\end{align*}
and
\begin{align*}
\widetilde{\Psi}\circ\widetilde{\mathrm{D}}_1(\xi)+\tilde{u}\tilde{\mu}_1(\xi)-\mathrm{ad}_{\tilde{u}}\circ\widetilde{\Psi}(\xi)&=\widetilde{\Psi}\big(w_1+t_1\xi\big)+\kappa_1(u+u_\xi\xi)-\mathrm{ad}_{u+u_\xi\xi}(v+a\xi)\\
&=\Phi(w_1)+\varrho(w_1)\xi+t_1(v+a\xi)+\kappa_1(u+u_\xi\xi)-\overline{\mathrm{ad}}_uv-\varphi_2(u,v)\xi-u_\xi\chi_2(v)\xi+a\chi_2(u)\xi\\
&=\Phi(w_1)+t_1v+\kappa_1u-\overline{\mathrm{ad}}_uv+\big(\varrho(w_1)+t_1a+\kappa_1u_\xi-\varphi_2(u,v)-u_\xi\chi_2(v)+a\chi_2(u)\big)\xi.
\end{align*}
Thus,
\begin{align*}
\theta\big(\mathrm{D}_2(v)+aw_2\big)-\sigma_\xi\big(\mathrm{D}_2(u)+u_\xi w_2\big)&=\Phi(w_1)+t_1v+\kappa_1u-\overline{\mathrm{ad}}_uv,
\end{align*}
and
\begin{align*}
\theta\big(\lambda_2(v)+at_2\big)-\sigma_\xi\big(\lambda_2(u)+u_\xi t_2\big)&=\varrho(w_1)+t_1a+\kappa_1u_\xi-\varphi_2(u,v)-u_\xi\chi_2(v)+a\chi_2(u).
\end{align*}
Going further, the second condition of Proposition~$\ref{Characterisationgeneralized}$, which can be written as
\begin{align*}
\theta\tilde{\mu}_2\circ\widetilde{\Psi}-\tilde{\mu}_2(\tilde{u})\sigma&=\theta\tilde{\mu}_1+\sigma\circ\widetilde{\mathrm{D}}_1.
\end{align*}
So, for all $x\in\overline{\G}$, we have
\begin{align*}
\theta\tilde{\mu}_2\circ\widetilde{\Psi}(x)-\tilde{\mu}_2(u)\sigma(x)&=\theta\tilde{\mu}_2(\Phi(x)+\varrho(x)\xi)-(\mu_2(u)+\kappa_2u_\xi)\sigma(x)\\
&=\theta\mu_2(\Phi(x))+\theta\varrho(x)\kappa_2-(\mu_2(u)+\kappa_2u_\xi)\sigma(x),
\end{align*}
and
\begin{align*}
\theta\tilde{\mu}_1(x)+\sigma\circ\widetilde{\mathrm{D}}_1(x)&=\theta\mu_1(x)+\sigma\big(\mathrm{D}_1(x)+\lambda_1(x)\xi\big)\\
&=\theta\mu_1(x)+\sigma\big(\mathrm{D}_1(x)\big)+\sigma_\xi\lambda_1(x).
\end{align*}
Hence,
\begin{align*}
\theta\Phi^\ast\mu_2+\theta\kappa_2\varrho-(\mu_2(u)+\kappa_2u_\xi)\sigma&=\theta\mu_1+\sigma\circ\mathrm{D}_1+\sigma_\xi\lambda_1.
\end{align*}
In the end, we have
\begin{align*}
\theta\tilde{\mu}_2\circ\widetilde{\Psi}(\xi)-\tilde{\mu}_2(\tilde{u})\sigma(\xi)&=\theta\tilde{\mu}_2(v+a\xi)-\sigma_\xi(\mu_2(u)+\kappa_2u_\xi)\\
&=\theta\mu_2(v)+a\theta\kappa_2-\sigma_\xi(\mu_2(u)+\kappa_2u_\xi),
\end{align*}
and
\begin{align*}
\theta\tilde{\mu}_1(\xi)+\sigma\circ\widetilde{\mathrm{D}}_1(\xi)&=\theta\kappa_1+\sigma\big(w_1+t_1\xi\big)\\
&=\theta\kappa_1+\sigma(w_1)+t_1\sigma_\xi.
\end{align*}
Consequently,
\begin{align*}
\theta\mu_2(v)+a\theta\kappa_2-\sigma_\xi(\mu_2(u)+\kappa_2u_\xi)&=\theta\kappa_1+\sigma(w_1)+t_1\sigma_\xi.
\end{align*}

\end{proof}

Considering the Lie bracket structure of $\mathrm{D}$-extensions of Lie algebras defined in (\ref{BracketDsym}) under the given hypotheses, we examine the case where $w = 0$ and $\kappa = 0$. According to Proposition~$\ref{Condisymp}$, the $\mathrm{D}$-extension of a non-central extension is symplectic if and only if $\varphi = \overline{\omega}$ and $\chi = -\mu$. This case precisely corresponds to the generalized symplectic oxidation of Lie algebras studied in Section~\ref{se2}. We thus obtain the following characterization for isomorphisms of generalized symplectic oxidations:

\begin{co}\label{Morerest}
	Let $(\G_{(\mathrm{D}_1,\mu_1,\lambda_1,t_1)},\om_1)$ and $(\G_{(\mathrm{D}_2,\mu_2,\lambda_2,t_2)},\om_2)$  be two generalized symplectic oxidation  of $(\overline{\G}_1,\overline{\omega}_1)$ and $(\overline{\G}_2,\overline{\omega}_1)$ respectively. Then $\G_{(\mathrm{D}_1,\mu_1,\lambda_1,t_1)}$ and $\G_{(\mathrm{D}_2,\mu_2,\lambda_2,t_2)}$  are  isomorphic if and only if there exist  an endomorphism $\Phi:\overline{\G}\longrightarrow\overline{\G}$, $u,v\in\overline{\G}$, $\varrho,\sigma\in\overline{\G}^\ast$ and $a,\sigma_\xi,u_\xi,\theta\in\R$ such that
\begin{enumerate}
\item  $\theta\mathrm{D}_2\circ\Phi-\mathrm{D}_2(u)\sigma=\Phi\circ\mathrm{D}_1+v\lambda_1+u\mu_1-\overline{\mathrm{ad}}_u\circ\Phi$,
\item $\theta\big(\Phi^\ast\lambda_2+t_2\varrho\big)-\big(\lambda_2(u)+t_2u_\xi\big) \sigma=\varrho\circ\mathrm{D}_1+a\lambda_1+u_\xi\mu_1-\iota_u\varphi_2\circ\Phi+u_\xi\Phi^\ast\mu_2-\mu_2(u)\varrho$,
\item $\theta\mathrm{D}_2(v)-\sigma_\xi\mathrm{D}_2(u)=t_1v-\overline{\mathrm{ad}}_uv$,
\item $\theta\big(\lambda_2(v)+at_2\big)-\sigma_\xi\big(\lambda_2(u)+u_\xi t_2\big)=t_1a-\varphi_2(u,v)+u_\xi\mu_2(v)-a\mu_2(u)$,
\item $\theta\Phi^\ast\mu_2-\mu_2(u)\sigma=\theta\mu_1+\sigma\circ\mathrm{D}_1+\sigma_\xi\lambda_1$,
\item $\theta\mu_2(v)-\sigma_\xi\mu_2(u)=t_1\sigma_\xi$. 
\end{enumerate}
Here, this data verifies the hypothesis given in Theorem~$\ref{Moreconditions}$.
\end{co}
\subsection{Generalized symplectic Lie algberas of dimension $8$}

Following our previous classifications of central and normal oxidations, we outline the major ideas behind our classification of generalized symplectic oxidations.

Given a generalized symplectic oxidation $\G_{\mathrm{D},\mu,\lambda,t}$ of the 
unique six-dimensional irreducible symplectic Lie algebra $\G_{(1,0,0,1)}$ with respect to 
the data $\mathrm{D},\mu,\lambda$, and $t$ (whose Lie bracket structure is described 
in \eqref{Brackestgenera} and  satisfy the conditions of Proposition~$\ref{prgeneralized}$), the first step consists of solving 
the associated linear system derived from $\mathrm{D},\mu,\lambda$, and $t$ that 
corresponds to the four conditions of Proposition~$\ref{prgeneralized}$. We prove that this system admits 
four distinct solutions (see Lemma~$\ref{Endomor}$ in the Appendix~$\ref{Appen}$).

Next, using Corollary~$\ref{Classifynoncentral}$, we classify the non-central extensions of the algebra $\G_{(1,0,0,1)}$. Lemma~$\ref{cohomonocentral}$ shows that there exist three non-isomorphic non-central extensions. Additionally, we determine the isomorphism classes of generalized symplectic oxidations using Corollary~\ref{Morerest}. Further details are provided in the proof of the following result.

\begin{pr}\label{Classigeneral}
Let $(\G,\omega)$ be an eight-dimensional generalized symplectic oxidation. Then, $(\G,\omega)$ is symplectomorphically isomorphic to one of the following symplectic Lie algebras$:$
\begin{align*}
\G_{8,1}:&~~[e_1, e_5] = e_2, ~~\quad 
		[e_2, e_5] = -e_1, \quad
		[e_3, e_6] = e_4, \quad 
		[e_4, e_6] = -e_3,\quad [e_7, e_5] = -\mu_1 e_7\\
		&~~[e_7, e_6] = -\mu_2 e_7,[e_8,e_5]=\mu_1 e_8,\hspace{0.23cm} [e_8,e_6]=\mu_2 e_8\\
		&\hspace{0.254cm}\omega=e^{12}+e^{34}+\eta e^{56}+e^{78},\quad\eta\in\R^{\ast+},~\mu_1,\mu_2\in\R^\ast.\\
		&\\
		\G_{8,2}:&~~[e_1, e_5] = e_2, ~~\quad 
		[e_2, e_5] = -e_1, \quad
		[e_3, e_6] = e_4, \quad 
		[e_4, e_6] = -e_3,\quad [e_7, e_5] = -\mu e_7\\
		&~~[e_8,e_5]=\mu e_8\\
		&\hspace{0.254cm}\omega=e^{12}+e^{34}+\eta e^{56}+e^{78},\quad\eta\in\R^{\ast+},~\mu\in\R^\ast.\\
		&\\
\G_{8,3}:&~~[e_1, e_5] = e_2, ~~\quad 
		[e_2, e_5] = -e_1, \quad
		[e_3, e_6] = e_4, \quad 
		[e_4, e_6] = -e_3,\quad [e_7, e_6] = -\mu e_7\\
		&~~[e_8,e_6]=-\mu e_8\\
		&\hspace{0.254cm}\omega=e^{12}+e^{34}+\eta e^{56}+e^{78},\quad\eta\in\R^{\ast+},~\mu\in\R^\ast.
\end{align*}
\end{pr}
\begin{proof}
Let $\G_{(\mathrm{D}_1,\mu_1,\lambda_1,t_1)}$ and $\G_{(\mathrm{D}_2,\mu_2,\lambda_2,t_2)}$ be two generalized symplectic oxidations of $\overline{\G}=\G_{(1,0,0,1)}$. Suppose that $\Upsilon:\G_{(\mathrm{D}_1,\mu_1,\lambda_1,t_1)}\longrightarrow\G_{(\mathrm{D}_2,\mu_2,\lambda_2,t_2}$ is an isomorphism of generalized symplectic oxidation. Then, $\Upsilon$ can be expressed as
\begin{align*}
\Upsilon(x)&=\Phi(x)+\varrho(x)\xi+\sigma(x)\ell,& \text{for all } x\in\overline{\G},~~\varrho,~~\sigma\in\overline{\G}^\ast,\\
\Upsilon(\xi)&=v+a\xi+\sigma_\xi\ell,& v\in\overline{\G},~~a,~~\sigma_\xi\in\R,\\
\Upsilon(\ell)&=u+u_\xi\xi+\theta\ell, & u\in\overline{\G},~~u_\xi,~~\theta\in\R.
\end{align*}
By Proposition~$\ref{isoofnocentral}$, and
since $\G_{(1,0,0,1)}$ has no one-dimensional ideal, we can assume that $v=0$. Since we have already classified the non-central extensions of Lie algebras in Lemma~\ref{cohomonocentral}, it is enough to consider $\Upsilon$ under the restrictions $v = 0$ and $\sigma \equiv 0$. Where $\varrho$ is a special one-form given in Lemma~$\ref{cohomonocentral}$. Therefore, $\Upsilon$ becomes
\begin{align*}
\Upsilon(x)&=\Phi(x)+\varrho_j(x)\xi,& \text{for all } x\in\overline{\G},~~\varrho_j\in\overline{\G}^\ast,\\
\Upsilon(\xi)&=a\xi,& a\in\R^\ast,\\
\Upsilon(\ell)&=u+u_\xi\xi+\theta\ell, & u\in\overline{\G},~~u_\xi\in\R,~~\theta\in\R^\ast.
\end{align*}
where $\varrho_j \in \G_{(1,0,0,1)}^\ast$ is a one-form corresponding to each case of non-central extensions classified in Lemma~$\ref{cohomonocentral}$. It follows from the previous discussions and Corollary~$\ref{Morerest}$, that $\G_{(\mathrm{D}_1,\mu_1,\lambda_1,t_1)}$ and $\G_{(\mathrm{D}_2,\mu_2,\lambda_2,t_2)}$ are isomorphic if and only if
\begin{enumerate}
\item[$(i)$] $\mathrm{D}_2=\frac{1}{\theta}\left(\Phi\circ\mathrm{D}_1\circ\Phi^{-1}+u\mu_1\circ\Phi^{-1}-\overline{\mathrm{ad}}_u\right)$,
\item[$(ii)$]$\lambda_2=\frac{1}{\theta}\big(\Phi^\ast\big)^{-1}\left(\varrho\circ\mathrm{D}_1+a\lambda_1+2u_\xi\mu_1-t_1\varrho\right)$,
\item[$(iii)$]$t_2=\frac{1}{\theta}\left(t_1-\big(\Phi^\ast\big)^{-1}\mu_1(u)\right)$,
\item[$(iv)$]$\mu_2=\big(\Phi^\ast\big)^{-1}\mu_1$.
\end{enumerate}
Since, $\Phi^\ast\big(\iota_u\omega_{\eta,\mathrm{D}_j}\big)=0$ for all $j=1,\ldots,4$ (see, Lemma~$\ref{cohomonocentral}$ in  Appendix~$\ref{Appen}$), where $\Phi\in\mathrm{Aut}(\G_{(1,0,0,1)})$. Observe first that, for all $j=1,\ldots,4$, there exist $\Phi\in\mathrm{Aut}(\overline{\G})$ and $u_j\in\overline{\G}$ such that
\begin{align*}
\Phi\circ\mathrm{D}_j\circ\Phi^{-1}+u_j\mu_j\circ\Phi^{-1}-\overline{\mathrm{ad}}_{u_j}=0\quad\text{and}\quad t_1-\big(\Phi^\ast\big)^{-1}\mu_1(u_j)=0.
\end{align*}
For an illustration,  we can take any $\Phi\in\mathrm{Aut}(\G_{(1,0,0,1)})$. In partilular, $\Phi=\mathbf{I}$, and thus 
\begin{align*}
u_1&= {\tfrac {1}{\mu_{6}}}\left(d_{25}\,\mu_{6}-d_{26}\,\mu_{5}\right) e_1-{\tfrac{1}{\mu_{6}}}d_{26}e_2-{\tfrac {1}{\mu_{5}}}\left(d_{45}\,\mu_{6}-d_{46}\,
\mu_{5}\right) e_3-{\tfrac {1}{\mu_{5}}}d_{45}e_4-d_{21}e_5-d_{43}e_6,
\\
u_2&=-\tfrac{1}{d_{21}(\mu_5^2 + 1)}(d_{15}d_{21}\mu_5 - \lambda_1)e_1 -\tfrac{1}{d_{21}(\mu_5^2 + 1)}(d_{15}d_{21} + \mu_5\lambda_1)+ d_{46}e_3-\tfrac{1}{\mu_5}d_{45}e_4 -d_{21}e_5 - d_{43}e_6,\\
u_3&= -\tfrac{1}{(\mu_5^2 + 1)}(d_{15}\mu_5 - d_{25})e_1 -\tfrac{1}{(\mu_5^2 + 1)}(d_{25}\mu_5 + d_{15})+ d_{46}e_3 -\tfrac{1}{\mu_5} d_{45}e_4 -d_{43}e_6,\\
u_4&=d_{25}e_1 -d_{15}e_2-\tfrac{1}{(\mu_6^2 + 1)}(d_{36}\mu_6 - d_{46})e_3-\tfrac{1}{(\mu_6^2 + 1)} (d_{46}\mu_6 + d_{36})e_4-d_{21}e_5-d_{43}e_6.
\end{align*}
Moreover, for all $j=1,2,3,4$, we have
\begin{align*}
\varrho_j\circ\mathrm{D}_j+a\lambda_j+2u_\xi\mu_j-t_j\varrho_j=0.
\end{align*}

Finally, for each generalized symplectic oxidation $(\G_{(\mathrm{D}_j,\mu_j,\lambda_j,t_j)},\omega)$, whose Lie brackets are given in Appendix~\ref{Appen}, consider the symplectomorphism $\Upsilon$ mapping to the Lie algebras listed below:
\begin{align*}
\Upsilon(x)&=x+\varrho_j(x)\xi,& \text{for all } x\in\overline{\G},~~\varrho_j\in\overline{\G}^\ast,\\
\Upsilon(\xi)&=a\xi,& a\in\R^\ast,\\
\Upsilon(\ell)&=u+u_\xi\xi+\theta\ell, & u\in\overline{\G},~~u_\xi\in\R,~~\theta\in\R^\ast.
\end{align*}
Consequently, we complete the classification of eight-dimensional non completely reducible Lie algebras.

\end{proof}

\begin{co}
Every  eight-dimensional non completely reducible symplectic Lie algebra is $2$-step solvable.
\end{co}

The following result (Theorem~$\ref{Clssiformnoncomp}$) holds for every eight-dimensional non completely reducible symplectic Lie algebras that reduce to the irreducible symplectic Lie algebra $\G_{(1,0,0,1)}$. However, Theorem~$\ref{Clssiformnoncomp}$  does not hold for  eight-dimensional irreducilbe symplectic Lie algebras (see,  Appendix~$\ref{Appen}$, Proof of Proposition~$\ref{Classi8forms}$, algebra $\G^{(1,0)}$).
\begin{theo}\label{Clssiformnoncomp}
Every symplectic form on eight-dimensional non completely reducible symplectic Lie algebra, gives rise to a non completely reducible  symplectic Lie
algebra.
\end{theo}
\begin{proof}
Let $(\G,\omega)$ be a symplectic Lie algebra and let $(\overline{\G}, \overline{\omega})$ be its symplectic reduction with respect to a one-dimensional isotropic ideal $\mathfrak{j}$. Since $\mathfrak{j}$ is one-dimensional, for any symplectic form $\Omega$ on $\G$, the ideal $\mathfrak{j}$ is isotropic in $(\G, \Omega)$. Let $\Omega \in \bigwedge^2 \G^\ast$ be an arbitrary symplectic form in general position. Consider now the symplectic reduction $(\overline{\G}, \overline{\omega}_\Omega)$ with respect to $\mathfrak{j}$.

As established in Section~\ref{se2}, the Lie brackets of $\G$ can be expressed as in Equation~$(\ref{Brackestgenera})$ under the hypotheses of Proposition~$\ref{prgeneralized}$, with the symplectic form decomposing as
\[
\Omega = \overline{\omega}_\Omega + \xi^\ast \wedge \ell^\ast,
\]
where $\xi^\ast$ and $\ell^\ast$ are the dual forms to $\xi$ and $\ell$ respectively. Let $\mathcal{S}(\overline{\G})$ denote the set of all symplectic forms on $\overline{\G}$ modulo  symplectomorphism. Then there exists a natural bijection
\begin{align*}
&\mathcal{S}(\overline{\G}) \longrightarrow \mathcal{S}(\G),\quad\overline{\omega}\longmapsto\overline{\omega}+\xi^\ast\wedge\ell^\ast,
\end{align*}
where $\mathcal{S}(\G)$ represents the the set of all symplectic forms on $\G$ modulo  symplectomorphism. Suppose now that $\overline{\G} = \G_{(1,0,0,1)}$. Then, according to Proposition~$\ref{Classi6}$, we have $\mathcal{S}(\G_{(1,0,0,1)}) = \{\omega_\eta\}_{\eta\in\R^{\ast+}}$. Since $(\G_{(1,0,0,1)}, \omega_\eta)$ is an irreducible symplectic Lie algebra for each $\eta \in \R^{\ast+}$, it follows that any symplectic form on $\G$ induces a non completely reducible symplectic Lie algebra.
\end{proof}

\section{Appendix}\label{Appen}
\subsection{Automorphisms of eight-dimensional irreducible symplectic Lie algebras}
$\triangleright$\textbf{Lie algebra $\G^{(a,0)}:$} Recall that the non-zero Lie brackets of $\G^{(a,0)}$ are
\begin{align*}
			[e_7, e_1]& = -e_2,& [e_7, e_2]& = e_1,& [e_8, e_3]& = -e_4,\\
			[e_8, e_4]& = e_3,&[e_7, e_5]& = -a e_6,& [e_7, e_6] &= a e_5.
		\end{align*}
Note that $\G^{(a,0)}$ is decomposable as a direct sum of five- and three-dimensional Lie algebras, i.e.,
\[
\G^{(a,0)} = \G_5^a \oplus \G_3,
\]
where 
\[
\G_5^a = \langle e_1, e_2, e_5, e_6,e_7 \rangle
\quad \text{and} \quad 
\G_3 = \langle e_3, e_4, e_8 \rangle.
\]
Let $\{e_1',e_2',e_3',e_4',e_5'\}$ (resp. $\{e_1'',e_2'',e_3''\}$) be a basis of $\G^a_5$ (resp. $\G_3$). After a short calculation we obtain the following automorphism groups:
\begin{small}
\begin{align*}
\mathrm{Aut}(\G^{a\neq1}_5)&=\left\lbrace
\left( \begin {array}{ccccc} x_{{11}}&x_{{12}}&0&0&x_{{15}}
\\ \noalign{\medskip}-\varepsilon_{2}\,x_{{12}}&\varepsilon_{2}\,x_{{11}}&0
&0&x_{{25}}\\ \noalign{\medskip}0&0&x_{{33}}&x_{{34}}&x_{{35}}
\\ \noalign{\medskip}0&0&-\varepsilon_{2}\,x_{{34}}&\varepsilon_{2}\,x_{{33}}&x_{{45}}\\ \noalign{\medskip}0&0&0&0&\varepsilon_{2}\end {array}
 \right),~~\varepsilon_2=\pm1\right\rbrace
\\
\mathrm{Aut}(\G^{a=1}_5)&=\left\lbrace
\left( \begin {array}{ccccc} x_{{11}}&x_{{12}}&x_{{13}}&x_{{14}}&
x_{{15}}\\ \noalign{\medskip}-\varepsilon_{2}\,x_{{12}}&\varepsilon_{2}\,x
_{{11}}&-\varepsilon_{2}\,x_{{14}}&\varepsilon_{2}\,x_{{13}}&x_{{25}}
\\ \noalign{\medskip}x_{{31}}&x_{{32}}&x_{{33}}&x_{{34}}&x_{{35}}
\\ \noalign{\medskip}-\varepsilon_{2}\,x_{{32}}&\varepsilon_{2}\,x_{{31}}&
-\varepsilon_{2}\,x_{{34}}&\varepsilon_{2}\,x_{{33}}&x_{{45}}
\\ \noalign{\medskip}0&0&0&0&\varepsilon_{2}\end {array} \right),~~\varepsilon_2=\pm1\right\rbrace
\\
\mathrm{Aut}(\G_3)&=\left\lbrace
\left( \begin {array}{ccc} \varepsilon_{1}\,x_{{22}}&-\varepsilon_{1}\,x_{
{21}}&x_{{13}}\\ \noalign{\medskip}x_{{21}}&x_{{22}}&x_{{23}}
\\ \noalign{\medskip}0&0&\varepsilon_{1}\end {array} \right)
,~~\varepsilon_1=\pm1\right\rbrace
\end{align*} 
\end{small}

Let now $\Phi \in \mathrm{Aut}(\G^{(a,0)})$ be a Lie algebra automorphism. Then, $\Phi$ can be represented in matrix form as
\[
\Phi = \begin{pmatrix}
A & 0 \\
0 & B
\end{pmatrix},
\]
where the blocks correspond to the decomposition $\G^{(a,0)} = \G_5^a \oplus \G_3$ such that $A\in\mathrm{Aut}(\G_5^{a})$, $B\in\mathrm{Aut}(\G_3)$.\\\\
$\triangleright$\textbf{Lie algebra $\G^{(a,b)}:$} Recall that the non-zero Lie brackets of $\G^{(a,b)}$ are
\begin{align*}
	[e_7, e_1] &= -e_2,& [e_7, e_2] &= e_1, &[e_8, e_3] &= -e_4, &[e_8, e_4]& = e_3,\\
[e_7, e_5] &= -a e_6,& [e_7, e_6]& = a e_5, &[e_8, e_5]& = -b e_6,& [e_8, e_6] &= b e_5.
		\end{align*}
Note that, $\G^{(a,b)}$ is an indecomposable Lie algebra, and can also be viewed as a semi-direct sum of five- and three-dimensional Lie algebras, i.e., $\G^{(a,b)}=\G_5^{a}\rtimes\G_3$. 		

Let $\{e_1', e_2', e_3', e_4', e_5'\}$ (resp.\ $\{e_1'', e_2'', e_3''\}$) be a basis of $\G^a_5$ (resp.\ $\G_3$). In this case, $e_3''$ acts on $\{e_3', e_4'\}$ via the adjoint representation. Therefore, every automorphism of $\G^a_5 \oplus \G_3$ preserves this action. Therefore, every automorphism of $\G^{(a,b)}$ has the following form
\begin{small}
\begin{align*}
\Phi:=\left(
\begin{array}{cccc|ccc}
&&&&0&0&0\\
&&&&\cdot&\cdot&0\\
&\mathrm{Aut}(\G_5^a)&&&\cdot&\cdot&\ast\\
&&&&\cdot&\cdot&\ast\\
&&&&0&0&0
 \\\hline
 0&\cdot\cdot\cdot&0&&&&\\
 \cdot&&\cdot&&&&\\
 \cdot&&\cdot&&&\mathrm{Aut}(\G_3)&\\
\cdot&&\cdot&&&&\\
0&\cdot\cdot\cdot&0&&&&
\end{array}
\right).
\end{align*}
\end{small}

Finally, by sweeping the basis $\{e_1', e_2', e_3', e_4', e_5'\}$ (resp.\ $\{e_1'', e_2'', e_3''\}$) of $\G^a_5$ (resp.\ $\G_3^b$) to be 
\[
\{e_1', e_2', e_3', e_4', e_5'\} = \{e_1, e_2, e_5, e_6, e_7\} \quad \text{(resp.\ } \{e_3, e_4, e_8\}\text{)},
\] 
we obtain the automorphism groups of $\G^{(a,0)}$, $\G^{(1,0)}$ and $\G^{(a,b)}$ given in the proof of Proposition~\ref{Classi8forms} below (see~\ref{proof of pr4}).

\subsection{Proof of Proposition~$\ref{Classi8forms}$}\label{proof of pr4}
\textbf{Symplectic forms on $\G^{(1,0)}$:} Let
$\omega\in Z^2(\G^{(a,0)})$, i.e., $\md\omega=0$ is equivalent 
\begin{align*}
\omega&=\omega_{12}e^{12}+\omega_{34}e^{34}+\omega_{78}e^{78}+\omega_{15}(e^{15}+ae^{26})+\omega_{25}(-ae^{16}+e^{25})\\
&~~~~+\omega_{17}e^{17}+\omega_{27}e^{27}+\omega_{38}e^{38}+\omega_{48}e^{48}+\omega_{56}e^{56}+\omega_{57}e^{57}+\omega_{67}e^{67},
\end{align*}
with $a^2\omega_{2 5} - \omega_{2 5}=-a^2\omega_{1 5} + \omega_{1 5}=0$. Since $a>0$, we consider two cases: $a=1$ or $a\neq1$.

$\bullet$ If $a=1$. Then, the $2$-cocycle $\omega$ given above is symplectic if and only if $\omega_{3 4}(\omega_{1 2}\omega_{5 6} - \omega_{1 5}^2 - \omega_{2 5}^2)\omega_{7 8}\neq0$. The automorphism group of  $\G=\G^{(1,0)}$
 is
 \begin{small}
\begin{equation*}
\mathrm{Aut}\big(\G^{(1,0)}\big)=
\left\lbrace\Phi_{(\varepsilon_1,\varepsilon_2)}=\left(
\begin {array}{cccccccc} \varepsilon_2x_{22}&-\varepsilon_2x_{{21}}&0&0&\varepsilon_2x_{{26}}&-\varepsilon_2x
_{{25}}&x_{{17}}&0\\ \noalign{\medskip}x_{{21}}&x_{22}&0&0&x_{{2
5}}&x_{{26}}&x_{{27}}&0\\ \noalign{\medskip}0&0&\varepsilon_1 x_{44}&-\varepsilon_1x_{{43}
}&0&0&0&x_{{38}}\\ \noalign{\medskip}0&0&x_{{43}}&x_{44}&0&0&0&x_
{{48}}\\ \noalign{\medskip}\varepsilon_2x_{{62}}&-\varepsilon_2x_{{61}}&0&0&\varepsilon_2x_{{66}}&-\varepsilon_2x_{65}&x_{{57}}&0\\ \noalign{\medskip}x_{{61}}&x_{{62}}&0&0&x_{65}&
x_{{66}}&x_{{67}}&0\\ \noalign{\medskip}0&0&0&0&0&0&\varepsilon_2&0
\\ \noalign{\medskip}0&0&0&0&0&0&0&\varepsilon_1\end {array}\right)
,~~\varepsilon_1,\varepsilon_2=\pm1\right\rbrace
\end{equation*}
\end{small}
such that the determinant $\Phi_{(\varepsilon_1,\varepsilon_2)}$ is nonzero. If $\omega_{12} \neq 0$, then consider the automorphism $\Phi_{(1,1)}$ with the following choice of parameters:
\begin{align*}
x_{{17}}&={\frac {\omega_{{25}}\omega_{{67}}-\omega_{{26}}\omega_{{57}}+\omega_{{27}}\omega_{{56}}}{\omega_{{12}}\omega_{{56}}-{\omega_{{25}}^{2}}-{\omega_{{26}}^{2}}}},&x_
{{21}}&=0,&x_{{25}}&={\frac {x_{65}\omega_{{25}}}{\omega_{{12}}}},&x_{{26}}&=
{\frac {x_{65}\omega_{{26}}}{\omega_{{12}}}},\\
x_{{27}}&={\frac {-\omega_{{17}}\omega_
{{56}}-\omega_{{25}}\omega_{{57}}-\omega_{{26}}\omega_{{67}}}{\omega_{{12}}\omega_{{56}}-{\omega_{
{25}}^{2}}-{\omega_{{26}}^{2}}}},&x_{{38}}&={\frac {\omega_{{48}}}{\omega_{{34}}}}
,&x_{{43}}&=0,&x_{{48}}&=-{\frac {\omega_{{38}}}{\omega_{{34}}}},\\
x_{{57}}&={
\frac {\omega_{{67}}\omega_{{12}}+\omega_{{26}}\omega_{{17}}+\omega_{{25}}\omega_{{27}}}{a_{{1
2}}\omega_{{56}}-{\omega_{{25}}^{2}}-{\omega_{{26}}^{2}}}},&x_{{61}}&=0,&x_{{62}}&=0
,&x_{{66}}&=0,\\
x_{{67}}&=-{\frac {\omega_{{57}}\omega_{{12}}+\omega_{{25}}\omega_{{17}}-
\omega_{{27}}\omega_{{26}}}{\omega_{{12}}\omega_{{56}}-{\omega_{{25}}^{2}}-{\omega_{{26}}^{2}}
}} ,&x_{65}&\in\R^\ast.
\end{align*}
Then, we obtain
\begin{align*}
\Phi^\ast(\omega)=x_{22}^2\omega_{12}e^{12}+x_{44}^2\omega_{34}e^{34}+\tfrac{1}{\omega_{12}}(\omega_{12}\omega_{56}-\omega_{25}^2-\omega_{26}^2)x_{65}^2e^{56}+\omega_{78}e^{78}
\end{align*}
for all $x_{22}, x_{44}, x_{65} \in \mathbb{R}^\ast$ with $\omega_{12}\omega_{34}\omega_{78}(\omega_{12}\omega_{56}-\omega_{25}^2-\omega_{26}^2) \neq 0$. Therefore,
\begin{align*}
\omega_{\kappa_1,\kappa_2,\kappa_3,\mu}&=\Phi^\ast(\omega)=\kappa_1 e^{12}+\kappa_2e^{34}+\kappa_3e^{56}+\mu e^{78},\quad\kappa_j=\pm1,~~\mu\in\R^\ast.
\end{align*}
Next, consider the following family of symplectic forms on $\G^{(1,0)}$ derived from $\omega_{\kappa_1,\kappa_2,\kappa_3,\mu}$:
\begin{align*}
\omega_{1,1,1,\mu_1},& &\omega_{1,-1,1,\mu_2}&,&\omega_{1,1,-1,\mu_3}&,&\omega_{1,-1,-1,\mu_4}&,\\
\omega_{-1,1,1,\mu_5},& &\omega_{-1,-1,1,\mu_6}&,&\omega_{-1,1,-1,\mu_7}&,&\omega_{-1,-1,-1,\mu_8}&.
\end{align*}
It is straightforward to verify the following:
\begin{align*}
\Phi^\ast_{(-1,1)}\omega_{1,1,1,\mu_1}=\omega_{1,-1,1,\mu_2},
\end{align*}
where $\Phi_{(-1,1)}$ is the block diagonal matrix given by
\[
\Phi =\begin{small}
 \begin{pmatrix}
0 & -1 \\
1 & 0 
\end{pmatrix} \oplus 
\begin{pmatrix}
0 & 1 \\
1 & 0 
\end{pmatrix} \oplus 
\begin{pmatrix}
0 & -1 \\
1 & 0 
\end{pmatrix} \oplus 
\begin{pmatrix}
1 & 0 \\
0 & -1 
\end{pmatrix}.
\end{small}
\]
In this case, we take $\mu_1 = -\mu_2$.

\begin{align*}
\Phi^\ast_{(1,-1)}\omega_{1,1,-1,\mu_3}=\omega_{1,-1,-1,\mu_4},\quad\text{and}\quad\Phi^\ast_{(1,-1)}\omega_{1,1,-1,\mu_3}=\omega_{-1,1,1,\mu_5},
\end{align*}
where $\Phi_{(1,-1)}$ is the block diagonal matrix given by
\[
\Phi =\begin{small}
 \begin{pmatrix}
-1 & 0 \\
0 & 1 
\end{pmatrix} \oplus 
\mathbf{I}_2 \oplus 
 \begin{pmatrix}
-1 & 0 \\
0 & 1 
\end{pmatrix} \oplus 
 \begin{pmatrix}
-1 & 0 \\
0 & 1 
\end{pmatrix}.
\end{small}
\]
In this case, we take $\mu_3 = -\mu_4$, and $\mu_3 = -\mu_5$.

\begin{align*}
\Phi^\ast_{(-1,-1)}\omega_{1,1,-1,\mu_3}=\omega_{-1,-1,1,\mu_6},
\end{align*}
where $\Phi_{(-1,-1)}$ is the block diagonal matrix given by
\[
\Phi =\begin{small}
 \begin{pmatrix}
-1 & 0 \\
0 & 1 
\end{pmatrix} \oplus 
\begin{pmatrix}
-1 & 0 \\
0 & 1 
\end{pmatrix} \oplus 
 \begin{pmatrix}
-1 & 0 \\
0 & 1 
\end{pmatrix} \oplus 
 \begin{pmatrix}
-1 & 0 \\
0 & -1 
\end{pmatrix}.
\end{small}
\]
In this case, we take $\mu_3 = \mu_6$.

\begin{align*}
\Phi^\ast_{(1,-1)}\omega_{1,1,1,\mu_1}=\omega_{-1,1,-1,\mu_7},
\end{align*}
where $\Phi_{(1,-1)}$ is the block diagonal matrix given by
\[
\Phi =\begin{small}
 \begin{pmatrix}
-1 & 0 \\
0 & 1 
\end{pmatrix} \oplus 
\mathbf{I}_2 \oplus 
 \begin{pmatrix}
-1 & 0 \\
0 & 1 
\end{pmatrix} \oplus 
 \begin{pmatrix}
-1 & 0 \\
0 & 1 
\end{pmatrix}.
\end{small}
\]
In this case, we take $\mu_3 = -\mu_7$.

\begin{align*}
\Phi^\ast_{(1,-1)}\omega_{1,1,1,\mu_1}=\omega_{-1,1,-1,\mu_7},
\end{align*}
where $\Phi_{(1,-1)}$ is the block diagonal matrix given by
\[
\Phi =\begin{small}
 \begin{pmatrix}
-1 & 0 \\
0 & 1 
\end{pmatrix} \oplus 
\begin{pmatrix}
-1 & 0 \\
0 & 1 
\end{pmatrix} \oplus 
 \begin{pmatrix}
-1 & 0 \\
0 & 1 
\end{pmatrix} \oplus 
 (-\mathbf{I}_2).
\end{small}
\]
In this case, we take $\mu_3 = \mu_8$.

Let us now consider the remaining symplectic forms, $\omega_{1,1,1,\mu_1}, \omega_{1,1,-1,\mu_3}$. These forms are never symplectomorphic to each other, for the following reason: We have 
\begin{align*}
&\Phi^\ast_{(\varepsilon_1,1)}\omega_{1,1,1,\mu_1}(e_5,e_6)-\omega_{1,1,-,\mu_3}(e_5,e_6)=0,
\end{align*}
or
\begin{align*}
&\Phi^\ast_{(1,\varepsilon_2)}\omega_{1,1,1,\mu_1}(e_1,e_2)-\omega_{1,1,-1,\mu_3}(e_1,e_2)=0
\end{align*}
if and only if 
\begin{align*}
x_{2 5}^2 + x_{2 6}^2 + x_{6 5}^2 + x_{6 6}^2=-1,\quad\text{or}\quad x_{2 1}^2 + x_{2 2}^2 + x_{6 1}^2 + x_{6 2}^2=-1.
\end{align*}
These equations represent a complex quadric hypersurface and therefore do not admit real solutions.

Now, if $\omega_{12}=0$, then the symplectic form $\omega$ on $\G^{(1,0)}$ becomes
\begin{align*}
\omega&=\omega_{34}e^{34}+\omega_{78}e^{78}+\omega_{15}(e^{15}+ae^{26})+\omega_{25}(-ae^{16}+e^{25})\\
&~~~~+\omega_{17}e^{17}+\omega_{27}e^{27}+\omega_{38}e^{38}+\omega_{48}e^{48}+\omega_{56}e^{56}+\omega_{57}e^{57}+\omega_{67}e^{67}.
\end{align*}
Note that, $\mathfrak{j}=\mathfrak{a}_1=\langle e_1,e_2\rangle$ is an isotropic ideal of $(\G^{(1,0)},\omega)$, and it is straightforward to show that  $(\G^{(1,0)},\omega)$ is completely reducible to the trivial symplectic Lie algebra. Then, this form is not considered in our classifications.\\\\
$\bullet$ If $a\neq0$. Then, the symplectic for $\omega$ becomes 
\begin{align*}
\omega&=\omega_{12}e^{12}+\omega_{34}e^{34}+\omega_{78}e^{78}+\omega_{17}e^{17}+\omega_{27}e^{27}+\omega_{38}e^{38}+\omega_{48}e^{48}+\omega_{56}e^{56}+\omega_{57}e^{57}+\omega_{67}e^{67},
\end{align*}
with $\omega_{12}\omega_{34}\omega_{56}\omega_{78}\neq0$. The automorphism group of  $\G=\G^{(a,0)}$
 is
 \begin{small}
\begin{equation*}
\mathrm{Aut}\big(\G^{(a,0)}\big)=
\left\lbrace\Phi_{(\varepsilon_1,\varepsilon_2)}=\left( \begin {array}{cccccccc} x_{{11}}&x_{{12}}&0&0&0&0&x_{{17}}
&0\\ \noalign{\medskip}-\varepsilon_{2}\,x_{{12}}&\varepsilon_{2}\,x_{{11}
}&0&0&0&0&x_{{27}}&0\\ \noalign{\medskip}0&0&x_{{33}}&x_{{34}}&0&0&0
&x_{{38}}\\ \noalign{\medskip}0&0&-\varepsilon_{1}\,x_{{34}}&\epsilon_{
1}\,x_{{33}}&0&0&0&x_{{48}}\\ \noalign{\medskip}0&0&0&0&x_{{55}}&x_
{{56}}&0&0\\ \noalign{\medskip}0&0&0&0&-\epsilon_{2}\,x_{{56}}&
\varepsilon_{2}\,x_{{55}}&0&0\\ \noalign{\medskip}0&0&0&0&0&0&\epsilon_{
2}&0\\ \noalign{\medskip}0&0&0&0&0&0&0&\varepsilon_{1}\end {array}
 \right)
,~~\varepsilon_1,\varepsilon_2=\pm1\right\rbrace
\end{equation*}
\end{small}
Consider the automorphism $\Phi_{(1,1)}$ with the following choice of parameters:
\begin{align*}
 x_{{12}}&=0,&x_{{17}}&={\frac {\omega_{{27}}}{\omega_{{12}}}},&x_{{27}
}&=-{\frac {\omega_{{17}}}{\omega_{{12}}}},&x_{{34}}&=0,&x_{{38}}&={\frac {\omega_{{4
8}}}{\omega_{{34}}}},&x_{{48}}&=-{\frac {\omega_{{38}}}{\omega_{{34}}}} 
\end{align*}
Then, we obtain
\begin{align*}
\Phi^\ast_{(1,1)}\omega&=x_{11}^2\omega_{12}e^{12}+x_{33}^2\omega_{34}e^{34}+(x_{55}^2+x_{56}^2)e^{56}+(\omega_{5 7}x_{5 5} - \omega_{6 7}x_{5 6})e^{57}+(\omega_{5 7}x_{5 6} + \omega_{6 7}x_{5 5})e^{67}+\omega_{78}e^{78},
\end{align*}
with $(x_{5 5}^2 + x_{5 6}^2)x_{11}x_{33}\det\omega\neq0$. If $a_{67}\neq0$ or $a_{67}$, we can easy show that $\Phi^\ast_{(1,1)}\omega$ is symplectomorphically isomorphic to the following:
\begin{align*}
\omega&=\kappa_1 e^{12}+\kappa_2 e^{34}+\kappa_3 e^{56}+\mu e^{67}+\lambda e^{78},
\end{align*}
where, $\kappa_j=\pm1$ and $\lambda,\mu\in\R^\ast$.

Next, consider the following family of symplectic forms on $\G^{(1,0)}$ derived from $\omega_{\kappa_1,\kappa_2,\kappa_3,\mu,\lambda}$:
\begin{align*}
\omega_{1,1,1,\mu_1,\lambda_1},& &\omega_{1,-1,1,\mu_2,\lambda_2}&,&\omega_{1,1,-1,\mu_3,\lambda_3}&,&\omega_{1,-1,-1,\mu_4,\lambda_4}&,\\
\omega_{-1,1,1,\mu_5,\lambda_5},& &\omega_{-1,-1,1,\mu_6,\lambda_6}&,&\omega_{-1,1,-1,\mu_7,\lambda_7}&,&\omega_{-1,-1,-1,\mu_8,\lambda_8}&.
\end{align*}

It is straightforward to verify the following:

\begin{align*}
\Phi^\ast_{(-1,1)}\omega_{1,1,1,\mu_1,\lambda_1}=\omega_{1,-1,1,\mu_2,\lambda_2},
\end{align*}
where $\Phi_{(-1,1)}$ is the block diagonal matrix given by
\[
\Phi =\begin{small}
 \mathbf{I}_2 \oplus 
\begin{pmatrix}
1 & 0 \\
0 & -1 
\end{pmatrix} \oplus 
 \mathbf{I}_2 \oplus 
\begin{pmatrix}
1 & 0 \\
0 & -1 
\end{pmatrix}.
\end{small}
\]
In this case, we take $\mu_1 = \mu_2$ and $\lambda_1=-\lambda_2$.

\begin{align*}
\Phi^\ast_{(1,-1)}\omega_{1,1,1,\mu_1,\lambda_1}=\omega_{-1,1,-1,\mu_7,\lambda_7},
\end{align*}
where $\Phi_{(1,-1)}$ is the block diagonal matrix given by
\[
\Phi =\begin{small}
\begin{pmatrix}
1 & 0 \\
0 & -1 
\end{pmatrix}\oplus 
\mathbf{I}_2 \oplus 
\begin{pmatrix}
1 & 0 \\
0 & -1 
\end{pmatrix} \oplus 
\begin{pmatrix}
-1 & 0 \\
0 & 1 
\end{pmatrix}.
\end{small}
\]
In this case, we take $\mu_1 = \mu_7$ and $\lambda_1=-\lambda_7$.

\begin{align*}
\Phi^\ast_{(-1,-1)}\omega_{1,1,1,\mu_1,\lambda_1}=\omega_{-1,-1,-1,\mu_8,\lambda_8},
\end{align*}
where $\Phi_{(-1,-1)}$ is the block diagonal matrix given by
\[
\Phi =\begin{small}
\begin{pmatrix}
1 & 0 \\
0 & -1 
\end{pmatrix}\oplus 
\begin{pmatrix}
1 & 0 \\
0 & -1 
\end{pmatrix} \oplus 
\begin{pmatrix}
1 & 0 \\
0 & -1 
\end{pmatrix} \oplus 
\begin{pmatrix}
-1 & 0 \\
0 & -1 
\end{pmatrix}.
\end{small}
\]
In this case, we take $\mu_1 = \mu_8$ and $\lambda_1=\lambda_8$.

\begin{align*}
\Phi^\ast_{(-1,1)}\omega_{1,1,-1,\mu_3,\lambda_3}=\omega_{1,-1,-1,\mu_4,\lambda_4},
\end{align*}
where $\Phi_{(-1,1)}$ is the block diagonal matrix given by
\[
\Phi =\begin{small}
\mathbf{I}_2\oplus 
\begin{pmatrix}
1 & 0 \\
0 & -1 
\end{pmatrix} \oplus 
\mathbf{I}_2 \oplus 
\begin{pmatrix}
1 & 0 \\
0 & -1 
\end{pmatrix}.
\end{small}
\]
In this case, we take $\mu_3 = \mu_4$ and $\lambda_3=-\lambda_4$.

\begin{align*}
\Phi^\ast_{(1,-1)}\omega_{1,1,-1,\mu_3,\lambda_3}=\omega_{-1,1,1,\mu_5,\lambda_5},
\end{align*}
where $\Phi_{(1,-1)}$ is the block diagonal matrix given by
\[
\Phi =\begin{small}
\begin{pmatrix}
1 & 0 \\
0 & -1 
\end{pmatrix}\oplus 
\mathbf{I}_2\oplus
\begin{pmatrix}
1 & 0 \\
0 & -1 
\end{pmatrix} \oplus 
\begin{pmatrix}
-1 & 0 \\
0 & 1 
\end{pmatrix}.
\end{small}
\]
In this case, we take $\mu_3 = \mu_5$ and $\lambda_3=-\lambda_5$.

\begin{align*}
\Phi^\ast_{(-1,-1)}\omega_{1,1,-1,\mu_3,\lambda_3}=\omega_{-1,-1,1,\mu_6,\lambda_6},
\end{align*}
where $\Phi_{(-1,-1)}$ is the block diagonal matrix given by
\[
\Phi =\begin{small}
\begin{pmatrix}
1 & 0 \\
0 & -1 
\end{pmatrix}\oplus 
\begin{pmatrix}
1 & 0 \\
0 & -1 
\end{pmatrix}\oplus
\begin{pmatrix}
1 & 0 \\
0 & -1 
\end{pmatrix} \oplus 
\begin{pmatrix}
-1 & 0 \\
0 & -1 
\end{pmatrix}.
\end{small}
\]
In this case, we take $\mu_3 = \mu_6$ and $\lambda_3=\lambda_6$.

Let us now consider the remaining symplectic forms, $\omega_{1,1,1,\mu_1,\lambda_1}, \omega_{1,1,-1,\mu_3,\lambda_3}$. These forms are never symplectomorphic to each other, for the following reason: We have 
\begin{align*}
&\Phi^\ast_{(\varepsilon_1,-1)}\omega_{1,1,1,\mu_1,\lambda_1}(e_1,e_2)-\omega_{1,1,-,\mu_3,\lambda_3}(e_1,e_2)=0,
\end{align*}
or
\begin{align*}
&\Phi^\ast_{(\varepsilon_1,1)}\omega_{1,1,1,\mu_1,\lambda_1}(e_3,e_4)-\omega_{1,1,-1,\mu_3,\lambda_3}(e_3,e_4)=0
\end{align*}
if and only if 
\begin{align*}
x_{11}^2 + x_{12}^2 =-1,\quad\text{or}\quad x_{33}^2 + x_{34}^2=-1.
\end{align*}
Note that $\Phi^\ast\omega_{1,1,1,\mu_1,\lambda_1}=\omega_{1,1,1,-\mu_1,\lambda_1}$ and $\Phi^\ast\omega_{1,1,-1,\mu_3,\lambda_3}=\omega_{1,1,-1,-\mu_3,\lambda_3}$, where
\begin{align*}
\Phi=\mathbf{I}_2\oplus\mathbf{I}_2\oplus(-\mathbf{I}_2)\oplus\mathbf{I}_2
\end{align*}
 It follows that $\mu_1,\mu_3\in\R^{\ast+}$.\\\\
 
\textbf{Symplectic forms on $\G^{(a,b)}$}, $a,b\neq0$. Let $\omega\in Z^2(\G^{(a,b)})$, i.e., $\md\omega=0$ is equivalent to
\begin{align*}
\omega &= \omega_{12}e^{12}+\omega_{17}e^{17}+\omega_{27}e^{27}+\omega_{34}e^{34}+\omega_{38}e^{38}+\omega_{48}e^{48}+\omega_{56}e^{56}\\
&~~~~+\omega_{57}e^{57}+\tfrac{b}{a}\omega_{57}e^{58}+\omega_{67}e^{67}+\tfrac{b}{a}\omega_{67}e^{68}+\omega_{78}e^{78}.
\end{align*}
Then, $\omega$ is symplectic if and only if $\omega_{12}\omega_{34}\omega_{56}\omega_{78}\neq0$. The automorphism group of $\G^{(a,b)}$ is
\begin{small}
\begin{equation*}
\mathrm{Aut}\big(\G^{(a,b)}\big)=
\left\lbrace\Phi_{(\varepsilon_1,\varepsilon_2)}=\left( \begin {array}{cccccccc} \varepsilon_{2}\,x_{{22}}&-\varepsilon_{2}
\,x_{{21}}&0&0&0&0&x_{{17}}&0\\ \noalign{\medskip}x_{{21}}&x_{{22}
}&0&0&0&0&x_{{27}}&0\\ \noalign{\medskip}0&0&\varepsilon_{1}\,x_{{44}}&
-\varepsilon_{1}\,x_{{43}}&0&0&0&x_{{38}}\\ \noalign{\medskip}0&0&x_{{4
3}}&x_{{44}}&0&0&0&x_{{48}}\\ \noalign{\medskip}0&0&0&0&\varepsilon_{2
}\,x_{{66}}&-\varepsilon_{2}\,x_{{65}}&{\frac {a}{b}x_{{58}}}&x_{{58}
}\\ \noalign{\medskip}0&0&0&0&x_{{65}}&x_{{66}}&{\frac {a}{
b}x_{{68}}}&x_{{68}}\\ \noalign{\medskip}0&0&0&0&0&0&\varepsilon_{2}&0
\\ \noalign{\medskip}0&0&0&0&0&0&0&\varepsilon_{1}\end {array} \right)
,~~\varepsilon_1,\varepsilon_2=\pm1\right\rbrace
\end{equation*}
\end{small}
Consider the automorphism $\Phi_{(1,1)}$ with the following choice of parameters:
\begin{align*}
 x_{{12}}&=0,&x_{{17}}&={\frac {\omega_{{27}}}{\omega_{{12}}}},&x_{2 1}& = 0,&x_{{27}
}&=-{\frac {\omega_{{17}}}{\omega_{{12}}}},&x_{{34}}&=0,&x_{{38}}&={\frac {\omega_{{4
8}}}{\omega_{{34}}}},\\
 x_{4 3}& = 0,&x_{{48}}&=-{\frac {\omega_{{38}}}{\omega_{{34}}}}, &x_{{58}}&={\frac {b\omega_{{67}}}{a\omega_{{56}}}},&x_{6 5}& = 0,&x_{{68}}&=-{\frac 
{b\omega_{{57}}}{a\omega_{{56}}}}, 
\end{align*}
Then, we obtain
\begin{align*}
\Phi^\ast_{(1,1)}\omega&=x^2_{22}\omega_{12}e^{12}+x^2_{44}\omega_{34}e^{34}+x^2_{66}\omega_{56}e^{56}+\omega_{78}e^{78}
\end{align*}
for all $x_{22},x_{44},x_{66}\in\R^\ast$, with $\omega_{12}\omega_{34}\omega_{56}\omega_{78}\neq0$. It follows that
\begin{align*}
\Phi^\ast_{(1,1)}\omega&=\kappa_1 e^{12}+\kappa_2 e^{34}+\kappa_3 e^{56}+\lambda e^{78},
\end{align*}
where, $\lambda\in\R^\ast$ and $\kappa_j=\pm1$.

Consider the following family of symplectic forms on $\G^{(a,b)}$ derived from $\omega_{\kappa_1,\kappa_2,\kappa_3,\lambda}$:
\begin{align*}
\omega_{1,1,1,\lambda_1},& &\omega_{1,-1,1,\lambda_2}&,&\omega_{1,1,-1,\lambda_3}&,&\omega_{1,-1,-1,,\lambda_4}&,\\
\omega_{-1,1,1,\lambda_5},& &\omega_{-1,-1,1,\lambda_6}&,&\omega_{-1,1,-1,\lambda_7}&,&\omega_{-1,-1,-1,\lambda_8}&.
\end{align*}

 Finalement et le meme raisonement précedent, on montre que $\omega$ is symplectomorphically isomorphic to one of the two non-symplectomophic forms $\omega_{1,1,1,\lambda_1}$ or $\omega_{1,1,-1,\lambda_3}$.

\begin{Le}\label{Endomor}
Let $(\G_{(\mathrm{D},\mu,\lambda,t)},\omega)$  be the generalized symplectic oxidation  of $(\G_{1,0,0,1},\omega_\eta)$. Then, the data $\mathrm{D},\mu,\lambda$ and $t$ are one of the following family$:$
\begin{align*}
\mathrm{D}_1e_1&=d_{21}e_2,\hspace{1.5cm}\mathrm{D}_1e_2=-d_{21}e_2,\hspace{1.5cm}\mathrm{D}_1e_3=d_{43}e_4,\hspace{1.5cm}\mathrm{D}_1e_4=-d_{43}e_3,\\
\mathrm{D}_1e_5&=\tfrac{1}{\mu_6}(-d_{25}\mu_5\mu_6+(\mu_5^2+1)d_{26})e_1+d_{25}e_2+(d_{45}\mu_6-d_{46}\mu_5)e_3+d_{45}e_4+d_{21}\mu_5e_5+d_{43}\mu_5e_6,\\
\mathrm{D}_1e_6&=(-d_{25}\mu_6 + d_{26}\mu_5)e_1+d_{26}e_2+\tfrac{1}{\mu_5}(d_{45}\mu_6^2 - d_{46}\mu_5\mu_6 + d_{45})e_3+d_{46}e_4+d_{21}\mu_6e_5+d_{43}\mu_6e_6,\\
\mu&=\mu_5e^5+\mu_6e^6,\quad\mu_5\mu_6\neq0,\\
\lambda&=\mathop{\resizebox{1\width}{!}{$\sum$}^6_{i=1}}\limits_{}\lambda_ie^i,\\
\lambda_1&=d_{21}d_{25} + d_{26}d_{43},\hspace{1cm}\lambda_2=\tfrac{1}{\mu_6}(d_{25}d_{43}\mu_6^2 + \mu_5(d_{21}d_{25} - d_{26}d_{43})\mu_6 - d_{21}d_{26}(\mu_5^2 + 1)),\\
\lambda_3&=d_{21}d_{45} + d_{43}d_{46},\hspace{1cm}\lambda_4=\tfrac{1}{\mu_5}(d_{21}d_{46}\mu_5^2 - \mu_6(d_{21}d_{45} - d_{43}d_{46})\mu_5 - d_{45}d_{43}(\mu_6^2 + 1)),\\
\lambda_5&\in\R,\hspace{2.93cm}\lambda_6=\tfrac{1}{\mu_6\mu_5^2}((-d_{21}^2\eta\mu_6 - d_{26}^2)\mu_5^3 - 2\mu_6(\eta d_{21}d_{43}\mu_6 - d_{25}d_{26} - \tfrac{1}{2} d_{46}^2)\mu_5^2 )\\
&\hspace{4.2cm}+\tfrac{1}{\mu_6\mu_5^2}( (-\eta d_{43}^2\mu_6^3 + (-d_{25}^2 - 2d_{45}d_{46} + \lambda_5)\mu_6^2 - d_{26}^2)\mu_5 + d_{45}^2\mu_6(\mu_6^2 + 1))\\
t &= -d_{21}\mu_5 - d_{43}\mu_6.\\
&\\
\mathrm{D}_2e_1&=d_{21}e_2,\hspace{1.5cm}\mathrm{D}_2e_2=-d_{21}e_2,\hspace{1.5cm}\mathrm{D}_2e_3=d_{43}e_4,\hspace{1.5cm}\mathrm{D}_2e_4=-d_{43}e_3,\\
\mathrm{D}_2e_5&=d_{15}e_1+\tfrac{1}{d_{21}}\lambda_1e_2-d_{46}\mu_5e_3+d_{45}e_4+d_{21}\mu_5e_5+d_{43}\mu_5e_6,\hspace{1.85cm}\mathrm{D}_2e_6=\tfrac{1}{\mu_5}d_{45}e_3+d_{46}e_4,\\
\mu&=\mu_5e^5,\quad\mu_5\neq0,\\
\lambda&=\mathop{\resizebox{1\width}{!}{$\sum$}^6_{i=1}}\limits_{}\lambda_ie^i,\\
\lambda_1 &\in\R, \hspace{2.93cm}\lambda_2 = -d_{15}d_{21},\hspace{1cm} \lambda_3 = d_{21}d_{45} + d_{43}d_{46},\\
 \lambda_4 &= \tfrac{1}{\mu_5}(d_{21}d_{46}\mu_5^2 - d_{43}d_{45}),\hspace{0.15cm} \lambda_5  \in\R,\hspace{1.86cm} \lambda_6 =\tfrac{1}{\mu_5^2} (-d_{21}^2\eta\mu_5^3 + d_{46}^2\mu_5^2 + d_{45}^2),\\
 t&=-d_{21}\mu_5.\\
 &\\
\mathrm{D}_3e_3&=d_{43}e_4,\hspace{1.5cm}\mathrm{D}_3e_4=-d_{43}e_3,\\
\mathrm{D}_2e_3&=d_{15}e_1+d_{25}e_2-d_{46}\mu_5e_3+d_{45}e_4+d_{43}\mu_5e_6,\hspace{1.85cm}\mathrm{D}_3e_6=\tfrac{1}{\mu_5}d_{45}e_3+d_{46}e_4,\\
\mu&=\mu_5e^5,\quad\mu_5\neq0,\\
\lambda&=\mathop{\resizebox{1\width}{!}{$\sum$}^6_{i=1}}\limits_{}\lambda_ie^i,\\
\lambda_1&=\lambda_2 =0, \hspace{1cm}\lambda_3=d_{43}d_{46},\hspace{1cm}\lambda_4=-\tfrac{1}{\mu_5}d_{45}d_{43},\hspace{1cm}\lambda_5\in\R,\hspace{1cm}\lambda_6=\tfrac{1}{\mu_5^2}(d_{46}^2\mu_5^2+d_{45}^2),\\
t&=0.
\end{align*}

\begin{align*}
\mathrm{D}_4e_1&=d_{21}e_2,\hspace{1.5cm}\mathrm{D}_4e_2=-d_{21}e_2,\hspace{1.5cm}\mathrm{D}_4e_3=d_{43}e_4,\hspace{1.5cm}\mathrm{D}_4e_4=-d_{43}e_3,\\
\mathrm{D}_4e_5&=d_{15}e_1+d_{25}e_2,\hspace{0.38cm}\mathrm{D}_4e_6=-d_{25}\mu_6e_3+d_{15}\mu_6e_2+d_{36}e_3+d_{46}e_4+d_{21}\mu_6e_5+d_{43}\mu_6e_6,\\
\mu&=\mu_6e^6,\quad\mu_6\neq0,\\
\lambda&=\mathop{\resizebox{1\width}{!}{$\sum$}^6_{i=1}}\limits_{}\lambda_ie^i,\\
\lambda_1 &= d_{15}d_{43}\mu_6 + d_{21}d_{25},\hspace{1cm} \lambda_2 = d_{25}d_{43}\mu_6 - d_{15}d_{21},\hspace{1cm} \lambda_3 = d_{43}d_{46},\\
 \lambda_4 &= -d_{36}d_{43},\hspace{2.37cm} \lambda_5 = d_{43}^2\eta\mu_6 + d_{15}^2 + d_{25}^2,\hspace{0.9cm} \lambda_6\in\R,\\
 t&=-d_{43}\mu_6.\\
\end{align*}
\end{Le}
\begin{proof}
Let $(\G_{(\mathrm{D},\mu,\lambda,t)}, \omega)$ be the generalized symplectic oxidation of $(\G_{1,0,0,1}, \omega_\eta)$. Then, the data $\mathrm{D}$, $\mu$, $\lambda$, and $t$ satisfy the four conditions of Proposition~$\ref{prgeneralized}$, which are:
\begin{enumerate}
\item[$(i)$] $\mathop{\resizebox{1.3\width}{!}{$\sum$}}\limits_{\mathrm{cycl}}\ol\om_\mathrm{D}(\ol{[x,y]},z)-\ol\om_\mathrm{D}(x,y)\mu(z)=0$
\item[$(ii)$] $\Delta \mathrm{D}-\mu\otimes\mathrm{D} = 0,$
\item[$(iii)$] $t\ol{\om}_{\eta,\mathrm{D}}-\ol{\om}_{\eta,\mathrm{D},\mathrm{D}}=\md\lambda-2\lambda\otimes\mu,$
\item[$(iv)$] $\mu\circ \mathrm{D}=t\mu$ and $\md\mu=0$.
\end{enumerate}
Recall that $(\G_{(1,0,0,1)} = \mathfrak{a} \rtimes \h,\omega_\eta)$ is an irreducible symplectic Lie algebra of dimension six, where:
\begin{itemize}
    \item $\mathfrak{a} = \mathfrak{a}_1 \oplus \mathfrak{a}_2$ is an abelian Lie algebra,
    \item $\mathfrak{a}_1$ and $\mathfrak{a}_2$ are ideals of dimension $2$,
    \item $\h$ is an abelian subalgebra.
    \item $\omega_\eta=e^{12}+e^{34}+\eta e^{56}$.
\end{itemize}
Set $\G=\G_{(1,0,0,1)}$. Let $\mu\in\G^\ast$ be a linear form such that \[\mu=\mu_1e^1+\mu_2e^2+\mu_3e^3+\mu_4e^4+\mu_5e^5+\mu_6e^6.\] Observe first that $\md\mu=0$ implies  $\mu_j=0$ for all $j=1,2,3,4$. Since we are in the case of generalized symplectic oxidation, from now on we assume that $\mu = \mu_5 e^5 + \mu_6 e^6~$, with $ \mu_5^2 + \mu_6^2 \neq 0$. Let $\mathrm{D}:\G\to\G$ be an endomorphism. Using condition $(ii)$, we have for all $i,j \in \{1,2\}$: \[[\mathrm{D}(\mathfrak{a}_i),\mathfrak{a}_j]+[\mathfrak{a}_i,\mathrm{D}(\mathfrak{a}_j)]=0.\]
This implies that $\mathrm{D}(\mathfrak{a}) \subset \mathfrak{a}$. Moreover, 
$\mathrm{D} \in \mathrm{Der}(\mathcal{D}^1(\G))$ is a derivation 
of the derived ideal $\mathcal{D}^1(\G)$. Also from  condition $(iv)$, we have $\mu\circ\mathrm{D}=t\mu$, then $\mu(\mathfrak{a})=0$ implies that $D(\mathfrak{a})\subset\mathfrak{a}$. Therefore, $\mathrm{D}$ is given by
\begin{small}
\begin{align*}
\mathrm{D}&=\left( \begin {array}{cccccc} d_{11}&d_{12}&d_{13}&d_{14}&d_{15}&d_{
16}\\ \noalign{\medskip}d_{21}&d_{22}&d_{23}&d_{24}&d_{25}&d_{26}
\\ \noalign{\medskip}d_{31}&d_{32}&d_{33}&d_{34}&d_{35}&d_{36}
\\ \noalign{\medskip}d_{41}&d_{42}&d_{43}&d_{44}&d_{45}&d_{46}
\\ \noalign{\medskip}0&0&0&0&d_{55}&d_{56}\\ \noalign{\medskip}0&0&0&0
&d_{65}&d_{66}\end {array} \right).
\end{align*}
\end{small}
On the other hand, Condition $(i)$ is equivalent to the following system
\begin{align*}
\left( d_{22}+d_{11} \right) \mu_{5}&=0,\\
\left( d_{22}+d_{11} \right) \mu_{6}&=0,\\
- \left( d_{23}-d_{41} \right) \mu_{5}+d_{13}+d_{42}&=0,\\
- \left( d_{23}-d_{41} \right) \mu_{6}-d_{24}-d_{31}&=0,\\
- \left( d_{24}+d_{31} \right) \mu_{5}+d_{14}-d_{32}&=0,\\
- \left( d_{24}+d_{31} \right) \mu_{6}+d_{23}-d_{41}&=0,\\
-\mu_{6}\,d_{25}+\mu_{5}\,d_{26}-d_{16}&=0,\\
 - \left( -d_{13}-d_{42} \right) \mu_{5}+d_{23}-d_{41}&=0,\\
 - \left( -d_{13}-d_{42} \right) \mu_{6}+d_{14}-d_{32}&=0,\\
 - \left( -d_{14}+d_{32} \right) \mu_{5}+d_{24}+d_{31}&=0,\\
 - \left( -d_{14}+d_{32} \right) \mu_{6}-d_{13}-d_{42}&=0,\\
 \mu_{6}\,d_{15}-\mu_{5}\,d_{16}-d_{26}&=0,\\
 \left( d_{44}+d_{33} \right) \mu_{5}&=0,\\
\left( d_{44}+d_{33} \right) \mu_{6}&=0,\\
-\mu_{6}\,d_{45}+\mu_{5}\,d_{46}+d_{35}&=0,\\
\mu_{6}\,d_{35}-\mu_{5}\,d_{36}+d_{45}&=0.
\end{align*}
Since $\mu \not\equiv 0$, this implies that $d_{11} = -d_{22}$ and $d_{33} = d_{44}$. The
previous system determines $12$ linear equations. It is straightforward to see that they have three solutions depending on the values of $\mu_5$ and $\mu_6$, which are given by:
\begin{small}
\begin{align*}
\mathrm{D}_1=\left(\begin {array}{cccccc} -d_{22}&d_{12}&-d_{42}&d_{32}&d_{15}&-
\mu_{6}\,d_{25}\\ \noalign{\medskip}d_{21}&d_{22}&d_{41}&-d_{31}&d_{25
}&\mu_{6}\,d_{15}\\ \noalign{\medskip}d_{31}&d_{32}&-d_{44}&d_{34}&0&d
_{36}\\ \noalign{\medskip}d_{41}&d_{42}&d_{43}&d_{44}&0&d_{46}
\\ \noalign{\medskip}0&0&0&0&d_{55}&d_{56}\\ \noalign{\medskip}0&0&0&0
&d_{65}&d_{66}\end {array} \right)
\end{align*}
\begin{align*}
\mathrm{D}_2=\left( \begin {array}{cccccc} -d_{22}&d_{12}&-d_{42}&d_{32}&d_{15}&0
\\ \noalign{\medskip}d_{21}&d_{22}&d_{41}&-d_{31}&d_{25}&0
\\ \noalign{\medskip}d_{31}&d_{32}&-d_{44}&d_{34}&-\mu_{5}\,d_{46}&{
\frac {d_{45}}{\mu_{5}}}\\ \noalign{\medskip}d_{41}&d_{42}&d_{43}&d_{
44}&d_{45}&d_{46}\\ \noalign{\medskip}0&0&0&0&d_{55}&d_{56}
\\ \noalign{\medskip}0&0&0&0&d_{65}&d_{66}\end {array} \right)
\end{align*}
\begin{align*}
\mathrm{D}_3:=\left( \begin {array}{cccccc} -d_{22}&d_{12}&-d_{42}&d_{32}&{\frac {-
d_{25}\,\mu_{5}\,\mu_{6}+d_{26}\,{\mu_{5}}^{2}+d_{26}}{\mu_{6}}}&-\mu_
{6}\,d_{25}+\mu_{5}\,d_{26}\\ \noalign{\medskip}d_{21}&d_{22}&d_{41}&-
d_{31}&d_{25}&d_{26}\\ \noalign{\medskip}d_{31}&d_{32}&-d_{44}&d_{34}&
\mu_{6}\,d_{45}-\mu_{5}\,d_{46}&{\frac {d_{45}\,{\mu_{6}}^{2}-d_{46}\,
\mu_{5}\,\mu_{6}+d_{45}}{\mu_{5}}}\\ \noalign{\medskip}d_{41}&d_{42}&d
_{43}&d_{44}&d_{45}&d_{46}\\ \noalign{\medskip}0&0&0&0&d_{55}&d_{56}
\\ \noalign{\medskip}0&0&0&0&d_{65}&d_{66}\end {array} \right)
\end{align*}
\end{small}
Now, for each $\mathrm{D}_j \in \mathrm{End}(\G_{(1,0,0,1)})$, we will use Condition~(iii) to simplify $\mathrm{D}_j$, $j = 1, 2, 3$. Consider the endomorphism $\mathrm{D}_1$. Applying Condition~(iii) and via direct computation, we obtain the following initial values:
\[
d_{31} = 0, \quad d_{32} = 0, \quad d_{41} = 0, \quad d_{65} = 0, \quad d_{42} = 0, \quad d_{22} = 0, \quad d_{55} = 0.
\]
Then, in this case, Condition~(iii) yields the following system:
\begin{align*}
d_{12} + d_{21}&=0, &-d_{12}\mu_6 - d_{56}&=0, &-d_{21}\mu_6 + d_{56}&=0, &d_{44}\mu_6 - d_{34} - d_{43}&=0, \\
d_{43} - d_{44}\mu_6 + d_{34}&=0,& -d_{34}\mu_6 - 2d_{44} - d_{66}&=0, &-d_{43}\mu_6 - 2d_{44} + d_{66}&=0.
\end{align*}
It is easy to see that they have
a unique solution, which is given by 
\begin{small}
\begin{align*}
\left( \begin {array}{cccccc} 0&-d_{21}&0&0&d_{15}&-\mu_{6}\,d_{25}
\\ \noalign{\medskip}d_{21}&0&0&0&d_{25}&\mu_{6}\,d_{15}
\\ \noalign{\medskip}0&0&0&-d_{43}&0&d_{36}\\ \noalign{\medskip}0&0&d_
{43}&0&0&d_{46}\\ \noalign{\medskip}0&0&0&0&0&d_{21}\,\mu_{6}
\\ \noalign{\medskip}0&0&0&0&0&d_{43}\,\mu_{6}\end {array} \right).
\end{align*}
\end{small}
This endomorphism is actually the one labeled by $\mathrm{D}_4$ in Lemma~$\ref{Endomor}$. The two remaining cases are treated in the same way.

\end{proof}
\subsection{Lie structure of generalized symplectic oxidations}
$\G_{(\mathrm{D}_1,\mu_1,\lambda_1,t_1)}:$
\begin{align*}
[e_1,e_6]&=d_{26}e_7,&[e_2,e_6]&=(d_{25}\mu_6 - d_{26}\mu_5)e_7,\\
[e_3,e_5]&=d_{45}e_7,&[e_4,e_5]&=(-d_{45}\mu_6 + d_{46}\mu_5)e_7,\\
[e_5,e_6]&=(d_{21}\eta\mu_5 + d_{43}\eta\mu_6)e_7,&[e_1,e_5]&=e_2+d_{25}e_7,\\
[e_2,e_5]&=-e_1-\tfrac{1}{\mu_6}\left(-d_{25}\mu_5\mu_6 + d_{26}\mu_5^2 + d_{26}\right)e_7,&[e_3,e_6]&=e_4+d_{46}e_7,\\
[e_4,e_6]&=-e_3-\tfrac{1}{\mu_5}\left(d_{45}\mu_6^2 - d_{46}\mu_5\mu_6 + d_{45}\right)e_7,
&[e_7,e_5]&=-\mu_5e_7,\\
[e_7,e_6]&=-\mu_6e_7,&[e_8,e_1]&=d_{21}e_2+(d_{21}d_{25} + d_{26}d_{43})e_7,\\
[e_8,e_2]&=-d_{21}e_1+\tfrac{1}{\mu_6}(d_{25}d_{43}\mu_6^2 + \mu_5(d_{21}d_{25} - d_{26}d_{43})\mu_6 - d_{21}d_{26}(\mu_5^2 + 1))e_7,&[e_8,e_3]&=d_{43}e_4+(d_{21}d_{45} + d_{43}d_{46})e_7,\\
[e_8,e_4]&=-d_{43}e_3+\tfrac{1}{\mu_5}\left(d_{21}d_{46}\mu_5^2 - \mu_6(d_{21}d_{45} - d_{43}d_{46})\mu_5 - d_{45}d_{43}(\mu_6^2 + 1)\right)e_7,\\
[e_8,e_5]&=\tfrac{1}{\mu_6}\left(-d_{25}\mu_5\mu_6 + d_{26}\mu_5^2 + d_{26}\right)e_1+d_{25}e_2+(d_{45}\mu_6 - d_{46}\mu_5)e_3+d_{45}e_4\\
&~~~~+d_{21}\mu_5e_5+d_{43}\mu_5e_6+\lambda_5e_7+\mu_5e_8,\\
[e_8,e_6]&=(-d_{25}\mu_6 + d_{26}\mu_5)e_1+d_{26}e_2+\tfrac{1}{\mu_5}(d_{45}\mu_6^2 - d_{46}\mu_5\mu_6 + d_{45})e_3+d_{46}e_4\\
&~~~~+d_{21}\mu_6e_5+d_{43}\mu_6e_6\\
&~~~~+\tfrac{1}{\mu_6\mu_5^2}\Big\lbrace((-d_{21}^2\eta\mu_6 - d_{26}^2)\mu_5^3 - 2\mu_6(\eta d_{21}d_{43}\mu_6 - d_{25}d_{26} -\tfrac{1}{2} d_{46}^2)\mu_5^2\\
&~~~~ +(-\eta d_{43}^2\mu_6^3 + (-d_{25}^2 - 2d_{45}d_{46} + \lambda_5)\mu_6^2 - d_{26}^2)\mu_5 + d_{45}^2\mu_6(\mu_6^2 + 1))\Big\rbrace e_7\\
&~~~~+\mu_6e_8,\\
[e_8,e_7]&=(-d_{21}\mu_5 - d_{43}\mu_6)e_7.
\end{align*}
$\G_{(\mathrm{D}_2,\mu_2,\lambda_2,t_2)}:$
\begin{align*}
[e_3,e_5]&=d_{45}e_7,&[e_4,e_5]&=d_{46}\mu_5e_7,&[e_5,e_6]&=d_{21}\eta\mu_5 e_7,\\
[e_1,e_5]&=e_2+\tfrac{1}{d_{21}}\lambda_1e_7,&[e_2,e_5]&=-e_1-d_{15}e_7,&[e_3,e_6]&=e_4+d_{46}e_7,\\
[e_4,e_6]&=-e_3-\tfrac{1}{\mu_5}d_{45}e_7,&[e_7,e_5]&=-\mu_5e_7,&[e_8,e_1]&=d_{21}e_2+\lambda_1e_7,\\
[e_8,e_1]&=-d_{21}e_1-d_{15}d_{21}e_7,&[e_8,e_3]&=d_{43}e_4+(d_{21}d_{45} + d_{43}d_{46})e_7,&[e_8,e_4]&=-d_{43}e_3+\tfrac{1}{\mu_5}(d_{21}d_{46}\mu_5^2 - d_{43}d_{45})e_7,
\end{align*}
\begin{align*}
[e_8,e_5]&=d_{15}e_1+\tfrac{1}{d_{21}}\lambda_1e_2-d_{46}\mu_5e_3+d_{45}e_4+d_{21}\mu_5e_5+d_{43}\mu_5e_6+\lambda_5e_7+\mu_5e_8,\\
[e_8,e_6]&=\tfrac{1}{\mu_5}d_{45}e_3+d_{46}e_4+\tfrac{1}{\mu_5^2}(-d_{21}^2\eta\mu_5^3 + d_{46}^2\mu_5^2 + d_{45}^2)e_7,\\
[e_8,e_7]&=-d_{21}\mu_5e_7.
\end{align*}

$\G_{(\mathrm{D}_3,\mu_3,\lambda_3,t_3)}:$
\begin{align*}
[e_3,e_5]&=d_{45}e_7,&[e_4,e_5]&=d_{46}\mu_5e_7,&[e_1,e_5]&=e_2d_{25}e_7,\\
[e_2,e_5]&=-e_1-d_{15}e_7,&[e_3,e_6]&=e_4+d_{46}e_7,&[e_4,e_6]&=-e_3-\tfrac{1}{\mu_5}d_{45} e_7,\\
[e_7,e_5]&=-\mu_5e_7,&[e_8,e_3]&=d_{43}e_4+d_{43}d_{46}e_7,&[e_8,e_4]&=-d_{43}e_3-\tfrac{1}{\mu_5}d_{45}d_{43}e_7,
\end{align*}
\begin{align*}
[e_8,e_5]&=d_{15}e_1+d_{25}e_2-d_{46}\mu e_3+d_{45}e_4+d_{43}\mu_5e_6+\lambda_5e_7+\mu_5e_8,\\
[e_8,e_6]&=\tfrac{1}{\mu_5}d_{45}e_3+d_{46}e_4+\tfrac{1}{\mu_5^2}(d_{46}^2\mu_5^2 + d_{45}^2)e_7.
\end{align*}

$\G_{(\mathrm{D}_4,\mu_4,\lambda_4,t_4)}:$
\begin{align*}
[e_1,e_6]&=d_{15}\mu_6e_7, &[e_2,e_6]&=d_{25}\mu_6e_7,&[e_5,e_6]&=d_{43}\eta\mu_6 e_7,\\
[e_1,e_5]&=e_2+d_{25}e_7,&[e_2,e_5]&=-e_1-d_{15}e_7,&[e_3,e_6]&=e_4+d_{46}e_7,\\
[e_4,e_6]&=-e_3-d_{36}e_7,&[e_7,e_6]&=-\mu_6e_7,&[e_8,e_1]&=d_{21}e_2+(d_{15}d_{43}\mu_6 + d_{21}d_{25})e_7,\\
[e_8,e_2]&=-d_{21}e_1+(d_{25}d_{43}\mu_6 - d_{15}d_{21})e_7,&[e_8,e_3]&=d_{43}e_4+d_{43}d_{46}e_7,&[e_8,e_4]&=-d_{43}e_3-d_{36}d_{43}e_7,\\
[e_8,e_5]&=d_{15}e_1+d_{25}e_2+(d_{43}^2\eta\mu_6 + d_{15}^2 + d_{25}^2)e_7,
\end{align*}
\begin{align*}
[e_8,e_6]&=-\mu_6d_{25}e_1+ \mu_6d_{15}e_2+ d_{36}e_3+d_{46}e_4+ d_{21}\mu_6e_5+d_{43}\mu_6e_6,+ \lambda_6e_7+ \mu_6e_8,\\
[e_8,e_7]&=-d_{43}\mu_6e_7.
\end{align*}

\begin{Le}\label{cohomonocentral}
The cohomology class $[a\omega_{\eta,\mathrm{D}_j}+\varrho\otimes\mu_j]\in H^2(\G_{(1,0,0,1)})$ vanishes. In particular, for all $j=1,\ldots,4$, there exists $\varrho_j\in\G_{(1,0,0,1)}^\ast$, $a\omega_{\eta,\mathrm{D}_j}+\varrho_j\otimes\mu_j=\mathbf{d}\varrho_j$. Moreover, every non-central extension of the Lie algebra $\G_{(1,0,0,1)}$ with respect to $\mu\in\G_{(1,0,0,1)}^\ast$ and $\varphi_j=\omega_{\eta,\mathrm{D}_j}$ is isomorphic to exactly one of the following Lie algebras$:$
\begin{align*}
\mathcal{N}_1:&~~[e_1, e_5] = e_2, ~~\quad 
		[e_2, e_5] = -e_1, \quad
		[e_3, e_6] = e_4, \quad 
		[e_4, e_6] = -e_3,\quad [e_7, e_5] = -\mu e_7\\
		&\\
		\mathcal{N}_2:&~~[e_1, e_5] = e_2, ~~\quad 
		[e_2, e_5] = -e_1, \quad
		[e_3, e_6] = e_4, \quad 
		[e_4, e_6] = -e_3,\quad [e_7, e_6] = -\mu e_7\\
		&\\
		\mathcal{N}_3:&~~[e_1, e_5] = e_2, ~~\quad 
		[e_2, e_5] = -e_1, \quad
		[e_3, e_6] = e_4, \quad 
		[e_4, e_6] = -e_3,\quad [e_7, e_5] = -\mu_1 e_7,\\
		&~~[e_7, e_6] = -\mu_2 e_7,
\end{align*}
where, $\mu,\mu_1,\mu_2\in\R^\ast$.
\end{Le}
\begin{proof}
The first statement follows directly.  For each $j = 1,\ldots,4$, we now present the following $:$
\begin{align*}
\omega_{\eta,\mathrm{D}_1}&=d_{25}e^{15}+d_{26}e^{16}-{\tfrac {1}{\mu_{6}}}\left(-d_{25}\,\mu_{5}\,\mu_{6}+d_{26}\,{\mu_{5}}^{2}+d_{26}\right)e^{25}+\left(_{25}\,\mu_{6}-d_{26}\,\mu_{5}\right)e^{26}\\
&+d_{45}e^{35}+d_{46}e^{36}+\left(-d_{45}\,\mu_{6}+d_{46}\,\mu_{5}\right)e^{45}-{\tfrac {1}{\mu_{5}}}\left(d_{45}\,{\mu_{6}}^{2}-d_{46}\,\mu_{5}\,\mu_{6}+d_{45}\right)e^{46}\\
&+\left(\eta\,d_{21}\,\mu_{5}+\eta\,d_{43}\,\mu_{6}\right)e^{56}
\end{align*}
Set $\varrho^1=\mathop{\resizebox{1\width}{!}{$\sum$}^6_{i=1}}\limits_{}\varrho_je^j$, with
\begin{align*}
\varrho_{1}&=-{\tfrac {1}{\mu_{6}}}ad_{26},&\varrho_{2}&=-{\tfrac {1 }{\mu_{6}}}a
 \left( d_{25}\,\mu_{6}-d_{26}\,\mu_{5} \right),&\varrho_{3}&=-
{\tfrac {1}{\mu_{5}}}ad_{45},\\
\varrho_{4}&={\tfrac {1 }{\mu_{5}}}a \left( d_{45}\,\mu_{6}-d
_{46}\,\mu_{5} \right),&\varrho_{5}&=-{\tfrac {1}{\mu_{6}}} \left(ad_{21}\,\eta\,
\mu_{5}+ad_{43}\,\eta\,\mu_{6}-\varrho_{6}\,\mu_{5}\right),&\varrho_6&\in\R
\end{align*}
where $d_{ij}$ are the coefficients of the endomorphism $\mathrm{D}_1$ given in Lemma~$\ref{Endomor}$, and $\mu = \mu_5 e^5 + \mu_6 e^6$ with $\mu_5\mu_6 \neq 0$, $a\in\R^\ast$   and $\eta\in\R^{\ast+}$.
\begin{align*}
\omega_{\eta,\mathrm{D}_2}&=\tfrac{1}{d_{21}}\lambda_1e^{15}-d_{15}e^{25}+d_{45}e^{35}+d_{46}e^{46}+d_{46}\mu_5e^{45}-\tfrac{1}{\mu_5}d_{45}e^{46}+d_{21}\eta\mu_5e^{56}
\end{align*}
\begin{align*}
\varrho_1 &= -\frac{ a(d_{15}d_{21} + \mu_5\lambda_1)}{d_{21}(\mu_5^2 + 1)}, &\varrho_2& =\frac{a(d_{15}d_{21}\mu_5 - \lambda_1)}{d_{21}(\mu_5^2 + 1)}, &\varrho_3& = -\frac{1}{\mu_5} ad_{45},\\
\varrho_4& = -ad_{46},&\varrho_5&\in\R, &\varrho_6& = ad_{21}\eta
\end{align*}
\begin{align*}
\omega_{\eta,\mathrm{D}_3}&=d_{25}e^{15}-d_{15}e^{25}+d_{45}e^{35}+d_{46}e^{46}+d_{46}\mu_5e^{45}-\tfrac{1}{\mu_5}d_{45}e^{46}
\end{align*}
\begin{align*}
\varrho_1 &= -\tfrac{a}{(\mu_5^2 + 1)} (d_{25}\mu_5 + d_{15}), &\varrho_2& =\tfrac{a}{(\mu_5^2 + 1)} (d_{15}\mu_5 - d_{25}), &\varrho_3& = -\tfrac{1}{\mu_5} ad_{45},\\
 \varrho_4&= -ad_{46},&\varrho_5&=0, &\varrho_6&\in\R
\end{align*}
\begin{align*}
\omega_{\eta,\mathrm{D}_4}&=d_{25}e^{15}+d_{15}\mu_6e^{16}-d_{15}e^{25}+d_{25}\mu_6e^{26}+d_{46}e^{36}-d_{36}e^{46}+d_{43}\eta\mu_6e^{56}
\end{align*}
\begin{align*}
\varrho_1&= -ad_{15}, &\varrho_2& = -ad_{25}, &\varrho_3& = -\tfrac{a}{(\mu_6^2 + 1)}(d_{46}\mu_6 + d_{36}),\\
\varrho_4& =\tfrac{a}{(\mu_6^2 + 1)}(d_{36}\mu_6 - d_{46}), &\varrho_5& = -ad_{43}\eta,&\varrho_6&\in\R
\end{align*}

\end{proof}
Note that the Lie brackets of non-central extension $\G_{\omega_{\eta,\mathrm{D}_j},\xi}$ are given as
\begin{align*}
[x,y]&=\overline{[x,y]}+\omega_{\eta,\mathrm{D}_j}(x,y)\xi,&\text{for all } x,y\in\G_{(1,0,0,1)},\\
[\xi,x]&=-\mu_j(x)\xi,&\text{for all } x\in\G_{(1,0,0,1)}.
\end{align*}

Therefore, for all $ j = 1, 2, 3, 4 $, we will map each non-central extension $ \G_{\omega_{\eta,\mathrm{D}_j},\mu_j} $ to one of the algebras given above via the isomorphism
\begin{align*}
\Psi(x)&=x+\varrho_j(x)\xi,&\text{for all } x\in\G_{(1,0,0,1)}\\
\Psi(\xi)&=\xi.
\end{align*}
Finally, note that the two non-central extensions $ \G_{\omega_{\eta,\mathrm{D}_1},\mu_1} $ and $ \G_{\omega_{\eta,\mathrm{D}_2},\mu_2} $ coincide.

\end{document}